\documentclass[11pt,letterpaper]{amsart}
\usepackage{amsmath}
\usepackage{amssymb}
\usepackage{amsfonts}
\usepackage{amsthm}
\usepackage{mathrsfs}
\usepackage{mathdots}
\usepackage{mathtools}
\usepackage{bbm}
\usepackage[pagebackref,colorlinks=true,linkcolor=blue,citecolor=blue]{hyperref}

\newcommand{\Z}{\mathbb{Z}}

\newcommand{\R}{\mathbb{R}}
\newcommand{\C}{\mathbb{C}}

\newcommand{\Fr}{{\mathrm{Fr}}}
\newcommand{\Hom}{{\mathrm{Hom}}}
\newcommand{\Per}{{\mathrm{Per}}}
\newcommand{\Evs}{{\mathrm{Evs}}}

\newcommand{\Ad}{{\mathrm{Ad}}}
\newcommand{\rank}{{\mathrm{rank}}}
\newcommand{\KS}{{\mathrm{KS}}}

\newcommand{\Lift}{{\mathrm{Lift}}}

\newcommand{\mic}{{\mathrm{mic}}}

\newcommand{\pure}{{\mathrm{pure}}}
\newcommand{\ABV}{{\mathrm{ABV}}}

\newcommand{\GL}{\mathrm{GL}}
\newcommand{\SO}{\mathrm{SO}}
\newcommand{\OO}{\mathrm{O}}
\newcommand{\SL}{\mathrm{SL}}
\newcommand{\Sp}{\mathrm{Sp}}
\newcommand{\U}{\mathrm{U}}
\newcommand{\BC}{\mathbb{C}}
\newcommand{\RG}{\mathrm{G}}

\newcommand{\IC}{\mathcal{IC}}

\newcommand{\comment}[1]{}

\newtheorem{thm}{Theorem}[section]
\newtheorem{cor}[thm]{Corollary}
\newtheorem{lemma}[thm]{Lemma}
\newtheorem{prop}[thm]{Proposition}
\newtheorem {conj}[thm]{Conjecture}

\newtheorem {ques/conj}[thm]{Question/Conjecture}

\newtheorem{defn}[thm]{Definition}
\newtheorem{rmk}[thm]{Remark}

\newtheorem{exmp}[thm]{Example}

\numberwithin{equation}{section}

\begin{document}

\title[Functoriality and the theta correspondence]{Functoriality and the theta correspondence}

\author[A. Hazeltine]{Alexander Hazeltine
}
\address{Department of Mathematics\\
Columbia University\\
New York, NY 10027, USA}
\email{alh2252@columbia.edu}


\dedicatory{}


\thanks{This research was partially supported by the AMS-Simons Travel Grant program.}

\begin{abstract}
    We study the functoriality of the local theta correspondence for classical $p$-adic groups. This is realized via the adaptation of the Adams conjecture to ABV-packets. We provide evidence for the conjecture, especially in the case of general linear groups.
\end{abstract}

\maketitle


\section{Introduction}\label{sec intro}

The theta correspondence was introduced by Howe (\cite{How79}) and has since proven to be a powerful tool within the Langlands program. However, despite its undeniable efficacy, it has historically been difficult to fit the theta correspondence into the theoretical framework of the Langlands program. Indeed, one early such attempt was by Langlands in a letter to Howe (\cite{Lan75}), where Langlands speculated if the theta correspondence was an instance of what is now called Langlands functoriality.

However, history has shown that the theta correspondence is not an instance of Langlands functoriality. Indeed, there are examples where the local theta correspondence does not preserve $L$-packets. Nevertheless, it remains desirable to pin down the theta correspondence within the framework of the Langlands program.  Adams proposed a version of functoriality for the theta correspondence which is now known as the Adams conjecture: the local theta correspondence should preserve local Arthur packets (\cite{Ada89}). Before stating the Adams conjecture, we introduce some notation. 

Let $F$ be a $p$-adic field and $W_F$ denote the Weil group. Let $\RG$ be a classical group which is quasi-split over $F$. We let $G=\RG(F)$ and denote the $L$-group by ${}^L G=\widehat{\RG}(\mathbb{C})\rtimes W_F$. Roughly, a local Arthur parameter is a homomorphism (see \S\ref{sec Local Arthur packets})
\[
\psi:W_F\times \SL_2(\C)\times\SL_2(\C)\rightarrow{}^LG.
\]
To a (bounded) local Arthur parameter $\psi$, Arthur and Mok attached a local Arthur packet denoted by $\Pi_\psi$ (\cite{Art13, Mok15}). This is a finite set of irreducible unitary representations of $G$. Furthermore, to $\psi$, one can associate an $L$-parameter $\phi_\psi$ such that the associated $L$-packet $\Pi_{\phi_\psi}$ is contained in $\Pi_\psi.$

Let $\mathrm{H}$ be a classical group which is quasi-split over $F$ and forms a reductive dual pair with $\RG.$ For an irreducible admissible representation $\pi$ of $G$, we let $\theta(\pi)$ denote the image of $\pi$ under the local theta correspondence (see \S\ref{sec Theta Correspondence}). Then $\theta(\pi)$ is an irreducible admissible representation (possibly vanishing) of $H=\mathrm{H}(F).$ We call $\theta(\pi)$ the theta lift of $\pi.$

With the above notation, we now state the Adams conjecture as follows (see Conjecture \ref{conj Adams naive} for a precise statement).

\begin{conj}[The Adams Conjecture (\cite{Ada89})]\label{conj Adams intro}
    Suppose that $\pi\in\Pi_\psi$ for some local Arthur parameter $\psi$ of $G.$ Then $\theta(\pi)\in\Pi_{\psi'}$, where $\psi'$ is an explicit local Arthur parameter of $H$ which only depends on $\psi.$
\end{conj}

It turned out that Adams was mostly correct. M{\oe}glin showed that the Adams conjecture is largely true (\cite[Theorem 6.1]{Moe11c}). Recently, for symplectic-even orthogonal dual pairs (although analogous results are expected more generally, see Conjecture \ref{conj Adams refined}), Baki{\'c} and Hanzer developed a way to determine precisely the validity of the Adams conjecture for representations in a fixed local Arthur packet (\cite{BH22}). In essence, these works completely determine when the Adams conjecture holds when the local Arthur packet $\Pi_\psi$ is fixed. 

However, M{\oe}glin exhibited examples where the Adams conjecture fails (\cite{Moe11c}). The failure can essentially be broken into two kinds.
\begin{enumerate}
    \item First, the theta lift of a representation could lie in a local Arthur packet, but not the one predicted by the Adams conjecture.
    \item Second, the theta lift of a representation may not lie in any local Arthur packet.
\end{enumerate}
The first failure has a hope to be resolved. Indeed, this failure was studied by the author in \cite{Haz24} for symplectic-even orthogonal dual pairs (again, the results are expected to hold more generally, see Conjecture \ref{conj Adams refined}). Specifically, from \cite{HLL22, HLL26} (see Conjecture \ref{conj closure max}), there exists a specific local Arthur parameter $\psi^{\max}(\pi)$ such that $\pi\in\Pi_{\psi^{\max}(\pi)}.$ For this local Arthur parameter, \cite[Theorem 1.5]{Haz24} states that if $\theta(\pi)\in\Pi_{\psi'}$, then $\theta(\pi)\in\Pi_{\psi^{\max}(\pi)'}$, i.e., the Adams conjecture will hold in its greatest generality for $\psi^{\max}(\pi).$ Furthermore, \cite[Conjecture 1.6]{Haz24} (see Conjecture \ref{conj stable Arthur}) essentially says that if the Adams conjecture fails for $\psi^{\max}(\pi),$ then it must be for the second reason: $\theta(\pi)$ does not lie in any local Arthur packet.

The second failure of the Adams conjecture is more critical. Indeed, the Adams conjecture only concerns local Arthur packets, while the second failure is when the theta lift does not lie in any local Arthur packet. This forces us to consider a generalization of local Arthur packets known as ABV-packets (see Theorem \ref{thm Vogan}). These packets were originally defined for real groups by Adams, Barbasch and Vogan (\cite{ABV92}). For connected reductive groups defined over a $p$-adic field, we follow a formulation given in \cite{CFMMX22}. We will only focus on the $p$-adic case in this article. ABV-packets are parameterized by $L$-parameters and consist of a certain finite set of irreducible admissible representations. To an $L$-parameter $\phi$, we let $\Pi_\phi^\ABV$ denote the corresponding ABV-packet. We conjecture that the Adams conjecture has an analogue for ABV-packets (see Conjecture \ref{conj Adams ABV naive} for a precise statement).

\begin{conj}[The Adams Conjecture for ABV-packets]\label{conj Adams ABV intro}
    Suppose that $\pi\in\Pi_\phi^\ABV$ for some $L$-parameter $\phi$ of $G.$ Then $\theta(\pi)\in\Pi_{\phi'}^\ABV$, where $\phi'$ is an explicit $L$-parameter of $H$ which only depends on $\phi.$
\end{conj}

Given a local Arthur parameter $\psi,$ we attach an $L$-parameter $\phi_\psi$ (see \eqref{eq phi_psi}).
When $\phi=\phi_\psi$ for some local Arthur parameter $\psi$, the main result of \cite{CHLLRXS26} (see Theorem \ref{thm Vogan}) shows that $\Pi_{\phi}^\ABV=\Pi_\psi$ (except in the case of metaplectic groups which we shall assume so as not to disrupt the discussion). Furthermore, we have that $\phi'=\phi_{\psi'}$. Consequently, Conjecture \ref{conj Adams ABV intro} provides a generalization of Conjecture \ref{conj Adams intro}. 

Of course, since Conjecture \ref{conj Adams intro} does fail, Conjecture \ref{conj Adams ABV intro} also does fail. However, every representation lies in some ABV-packet and so the only possible failure is of the first kind, i.e., $\theta(\pi)$ lies in some ABV-packet, but not $\Pi_{\phi'}^\ABV$. For symplectic-even-orthogonal dual pairs, the resolution for this failure was to consider a specific local Arthur parameter $\psi^{\max}(\pi).$ This has a natural analogue for ABV-packets, namely the $L$-parameter $\phi_\pi$ of $\pi.$ Our first piece of evidence for Conjecture \ref{conj Adams ABV intro} is to verify it for $\phi_\pi$ (Lemma \ref{lemma going down Adams ABV}).

\begin{lemma}\label{lemma going down Adams ABV intro}
    If $H$ is the ``going-down'' tower (see \S\ref{sec Theta Correspondence}) for $\pi,$ then $\theta(\pi)\in\Pi_{(\phi_\pi)'}^\ABV.$
\end{lemma}

This follows from the computation of the $L$-parameter of $\theta(\pi)$ (\cite{AG17a, BH21}) and the fact that the $L$-packet $\Pi_\phi$ is contained in $\Pi_{\phi}^\ABV$ (\cite[Proposition 7.13(b)]{CFMMX22}, see Proposition \ref{prop Lpacket in ABV}). 

Our second piece of evidence is that we establish the analogue of M{\oe}glin's result (\cite[Theorem 6.1]{Moe11c}) for ABV-packets of general linear groups (which are dual pairs of type II). That is, let $G=\GL_n(F)$ and $H=\GL_m(F)$. For an irreducible admissible representation $\pi$ of $\GL_n(F)$, we let $\theta(\pi)$ be the irreducible admissible representation of $\GL_m(F)$ which is the image of $\pi$ under the local theta correspondence (see Theorem \ref{thm theta GL}). We verify Conjecture \ref{conj Adams ABV intro} when $m\gg n$ (Theorem \ref{thm GL main thm}).

\begin{thm}\label{thm GL main thm intro}
    Suppose that $\pi\in\Pi_\phi^\ABV$ for some $L$-parameter $\phi$ of $\GL_n(F).$ If $m\gg n,$ then $\theta(\pi)\in\Pi_{\phi'}^\ABV.$ 
\end{thm}

We remark that ABV-packets for $\GL_n(F)$ are not necessarily singletons. Indeed, it was demonstrated in \cite{CFK22} that there is an ABV-packet of $\GL_{16}(F)$ of size 2. The existence of nonsingleton ABV-packets presents the primary obstacle in the proof of Theorem \ref{thm GL main thm intro}. As an application of Theorem \ref{thm GL main thm intro}, we obtain that there are ABV-packets of $\GL_n(F)$ of size at least 2 where $n=16, 18, 20$ or $n\geq 21$ (Corollary \ref{cor nonsingleton}).

In \cite[Theorem 25.8]{ABV92}, a geometric analogue of endoscopy (for real groups) is given through the use of a fixed point formula.
In \cite[Proposition 3.2]{CR26}, an analogue of this fixed point formula is established for local Arthur parameters of $\GL_n(F)$. However, for our setting, we require an extension of this result to certain $L$-parameters, not necessarily of Arthur type. To establish a fixed point formula, one needs to find a lifting of the regular conormal elements (we defer to \S\ref{sec Conormal bundles} for the terminology and precise meanings). This amounts to studying the intersections of the closures of conormal bundles which is generally a difficult problem. Under the assumption that $m\gg n$, we exploit explicit descriptions of the underlying geometry to prove that there is a lifting of regular conormal elements (Proposition \ref{prop regular conorm}). The proof of this result forms the technical heart of this paper. From this lifting of regular conormal elements, we are able to extract the fixed point formula (Theorem \ref{thm GL fixed point formula}) in our setting. 

The fixed point formula then shows that 
$\theta(\pi)\in\Pi_{\phi'}^\ABV$ if and only if  $\pi^\vee\in\Pi_{\phi^\vee}^\ABV,$ where $\pi^\vee$ and $\phi^\vee$ are the contragredient of $\pi$ and $\phi.$ In general, it is expected that ABV-packets are preserved by the contragredient. We verify this for $\GL_n(F)$ (Lemma \ref{lemma contragredient ABV}) from which 
Theorem \ref{thm GL main thm intro} follows.

\begin{lemma}\label{lemma contragredient ABV intro}
    Let $\pi$ be an irreducible admissible representation of $\GL_n(F)$ and $\phi$ be an $L$-parameter of $\GL_n(F).$
    We have $\pi\in\Pi_\phi^\ABV$ if and only if $\pi^\vee\in\Pi_{\phi^\vee}^\ABV.$
\end{lemma}

For a general reductive dual pair $(G,H)$, we expect a similar argument to yield an analogue of Theorem \ref{thm GL main thm intro}, but there are serious complications that need to be resolved. For example, we made use of the fact that $L$-packets of $\GL_n(F)$ are singletons. In general, this is not the case and so one needs to keep track of the enhanced $L$-parameters of the representations. This is done in \cite{AG17a, BH21}, but it needs to be translated into Vogan's perspective on the local Langlands correspondence (\cite{Vog93}; see also \S\ref{sec ABV-packets}). Another complication arises in establishing the analogue of Proposition \ref{prop regular conorm} as this is the most technically involved part of the proof. In the proof of Theorem \ref{thm GL main thm intro}, we also make use of a shortcut for general linear groups in order to write $\theta(\pi)$ as a certain irreducible parabolic induction (see Lemma \ref{lemma ind is irred for alpha large}). This result would also need to be generalized or circumvented.
Finally, the most serious problem occurs when $H$ is the ``going-up'' tower for $\pi.$  In this case, the $L$-parameter of $\theta(\pi)$ is not necessarily $\phi_\pi'$. This is a consequence of \cite[Theorem 4.5]{AG17a} (and \cite[Theorem 6.8]{BH22};  see also \cite[p. 558]{Sak17}).
All of these issues (and possibly more) would appear in adapting our approach to Theorem \ref{thm GL main thm intro} to general reductive dual pairs. However, we remark that we expect these issues to be technical hurdles rather than critical ones.

Finally, we make some remarks on why it is desirable to have the Adams conjecture for ABV-packets.
\begin{enumerate}
    \item Langlands originally conjectured that the local theta correspondence was an instance of Langlands functoriality (\cite{Lan75}). This conjecture encompasses all irreducible admissible representations, not just those of Arthur type. By passing from local Arthur packets to ABV-packets, the Adams conjecture now applies to any irreducible admissible representation and is hence closer to Langlands' original conjecture.
    \item As remarked earlier, the Adams conjecture for local Arthur packets can and does fail. The critical failure was when the theta lift was not of Arthur type. By considering ABV-packets, we are able to resolve this failure.
    \item The Adams conjecture for local Arthur packets fits into the framework of the relative Langlands program (\cite{BZSV24}). In this theory, the Adams conjecture is predicted to be the ``dual problem'' to the Gan-Gross-Prasad conjectures (\cite{GGP12, GGP20}, see \cite[Remark 7.12]{GW25}). As conjectural generalizations of local Arthur packets, it is natural to ask if ABV-packets may play a role in the relative Langlands program. Having the Adams conjecture for ABV-packets would be suggestive of a positive answer to this question. It would be very interesting if there were an analogue of the Gan-Gross-Prasad conjectures for ABV-packets. 
    \item It is an open problem to determine when the theta lift of a unitary representation is also unitary. In the stable range this is known to be true (\cite{Li89}), but remains open more generally. The failure of the Adams conjecture for local Arthur packets at a specific local Arthur parameter conjecturally determines a lower bound for this problem (see Remark \ref{rmk nonunitary}). Determining this bound currently remains mysterious; however, we suspect that the Adams conjecture for ABV-packets may play a role in this (see Example \ref{exmp curious}).
\end{enumerate}

Here is the outline of this article. In \S\ref{sec Setup}, we recall the local theta correspondence, local Arthur packets, and the Adams conjecture for local Arthur packets precisely (Conjectures \ref{conj Adams naive} and \ref{conj Adams refined}). In \S\ref{sec ABV-packets}, we recall the definition of ABV-packets and discuss the Adams conjecture for ABV-packets for type I dual pairs (Conjectures \ref{conj Adams ABV naive} and \ref{conj Adams ABV refined}). In \S\ref{sec The Adams conjecture for general linear groups}, we discuss the Adams conjecture for ABV-packets for type II dual pairs (Conjectures \ref{conj Adams ABV GL naive} and \ref{conj Adams ABV GL refined}). We provide more detail in this situation and carry out the above strategy to prove Theorem \ref{thm GL main thm intro}. Finally, in Appendix \ref{appenedix fpf}, we carry out the proof of the fixed point formula, Theorem \ref{thm GL fixed point formula}.

\subsection*{Acknowledgments}

The author thanks Jeffrey Adams, Baiying Liu, and Chi-Heng Lo for their comments and support. The author additionally thanks Clifton Cunningham and Mishty Ray for helpful discussions and comments, especially in relation to Theorem \ref{thm GL fixed point formula}. 

\subsection*{AI Disclosure Statement}
ChatGPT-6 was used to provide a list of potential errors in early drafts of this article. The author manually reviewed and corrected these errors. The author assumes full responsibility for the accuracy and integrity of this article.

\section{Background}\label{sec Setup}

Let $F$ be a non-Archimedean local field of characteristic $0$ and $q=q_F$ be the cardinality of the residual field. We set $|\cdot|$ to be the normalized $p$-adic absolute value on $F$. By abuse of notation, we also set $|\cdot|$ to the composition of the $p$-adic absolute value with the determinant. For a set $S$ acted upon by a group $H,$ we let $Z_{H}(S)$ denote the centralizer of $S$ in $H.$ When $S=\{s\}$ is a singleton, we simply write $Z_{H}(S)=Z_H(s).$

We let $E$ be a field such that $[E:F]\leq 2$ and $c\in\mathrm{Gal}(E / F)$ be a generator. We fix a nontrivial additive character $\psi$ of $F$ and let $\psi_E$ be the additive character of $E$ defined by $\psi_E=\psi\circ \mathrm{tr}_{E/F}.$

Let $\epsilon\in\{\pm1\}$,
$W_n$ be an $\epsilon$-Hermitian space of dimension $n$ over $E$, and $V_m$ be an $-\epsilon$-Hermitian space of dimension $m$ over $E.$ We let $\langle\cdot,\cdot\rangle_W:W\times W\rightarrow E$ and $\langle\cdot,\cdot\rangle_V:V\times V\rightarrow E$ denote the $\epsilon$-Hermitian and $-\epsilon$-Hermitian forms of $W$ and $V$, respectively. We set
\begin{align*}
    \epsilon_0=\left\{\begin{array}{cc}
        -\epsilon & \mathrm{if} \ \ E=F,  \\
        0 & \mathrm{otherwise.}
    \end{array}\right.
\end{align*}
The isometry groups of $W_n$ and $V_m$ are denoted by $G=G(W_n)$ and $H=H(V_m)$ respectively (except in the below case). For example, when $\epsilon=-1$, $E=F,$ and $n$ and $m$ are even, $G$ is a symplectic group and $H$ is an even orthogonal group. The exceptions are when $E=F,$ $\epsilon=-1,$ $n$ is even, and $m$ is odd, we set $G=\mathrm{Mp}(W_n)$ and $H=\OO(V_m)$. Similarly, when $E=F,$ $\epsilon=1$, $n$ is odd, and $m$ is even, we set $G=\OO(W_n)$ and $H=\mathrm{Mp}(V_m)$.

Let $\mathbb{H}$ denote a hyperbolic plane. Any $\epsilon$-Hermitian space $W_n$ has a Witt decomposition
\begin{equation}\label{eqn Witt decomp}
    W_n=W_{n_0}+W_{r,r},
\end{equation}
where $n=n_0+2r$, $W_{n_0}$ is anisotropic and $W_{r,r}\cong \mathbb{H}^r.$ The isomorphism class of $W_n$ uniquely determines the Witt index $r$ and the space $W_{n_0}$. Fix an anisotropic $\epsilon$-Hermitian space $W_{n_0}.$ Then we associate a Witt tower to $W_{n_0}$ as follows:
\begin{equation}\label{eqn Witt tower W}
    \mathcal{W}=\{W_{n_0}+W_{r,r} \ | \ r\geq 0\}.
\end{equation}
Similarly, we associate a Witt tower to an anisotropic $-\epsilon$-Hermitian space $V_{n_0}$ via
\begin{equation}\label{eqn Witt tower}
    \mathcal{V}=\{V_{n_0}+V_{r,r} \ | \ r\geq 0\}.
\end{equation}

For brevity, we often write $G'=G,H.$ We let $\Pi(G')$ be the set of equivalence classes of irreducible admissible representations of $G'.$

\subsection{Theta Correspondence}\label{sec Theta Correspondence}

Recall that we fixed an additive character $\psi$ on $F.$ The pair $(G,H)$ forms a reductive dual pair of a certain metaplectic group. We fix a pair of characters $\chi_W, \chi_V$ of $E^\times$ as in \cite[\S3.2]{GI14} and write $\chi=(\chi_W,\chi_V).$ This choice gives a splitting of the metaplectic group through which we may consider the Weil representation $\omega_{W_n, V_m, \psi}$ of $G\times H$ (see \cite[\S4.1]{GI14}). Let $\Pi(G)$ denote the set of equivalence classes of irreducible genuine smooth representations of $G$ (and similarly for $H$).
Given $\pi\in\Pi(G)$, 
we denote the maximal $\pi$-isotypic quotient of the Weil representation by
$$
\pi\boxtimes \Theta_{W_n,V_m,\chi,\psi}(\pi),
$$
where $\Theta_{W_n,V_m,\chi,\psi}(\pi)$ is a smooth representation of $H$ which is called the big theta lift of $\pi.$ We let $\theta_{W_n,V_m,\chi,\psi}(\pi)$, the (little) theta lift of $\pi,$ be the maximal semi-simple quotient of $\Theta_{W_n,V_m,\chi,\psi}(\pi)$. Originally conjectured by Howe (\cite{How79}), the following theorem was first proven by Waldspurger (\cite{Wal90}) when the residual characteristic of $F$ is not 2 and then in full generality by Gan and Takeda (\cite{GT16}) and  Gan and Sun (\cite{GS17}).

\begin{thm}[Howe Duality] Let $\pi_1,\pi_2\in\Pi(G).$
\begin{enumerate}
    \item If $\theta_{W_n,V_m,\chi,\psi}(\pi_2)\neq 0$, then $\theta_{W_n,V_m,\chi,\psi}(\pi_2)$ is irreducible.
    \item If $\pi_1\not\cong\pi_2$ and both $\theta_{W_n,V_m,\chi,\psi}(\pi_1)$ and $\theta_{W_n,V_m,\chi,\psi}(\pi_2)$ are nonzero, then \[\theta_{W_n,V_m,\chi,\psi}(\pi_1)\not\cong \theta_{W_n,V_m,\chi,\psi}(\pi_2).\]
\end{enumerate}
\end{thm}

For our purposes, we need to consider several towers of theta lifts at once. Recall that $V_m$ lies in some Witt tower (see \eqref{eqn Witt tower}). If $E=F$ and $\epsilon=1,$ then there is only one choice of anisotropic $V_{n_0}.$ Otherwise, there are always two towers of the form \eqref{eqn Witt tower}, say $\mathcal{V}'$ and $\mathcal{V}''.$ Fix a representation $\pi\in\Pi(G(W_n)).$ We define the first occurrence of $\pi$ in the tower $\mathcal{V}'$, denoted $m'(\pi)$, to be the minimal integer $m'=\dim V'$ such that $V'\in\mathcal{V}'$ and $\theta_{W_n,V',\chi,\psi}(\pi)\neq 0.$ We define $m''(\pi)$ analogously for the tower $\mathcal{V}''$. We define 
\begin{align*}
    m^+(\pi)&=\max\{m'(\pi),m''(\pi)\}, \\
    m^-(\pi)&=\min\{m'(\pi),m''(\pi)\}.
\end{align*}

When $E=F$ and $\epsilon=1,$ we have that $G(W_n)=\OO_{n}(F)$. In this case, we have representations $\pi$ and $\pi\otimes\det$ of $G(W_n).$ Thus for $V_m\in\mathcal{V}$, we have ``towers''
\[
\theta_{W_n,V_m,\chi,\psi}(\pi) \ \ \mathrm{and} \ \ \theta_{W_n,V_m,\chi,\psi}(\pi\otimes \det).
\]
We let $m(\pi)$ be the minimal integer $m=\dim V$ such that $V\in\mathcal{V}$ and $\theta_{W_n,V,\chi,\psi}(\pi)\neq 0.$ We define
\begin{align*}
        m^+(\pi)&=\max\{m(\pi),m(\pi\otimes\det)\}, \\
    m^-(\pi)&=\min\{m(\pi),m(\pi\otimes \det)\}.
\end{align*}

In general, when $\pi\in\Pi(G(W_n))$, the conservation relation gives a relation between $m^+(\pi)$ and $m^-(\pi).$ Note that $n=\dim(W_n).$

\begin{thm}[Conservation relation, \cite{SZ15}]\label{thm conservation relation}
    Let $\pi\in\Pi(G(W_n)).$ Then
    $$
    m^+(\pi)+m^-(\pi)=2n+2\epsilon_0+2.
    $$
\end{thm}

As a consequence, we have that $m^+(\pi)\geq n+\epsilon_0+1\geq m^-(\pi).$ Also, if one inequality is strict, then both inequalities are strict. In this situation, we call the tower whose first occurrence is $m^+(\pi)$ the ``going-up'' tower for $\pi$ and denote it by $\mathcal{V}^+.$ Similarly, we call the tower whose first occurrence is $m^-(\pi)$ the ``going-down'' tower for $\pi$ and denote it by $\mathcal{V}^-.$ When $m^+(\pi)=m^-(\pi),$ the designations of ``going-up'' or ``going-down'' will not matter (see Remarks \ref{rmk equal first occurences} and \ref{rmk equal first occurences ABV}).

Fix $\pi\in\Pi(G(W_n))$ and let $V_m^\pm\in\mathcal{V}^\pm,$ where $n=\dim W_n$ and $m=\dim V_m^\pm.$ We set $\alpha=M-N$, where $N$ is the dimension of the standard representation of the complex dual group of $G(W_n)$ and $M$ is the dimension of the standard representation of the complex dual group of $H(V_m^\pm)$. We note that $H(V_m^\pm)$ are pure inner forms of each other (by which we mean that $H(V_m^\pm)^\circ$ are pure inner forms) and hence have the same complex dual group. Given $\beta\in\Z_{\geq 0}$, we say $\beta$ is suitable if $\beta=M-N$ for some suitable $G(W_n)$ and $H(V_m^\pm)$. For example, if $\epsilon=-1$, $E=F,$ and $n$ and $m$ are even, then $G(W_n)=\Sp(W_n)$ and $H(V_m)=\OO(V_m).$ The complex dual groups are $\SO_{n+1}(\mathbb{C})$ and $\OO_m(\mathbb{C})$ and so $\alpha=m-n-1.$ In this case, $\beta\in\mathbb{Z}_{\geq 0}$ is suitable if and only if $\beta$ is odd. In general, we write $\theta_{W_n,V_m^\pm,\chi,\psi}(\pi)=\theta^\pm_{-\alpha}(\pi).$ When it is clear in context, we suppress the notation $\pm.$ Furthermore, we let $m^{\pm,\alpha}(\pi)$ denote the value of $\alpha$ corresponding to $m^\pm(\pi).$ For example, if $\epsilon=-1$, $E=F,$ and $n$ and $m$ are even, then
$m^{\pm,\alpha}(\pi)=m^\pm(\pi)-n-1.$ 

We give a table for the dual groups and $N$ below ($M$ is determined similarly). 
\begin{center}
    \begin{tabular}{|c|c|c|c|}
\hline
     & $G$ & $\widehat{G}(\mathbb{C})$ & $N$  \\
     \hline
    $\substack{E=F, \ n \ \mathrm{odd} \ \epsilon=1}$ & $\OO(W_n)$ & $\Sp_{n-1}(\mathbb{C})\times\{\pm1\}$ & $n-1$  \\
    \hline
    $\substack{E=F, \ n, m \ \mathrm{even}, \ \epsilon=1}$ & $\OO(W_n)$ & $\OO_{n}(\mathbb{C})$ & $n$  \\
    \hline
    $\substack{E=F, \ n, m \ \mathrm{even}, \ \epsilon=-1}$ & $\Sp(W_n)$ & $\SO_{n+1}(\mathbb{C})$ & $n+1$  \\
    \hline
    $\substack{E\neq F, \ n \ \mathrm{arbitrary}, \ \epsilon=\pm1}$ & $\U(W_n)$ & $\GL_{n}(\mathbb{C})$ & $n$  \\
    \hline
\end{tabular}
\end{center}
We note that these are the dual groups in the disconnected cases as given in \cite[Definition 5.1]{CHLLRXS26}.
Except for the following case, we set the $L$-group to be the $L$-group of the connected component, ${}^LG={}^LG^0.$
We remark that when $E=F$, $m$ is odd, and $\epsilon=-1,$ we set $G=\mathrm{Mp}(W_n).$ In this case, $n$ must be even and we set $\widehat{G}(\BC)=\Sp_n(\BC)={}^LG$ and $N=n.$

\subsection{Local Arthur packets}\label{sec Local Arthur packets} In this subsection, we let $G=G(W_n)$ and discuss local Arthur parameters and local Arthur packets for $G$. We note that the analogous notions also make sense for $H=H(V_m).$
  A local Arthur parameter may be considered as a homomorphism
$$\psi: W_F \times \SL_2(\mathbb{C}) \times \SL_2(\mathbb{C}) \rightarrow {}^LG,$$
which when composed with restriction to $W_E\times \SL_2(\mathbb{C}) \times \SL_2(\mathbb{C})$ and the standard representation gives a direct sum of irreducible representations
\begin{equation}\label{eq decomp psi +}
  \psi = \bigoplus_{i=1}^r \phi_i|\cdot|^{x_i} \otimes S_{a_i} \otimes S_{b_i},  
\end{equation}
satisfying the following conditions:
\begin{enumerate}
    \item [(1)] $\dim(\phi_i)=d_i$ and $\phi_i(W_E)$ is bounded and consists of semi-simple elements;
    \item [(2)] $x_i \in \R$ and $|x_i|<\frac{1}{2}$;
    \item [(3)]the restrictions of $\psi$ to the two copies of $\SL_2(\mathbb{C})$ are analytic, $S_k$ is the $k$-dimensional irreducible representation of $\SL_2(\mathbb{C})$, and 
    $$\sum_{i=1}^r d_ia_ib_i = N.
$$ 
\end{enumerate}

We consider local Arthur parameters up to $\widehat{G}(\BC)$-conjugacy, i.e., we say two local Arthur parameters are equivalent if they are conjugate under $\widehat{G}(\BC)$. We will not distinguish between $\psi$ and its equivalence class.
We define $\Psi^{+}(G)$ to be the set of equivalence classes of local Arthur parameters. Let $\Psi(G)$ be the subset of $\Psi^+(G)$ consisting of local Arthur parameters $\psi$ whose restriction to $W_F$ is bounded, i.e., if we decompose $\psi$ as in the decomposition \eqref{eq decomp psi +}, then $\psi\in\Psi(G)$ if and only if $x_i=0$ for $i=1,\dots, r$. In the literature, sometimes the set $\Psi(G)$ is considered in place of $\Psi^+(G).$ This is sufficient for global applications if one assumes the Ramanujan conjecture. We do not adopt this viewpoint and hence consider $\Psi^+(G)$.

Given $\psi\in\Psi^+(G),$ Arthur's conjectures (\cite{Art89}) predict that there should exist a finite set $\Pi_\psi$ consisting of equivalence classes of irreducible smooth representations which satisfy certain twisted endoscopic character identities. The set $\Pi_\psi$ is called the local Arthur packet attached to $\psi.$ Given a representation $\pi$ of $G$, we say that $\pi$ is of Arthur type if $\pi\in\Pi_\psi$ for some $\psi\in\Psi^+(G).$

For the purposes of this article, we shall assume that local Arthur packets exist. We remark briefly on the current status of this assumption. For quasi-split classical groups, the existence of local Arthur packets is essentially proven by the works of Arthur and Mok (\cite{Art13, Mok15}) when supplemented with the work of \cite{AGIKMS24}. The only remaining step in this case is the verification of the twisted weighted fundamental lemma. Based on the quasi-split case, the extension to pure inner forms follows from \cite{Ish23, KMSW14, Moe06, Moe09, Moe11a, MR18}. For metaplectic groups, see \cite{Li24}.

Before we proceed to state the Adams conjecture for local Arthur packets, we recall a conjecture which will play a role in the refinement of the Adams conjecture. Let $\pi$ be an irreducible admissible representation of $G.$ We let
\[
\Psi(\pi)=\{\psi\in\Psi^+(G) \ | \ \pi\in\Pi_\psi\}.
\]
We will consider a partial ordering $\geq_C$ on $\Psi(\pi)$. Its existence is enough for us in this section and so we defer some of the following unexplained terminology to \S\ref{sec ABV-packets}.
We have an injection from $\Psi^+(G)$ to the set of $L$-parameters of $G$  given by $\psi\mapsto\phi_\psi$ where
\begin{align}\label{eq phi_psi}
\phi_\psi(w,x)=\psi(w,x,\begin{pmatrix}
    |w|^{\frac{1}{2}} & 0 \\
    0 & |w|^{\frac{-1}{2}}
\end{pmatrix}).
\end{align}
The representation $\pi$ determines an infinitesimal parameter $\lambda_{\phi_\pi}$, where $\phi_\pi$ is the $L$-parameter of $\pi$. On the set of $L$-parameters with fixed infinitesimal parameter $\lambda_{\phi_\pi}$ there exists a partial ordering $\geq_C$ which is known as the closure ordering (Definition \ref{def C ordering}). Given $\psi_1,\psi_2\in\Psi^+(G),$ we define $\psi_1\geq_C \psi_2$ if $\lambda_{\phi_{\psi_1}}=\lambda_{\phi_{\psi_2}}$ and $\phi_{\psi_1}\geq_C \phi_{\psi_2}.$ It is conjectured that the closure ordering gives a partial order on $\Psi(\pi).$

\begin{conj}[{\cite[Conjecture 1.4]{HLLZ25}}]\label{conj closure max}
    Let $\pi$ be an irreducible admissible representation of $G$ of Arthur type. Then for any $\psi_1,\psi_2\in\Psi(\pi),$ we have that $\lambda_{\phi_{\psi_1}}=\lambda_{\phi_{\psi_2}}$. Furthermore, there exist elements $\psi^{\max}(\pi)$ and $\psi^{\min}(\pi)$ in $\Psi(\pi)$ such that for any $\psi\in\Psi(\pi),$ we have
    \[
    \psi^{\max}(\pi)\geq_C\psi\geq_C\psi^{\min}(\pi).
    \]
\end{conj}
When $G$ is a quasi-split symplectic or orthogonal group, this conjecture has been verified in \cite{HLL26, HLLZ25}. We assume the above conjecture (which serves as our definition of $\psi^{\max}(\pi)$); however, it is only needed in a few places, e.g., the statement of Conjecture \ref{conj Adams refined}(4).

\subsection{The Adams Conjecture for local Arthur packets}\label{sec The Adams Conjecture for local Arthur packets} 

In this subsection, we explicate Conjecture \ref{conj Adams intro} and also conjecture a refinement. Recall that we are considering the classical groups $G=G(W_n)$ and $H=H(V_{m}^\pm)$ along with the theta lift $\theta_{-\alpha}^\pm$ between them. Let $\psi\in\Psi^+(G).$ We define
\[
\psi_\alpha=(\chi_W\chi_V^{-1}\otimes\psi)\oplus\chi_W\otimes S_1\otimes S_\alpha,
\]
where we view the characters $\chi_V,\chi_W$ as representations of $W_E$ via local class field theory.
We have that $\psi_\alpha\in\Psi^+(H)$. Note that our definition of $\psi_\alpha$ only makes sense if $\alpha>0.$ We shall assume that $\alpha>0$ throughout the rest of this article as it will be implicit in the statements.

The Adams conjecture predicts that the theta lift sends $\Pi_\psi$ to $\Pi_{\psi_\alpha}.$

\begin{conj}[The (naive) Adams Conjecture {\cite{Ada89}}]\label{conj Adams naive}
Suppose that $\pi\in\Pi_\psi.$ Then $\theta_{-\alpha}^\pm(\pi)\in\Pi_{\psi_\alpha}$ provided that $\theta_{-\alpha}^\pm(\pi)\neq 0.$
\end{conj}

As mentioned in \S\ref{sec intro}, Conjecture \ref{conj Adams naive} can and does fail. However, the works of \cite{BH22, Haz24, Moe11c} give a precise description of when the Adams conjecture holds for symplectic-even orthogonal dual pairs.

\begin{thm}\label{thm Adams refined}
    Suppose that the pair $(G,H)$ is a quasi-split symplectic-even orthogonal dual pair and that $\pi\in\Pi_\psi$ for some $\psi\in\Psi^+(G).$
    \begin{enumerate}
        \item For $\alpha\gg 0,$ we have $\theta_{-\alpha}^\pm(\pi)\in\Pi_{\psi_\alpha}$ $($\cite[Theorem 6.1]{Moe11c}$)$.
        \item If $\theta_{-\alpha}^+(\pi)\neq 0,$ then $\theta_{-\alpha}^\pm(\pi)\in\Pi_{\psi_\alpha}$ $($\cite[Theorem 2]{BH22}$)$.
        \item If $\theta_{-\alpha}^\pm(\pi)\in\Pi_{\psi_\alpha}$ for some $\alpha,$ then $\theta_{-(\alpha+2)}^\pm(\pi)\in\Pi_{\psi_{\alpha+2}}$ $($\cite[Theorem C]{BH22}$)$.
        \item If $\theta_{-\alpha}^\pm(\pi)\in\Pi_{\psi_\alpha}$, then $\theta_{-\alpha}^\pm(\pi)\in\Pi_{\psi^{\max}(\pi)_\alpha}$ $($\cite[Theorem 1.5]{Haz24}$)$.
    \end{enumerate}
\end{thm}

We conjecture that the above theorem should also hold in the general setting. We call this the refined Adams conjecture.

\begin{conj}\label{conj Adams refined}
    Consider a dual pair $(G,H)$ as above and let $\pi\in\Pi_\psi$ for some $\psi\in\Psi^+(G).$
    \begin{enumerate}
        \item For $\alpha\gg 0,$ we have $\theta_{-\alpha}^\pm(\pi)\in\Pi_{\psi_\alpha}$.
        \item If $\theta_{-\alpha}^+(\pi)\neq 0,$ then $\theta_{-\alpha}^\pm(\pi)\in\Pi_{\psi_\alpha}$.
        \item If $\theta_{-\alpha}^\pm(\pi)\in\Pi_{\psi_\alpha}$ for some $\alpha,$ then $\theta_{-(\alpha+2)}^\pm(\pi)\in\Pi_{\psi_{\alpha+2}}$.
        \item If $\theta_{-\alpha}^\pm(\pi)\in\Pi_{\psi_\alpha}$, then $\theta_{-\alpha}^\pm(\pi)\in\Pi_{(\psi^{\max}(\pi))_\alpha}$.
    \end{enumerate}
\end{conj}

\begin{rmk}\label{rmk equal first occurences}
    When $m^+(\pi)=m^-(\pi)$, it follows from the conservation relation that $m^+(\pi)=n+\epsilon_0+1.$ Consequently, $m^{+,\alpha}(\pi)=1$ and so Conjecture \ref{conj Adams refined}(2) implies that the Adams conjecture is true for any suitable positive integer $\alpha$ and any $\psi\in\Psi(\pi).$ That is, the choice of ``going-up'' or ``going-down'' tower does not matter.
\end{rmk}

We remark that Conjecture \ref{conj Adams refined}(4) implicitly assumes Conjecture \ref{conj closure max}.
We also remark that for quasi-split symplectic-even orthogonal dual pairs, \cite[Theorem 1.3]{Haz24} proves a stronger statement which implies \cite[Theorem 1.5]{Haz24}. In general, we expect an analogue of \cite[Theorem 1.3]{Haz24} for any reductive dual pair of type I (which could be included in Conjecture \ref{conj Adams refined} above); however, we opted not to include it in the above list as it does not seem to generalize to $\ABV$-packets. We discuss this issue later (see Conjecture \ref{conj Adams closure}).

Conjecture \ref{conj Adams refined} predicts how the failure of the Adams conjecture may occur. Indeed,
Conjecture \ref{conj Adams refined}(1) states the Adams conjecture is always true when the difference in the ranks is large enough.  
Conjecture \ref{conj Adams refined}(2) states that the Adams conjecture is always true when considering the ``going-up'' tower for $\pi.$  Conjecture \ref{conj Adams refined}(3) states that if the Adams conjecture holds at some level $\alpha,$ then it holds at any greater level which implies that if the Adams conjecture fails at some level, then it fails at every lower level. Together these conjectures determine the Adams conjecture's validity when the local Arthur parameter $\psi$ is fixed. Conjecture \ref{conj Adams refined}(4) controls the validity of the Adams conjecture when $\psi\in\Psi(\pi)$ varies. In particular, Conjecture \ref{conj Adams refined}(4) says that, among all $\psi\in\Psi(\pi),$ the Adams conjecture holds in its greatest generality for $\psi^{\max}(\pi)$. 

So if the Adams conjecture fails for some fixed $\psi$, it must happen on the going-down tower and for small $\alpha$. The failure could be for one of two reasons, either the Adams conjecture failed for $\psi$ but holds for some $\psi'\in\Psi(\pi)$ (with $\psi'\geq_C\psi$) or the Adams conjecture  fails for $\psi^{\max}(\pi)$ (and hence any $\psi'\in\Psi(\pi)$). The following conjecture predicts if our failure is in the latter case, then it cannot be fixed using local Arthur packets.

\begin{conj}[{\cite[Conjecture 1.6]{Haz24}}]\label{conj stable Arthur}
Consider a dual pair $(G,H)$ as above and let $\pi$ be a representation of $G$ of Arthur type.
    Let $\alpha_0$ be the minimum among all positive suitable integers $\alpha$ such that $\theta_{-\alpha}^-(\pi)\in\Pi_{(\psi^{\max}(\pi))_\alpha}$. If $\alpha_0\geq 3,$ then $\theta_{-(\alpha_0-2)}^-(\pi)$ is not of Arthur type.
\end{conj}

Therefore, in order to remedy the failure of the Adams conjecture for $\psi^{\max}(\pi)$, we are forced to consider a conjectural generalization of local Arthur packets known as ABV-packets.

\begin{rmk}\label{rmk nonunitary}
In this remark we suppose that $\pi$ is of good parity. Generally, this is not an issue as the Adams conjecture always reduces to this case (e.g., \cite[Lemma 2.33]{Haz24}). Now we note that the Adams conjecture for $\pi$ is expected to hold for some $\psi\in\Psi(\pi)$ if the theta lift is unitary. This expectation is conjecturally equivalent to Conjecture \ref{conj stable Arthur}. Indeed, first note that for any suitable $\alpha\geq\alpha_0,$ we have $\theta_{-\alpha}^-(\pi)\in\Pi_{(\psi^{\max}(\pi))_\alpha}$ and hence $\theta_{-\alpha}^-(\pi)$ is unitary. It follows from \cite{AG17a, BH21} that $\pi$ is of good parity if and only if $\theta_{-\alpha}^-(\pi)$ is of good parity. It is conjectured in \cite[Conjecture 1.2]{HJLLZ24} that an irreducible admissible representation of good parity is of Arthur type if and only if it is unitary. Consequently, Conjecture \ref{conj stable Arthur} and \cite[Conjecture 1.2]{HJLLZ24} imply that if $\alpha_0\geq 3,$ then $\theta_{-(\alpha_0-2)}^-(\pi)$ is not unitary. We also note that \cite[Conjecture 1.2]{HJLLZ24} is known for split symplectic and odd special orthogonal groups by the work of Atobe and M{\'i}nguez (\cite[Theorem 1.1]{AM25}) and quasi-split even orthogonal groups by \cite[Theorem 11.1]{HLL26}.
\end{rmk}

\section{ABV-packets}\label{sec ABV-packets}

In this section, we recall the construction of $p$-adic ABV-packets following \cite{CFMMX22}. The nomenclature ``ABV'' stands for Adams-Barbasch-Vogan and is an hommage to \cite{ABV92} where ABV-packets were defined for real groups. We remark that \cite{CFMMX22} only treats connected reductive groups; however, for the local theta correspondence we must also consider metaplectic and orthogonal groups. Orthogonal groups (or more generally certain kinds of possibly disconnected reductive groups) are treated in \cite[Part III]{CHLLRXS26}. For even orthogonal groups, a discussion of the geometry is given in \cite[\S2.1]{HL26}. Analogous definitions and results are also expected to hold for metaplectic groups and we shall assume them here.

We continue to restrict ourselves to the setting that $G=G(W_n)$, $H=H(V_m)$, and we let $G'\in\{G,H\}.$ We also allow for $G'=\GL_n(F)$. An $L$-parameter of $G'$ may be regarded as a $\widehat{G}'(\C)$-conjugacy class of an admissible homomorphism $\phi:W_F\times\SL_2(\C)\rightarrow{}^LG'$ (\cite[\S8]{Bor79}). We do not require that $\phi$ is relevant for $G'.$ Let $\Phi(G')$ denote the set of $L$-parameters of $G'.$ We do not distinguish a representative $\phi$ from its conjugacy class. We assume that there is a local Langlands correspondence for $G'.$ Specifically, the local Langlands correspondence defines a map $rec:\Pi(G')\rightarrow\Phi(G').$ The $L$-packet attached to $\phi$ is $\Pi_\phi:=rec^{-1}(\phi).$ For $\pi\in\Pi(G')$, we let $\phi_\pi:=rec(\pi)$ denote the $L$-parameter of $\pi.$ The local Langlands correspondence for $G'$ is understood (subject to the twisted weighted fundamental lemma) through the works (\cite{Art13, AGIKMS24, CZ21a, CZ21b, GS12, HT01, Hen00, Ish23, KMSW14, MR18, Mok15,Sch13}).

An infinitesimal character of $G'$ is a continuous homomorphism $\lambda:W_F\rightarrow{}^LG'$ which is a section of ${}^LG'\rightarrow W_F$ and whose image consists only of semi-simple elements. Given $\phi\in\Phi(G')$, we denote the infinitesimal character associated to $\phi$ by 
\[
\lambda_\phi(w):=\phi\left(w,\left(\begin{matrix}
    |w|^{\frac{1}{2}} & 0 \\
    0 & |w|^{\frac{-1}{2}}
\end{matrix}\right)\right),
\]
where $|\cdot|$ denotes the norm on $W_F$ which is trivial on the inertia subgroup and sends the Frobenius element to $q.$ 

Let $\lambda$ be an infinitesimal character of $G'.$ We let 
\[
\Phi_\lambda(G')=\{\phi\in\Phi(G') \ | \ \lambda_\phi=\lambda\}.
\]
Similarly, we let 
\[
\Pi_\lambda(G')=\{\pi\in\Pi(G') \ | \ \lambda_{\phi_\pi}=\lambda\}.
\]
Both of these sets are finite.

Let $\mathfrak{g}'(\C)$ be the Lie algebra of $\widehat{G}'(\C)$.
We let
\begin{align*}
H_\lambda&=\{g\in\widehat{G}'(\C) \ | \ \lambda(w)g\lambda(w)^{-1}=g, \forall w\in W_F\}, \\
K_\lambda&=\{g\in\widehat{G}'(\C) \ | \ \lambda(w)g\lambda(w)^{-1}=g, \forall w\in I_F\},
\end{align*}
where $I_F$ is the inertia subgroup of the absolute Galois group of $F.$
Note that $H_\lambda=Z_{\widehat{G}'(\C)}(\lambda)$ is the centralizer of the image of $\lambda.$
We consider the Vogan variety 
\[
V_\lambda=\{x\in\mathrm{Lie}(K_\lambda) \ | \ \Ad(\lambda(w))x=|w|x, \forall w\in W_F\}.
\]
The group $H_\lambda$ acts on $V_\lambda$ via conjugation. The action stratifies $V_\lambda$ into finitely many orbits and we let $C_\lambda(G')$ denote the collection of these orbits. These orbits are in bijection with $\Phi_\lambda(G')$ (\cite[Proposition 4.2.2]{CFMMX22}). The bijection is given by identifying $\phi\in\Phi_\lambda(G')$ with the $H_\lambda$-orbit of $x_\phi$ where \[
x_\phi=d(\phi|_{\SL_2(\C)})\left(\begin{matrix}
    0 & 1 \\
    0 & 0
\end{matrix}\right).
\]
We let $C_\phi$ be the $H_\lambda$-orbit of $x_\phi$. The geometry of $V_{\lambda}$ induces a partial ordering $\geq_C$ on $\Phi(G)_{\lambda}$. We call this partial ordering, the closure ordering.
\begin{defn}\label{def C ordering}
Let $\phi_1,\phi_2\in\Phi_{\lambda}(G').$ We define a partial ordering $\geq_C$ on $\Phi_{\lambda}(G')$ by $\phi_{1} \geq_C \phi_2$ if $\overline{C_{\phi_1}} \supseteq C_{\phi_2}$.
\end{defn}

As in \cite[Definition 3.5]{CHLLRXS26}  (see also \cite[Lemma 3.8]{CHLLRXS26}), we consider pure inner forms of $G'$ to be those corresponding to the image of $H^1(F,(G')^\circ)\rightarrow H^1(F,G').$
Now, $G'$ always has a quasi-split pure inner form which we denote by $G'_0.$ When $G'$ is a symplectic or general linear group, it is split and so $G'=G'_0.$ Otherwise, $G'$ always has exactly one nontrivial pure inner form corresponding to the ``other'' tower (see \S\ref{sec Theta Correspondence}). It is possible that both pure inner forms are quasi-split in which case we fix a choice of $G'_0.$ In our situation, the pure inner forms of $G'$ are indexed by $H^1(F,(G'_0)^\circ)$ and $G_0'$ corresponds to the trivial class. We let $\Pi^\pure(G')$ denote the set of equivalence classes of representations of $G'$ and its pure inner forms. Similarly, given an $L$-parameter $\phi$ of $G'$, we let $\Pi_\phi^\pure$ denote the union over the pure inner forms of the corresponding $L$-packets.

Let $\Per_{H_\lambda}(V_\lambda)$ denote the category of $H_\lambda$-equivariant perverse sheaves on $V_\lambda.$ Vogan's perspective on the local Langlands correspondence (\cite{Vog93}) gives a bijection between $\Pi^{\pure}_\lambda(G')$ and the simple objects in $\Per_{H_\lambda}(V_\lambda)$ (up to isomorphism). For $\pi\in\Pi_\lambda(G')$, we write $\mathcal{P}(\pi)$ for the corresponding simple perverse sheaf.

Let $\phi\in\Phi_\lambda(G').$ Cunningham et al. defined the ABV-packet attached to $\phi$ (\cite[\S8.1]{CFMMX22}). It is denoted by  $\Pi_\phi^{\ABV}$ and in our setting, we have
\[
\Pi_\phi^{\ABV}:=\{\pi\in\Pi^\pure_\lambda(G') \ | \ \Evs_{C_\phi}(\mathcal{P}(\pi))\neq 0\}.
\]
Here $\Evs$ is the microlocal vanishing cycles functor defined in \cite[\S7.9]{CFMMX22}. This functor is essential to compute ABV-packets; however, for our purposes, it suffices to give several properties of ABV-packets instead. We set \[\Pi_\phi^{\ABV}(G')=\Pi_\phi^{\ABV}\cap\Pi(G').\]

First, we have that ABV-packets respect the closure ordering.

\begin{prop}[{\cite[Proposition 7.10]{CFMMX22}, \cite[Theorem 11.6]{CHLLRXS26}}]\label{prop ABV closure}
    If $\pi\in\Pi_\phi^{\ABV},$ then  $\phi_\pi\geq_C\phi.$
\end{prop}

Second, we have that ABV-packets contain their $L$-packets.

\begin{prop}[{\cite[Proposition 7.13(b)]{CFMMX22}, \cite[Theorem 2.3(3)]{CHLLRXS26}}]\label{prop Lpacket in ABV}
    We have that $\Pi_\phi^\pure\subseteq\Pi_\phi^{\ABV}.$
\end{prop}

The following proposition follows directly from Propositions \ref{prop ABV closure} and \ref{prop Lpacket in ABV} (see also \cite[\S10.2.6]{CFMMX22}).

\begin{prop}\label{prop open orbit ABV=Lpacket}
    If $C_\phi$ is the unique open orbit in $C_\lambda(G')$, then $\Pi_{\phi}^{\ABV}=\Pi_\phi^\pure.$
\end{prop}

Finally, we mention that ABV-packets are expected to generalize local Arthur packets for any connected reductive group (\cite[Conjecture 8.3.1]{CFMMX22}). This was verified for $\GL_n(F)$ in \cite{CR24, CR26, Lo24, Rid23} and for symplectic, orthogonal, and unitary groups (note that we exclude the case of metaplectic groups) recently in \cite{CHLLRXS26}.

\begin{thm}[{\cite[Theorem 1.1]{CHLLRXS26}}]\label{thm Vogan}
    Suppose $G'$ is not a metaplectic group and that a certain twisted Kazhdan-Lusztig hypothesis for a certain disconnected group holds (see \cite[Theorem 1.1]{CHLLRXS26}). Let
    $\phi=\phi_\psi$ for some $\psi\in\Psi^+(G').$ Then $$
    \Pi_\psi^\pure=\Pi_{\phi_\psi}^{\ABV}.
    $$
    Here $\Pi_\psi^\pure$ is the union of the local Arthur packets attached to $\psi$ for all of the pure inner forms of $G'.$
\end{thm}

We shall assume the theorem for metaplectic groups (just as we are assuming that they have a theory of $\ABV$-packets). In addition, the twisted Kazhdan-Lusztig hypothesis mentioned above is expected to follow from a reduction to the unipotent case, where it was proved in \cite{Cha26}. We shall assume that this holds implicitly whenever we apply the theorem. The only statement that this affects is Proposition \ref{prop connections Adams}.
We also note that \cite[Theorem 1.1]{CHLLRXS26} is stated for $\Psi(G')$, but its extension to $\Psi^+(G')$ follows immediately from \cite[\S4.2]{Moe11a}.

In many ways, results about local Arthur packets are expected to be generalized to ABV-packets, e.g., the Adams conjecture. However, this is not always the case. Indeed, $L$-packets and local Arthur packets of $\GL_n(F)$ are singletons, but ABV-packets of $\GL_n(F)$ may not be singletons (see \cite{CFK22}). In fact, this is the main reason that the Adams conjecture for ABV-packets of $\GL_n(F)$ (Conjecture \ref{conj Adams ABV GL refined}) will not follow trivially from Theorem \ref{thm theta GL} below.

\subsection{The Adams conjecture for ABV-packets}

We continue with the case that $G=G(W_n)$. We briefly recall the notation from \S\ref{sec Theta Correspondence}. We consider towers $\mathcal{V}^\pm$ and let $H^\pm=H(V_m^\pm)$ for $V_m^\pm\in\mathcal{V}^\pm.$ We let $\theta_{-\alpha}^\pm$ denote the local theta correspondence from $G$ to $H^\pm.$ Recall that our choice of $\pm$ depends on the first occurrence of $\pi\in\Pi(G).$ 

Let $\phi\in\Phi(G).$ We define $\phi_\alpha\in\Phi(H^\pm)$ by
\[
\phi_\alpha=(\chi_W\chi_V^{-1}\otimes{}^c\phi^\vee)\oplus\left(\bigoplus_{i=0}^{\alpha-1} \chi_W|\cdot|^{\frac{\alpha-1}{2}-i}\otimes S_1\right).
\]
Here, we recall that $c$ is the generator of $\mathrm{Gal}(E/F)$.
Note that since $G$ is a classical group, we have ${}^c\phi^\vee=\phi$. However, in the next section we consider $G=\GL_n(F)$ and $H=\GL_m(F)$, where we take $c$ to be trivial, but we do not necessarily have that $\phi^\vee=\phi$ and the contragredient will be needed. Thus, for uniformity, we define $\phi_\alpha$ as above.

The map $\phi\mapsto\phi_\alpha$ is understood as a generalization of the map $\psi\mapsto\psi_\alpha.$ Indeed, if $\phi=\phi_\psi$ for some local Arthur parameter $\psi$ of $G,$ then $\phi_{\psi_\alpha}=\phi_\alpha.$ Motivated by Theorem \ref{thm Vogan}, we now formulate the analogues of Conjectures \ref{conj Adams naive} and \ref{conj Adams refined}.

\begin{conj}[The (naive) Adams Conjecture for ABV-packets]\label{conj Adams ABV naive}
If $\pi\in\Pi_\phi^{\ABV},$ then $\theta_{-\alpha}^\pm(\pi)\in\Pi_{\phi_\alpha}^{\ABV}$ provided that $\theta_{-\alpha}^\pm(\pi)\neq 0.$
\end{conj}

By Theorem \ref{thm Vogan}, it follows that Conjecture \ref{conj Adams ABV naive} can and does fail since its analogue for local Arthur packets, Conjecture \ref{conj Adams naive}, also does fail. Analogous to Conjecture \ref{conj Adams refined}, we conjecture the following refinement which we call the refined Adams conjecture for $\ABV$-packets.

\begin{conj}\label{conj Adams ABV refined}
    Let $\pi\in\Pi_\phi^{\ABV}$ for some $\phi\in\Phi(G).$
    \begin{enumerate}
        \item For $\alpha\gg 0,$ we have $\theta_{-\alpha}^\pm(\pi)\in\Pi_{\phi_\alpha}^{\ABV}$.
        \item If $\theta_{-\alpha}^+(\pi)\neq 0,$ then $\theta_{-\alpha}^\pm(\pi)\in\Pi_{\phi_\alpha}^{\ABV}$.
        \item If $\theta_{-\alpha}^\pm(\pi)\in\Pi_{\phi_\alpha}^{\ABV}$ for some $\alpha,$ then $\theta_{-(\alpha+2)}^\pm(\pi)\in\Pi_{\phi_{\alpha+2}}^{\ABV}$.
        \item If $\theta_{-\alpha}^\pm(\pi)\in\Pi_{\phi_\alpha}^{\ABV}$, then $\theta_{-\alpha}^\pm(\pi)\in\Pi_{(\phi_\pi)_\alpha}^{\ABV}$.
        \item Assume that $\pi\in\Pi_\phi^\ABV\cap\Pi_{\phi'}^\ABV$ with $\phi\geq_C\phi'.$ If $\theta_{-\alpha}^\pm(\pi)\in\Pi_{(\phi')_\alpha}^{\ABV}$, then $\theta_{-\alpha}^\pm(\pi)\in\Pi_{\phi_\alpha}^{\ABV}$.
    \end{enumerate}
\end{conj}

\begin{rmk}\label{rmk equal first occurences ABV}
    When $m^+(\pi)=m^-(\pi)$, it follows from the conservation relation that $m^+(\pi)=n+\epsilon_0+1.$ Consequently, $m^{+,\alpha}(\pi)=1$ and so Conjecture \ref{conj Adams ABV refined}(2) implies that the Adams conjecture for ABV-packets (Conjecture \ref{conj Adams ABV naive}) is true for any suitable positive integer $\alpha$ and any $\phi\in\Phi(\pi)$ (see below for notation). That is, the choice of ``going-up'' or ``going-down'' tower does not matter.
\end{rmk}

We  remark that while Conjecture \ref{conj Adams refined}(4) implicitly assumes Conjecture \ref{conj closure max}, its analogue here, Conjecture \ref{conj Adams ABV refined}(4), does not. This is because the analogue of $\psi^{\max}(\pi)$ for ABV-packets is well-understood. Indeed, Conjecture \ref{conj closure max} says that $\psi^{\max}(\pi)$ is the unique maximal element in $\Psi(\pi)$ with respect to $\geq_C.$ This is not the case for ABV-packets. Here, the analogue of  $\Psi(\pi)$ is
\[
\Phi(\pi)=\{\phi\in\Phi(G) \ | \ \pi\in\Pi_\phi^\ABV\}.
\]
There is a unique maximal element of $\Phi(\pi)$ with respect to $\geq_C$, namely $\phi_\pi$. Indeed, this is an immediate consequence of Propositions \ref{prop ABV closure} and \ref{prop Lpacket in ABV}. However, if $\phi_\pi$ is not of Arthur type, then $\phi_\pi\neq \phi_{\psi^{\max}(\pi)}.$ Indeed, \cite[Conjecture 1.4]{HLLZ25} predicts that $\phi_\pi\geq_C \phi_{\psi^{\max}(\pi)}$ and so by passing to ABV-packets, it sometimes is necessary to go beyond $\psi^{\max}(\pi)$. Note that Conjecture \ref{conj Adams ABV refined}(5) implies Conjecture \ref{conj Adams ABV refined}(4) based on the above discussion.

We show that Conjecture \ref{conj Adams ABV refined} generalizes Conjecture \ref{conj Adams refined}.

\begin{prop}\label{prop connections Adams}
    The following holds.
    \begin{enumerate}
        \item Conjecture \ref{conj Adams ABV refined}(1) implies Conjecture \ref{conj Adams refined}(1).
        \item Conjecture \ref{conj Adams ABV refined}(2) implies Conjecture \ref{conj Adams refined}(2).
        \item Conjecture \ref{conj Adams ABV refined}(3) implies Conjecture \ref{conj Adams refined}(3).
        \item Conjecture \ref{conj Adams ABV refined}(5) implies Conjecture \ref{conj Adams refined}(4).
    \end{enumerate}
\end{prop}

\begin{proof}
    The first three statements are immediate since $\Pi_\psi^\pure=\Pi_{\phi_\psi}^{\ABV}$ and $\Pi_\psi\subseteq\Pi_\psi^\pure.$ We give the details for the last statement. Assume that $\pi\in\Pi_\psi$ and $\theta_{-\alpha}^\pm(\pi)\in\Pi_{\psi_\alpha}.$ By Theorem \ref{thm Vogan}, we have that $\theta_{-\alpha}^\pm(\pi)\in\Pi_{\phi_{\psi_\alpha}}^{\ABV}.$ Note that $\phi_{\psi^{\max}(\pi)}\geq_C\phi_\psi$ by Conjecture \ref{conj closure max}. Also, by definition of $\psi^{\max}(\pi),$ we have that $\pi\in\Pi_{\psi^{\max}(\pi)}$ and hence $\pi\in\Pi_{\phi_{\psi^{\max}(\pi)}}^\ABV$ by Theorem \ref{thm Vogan}. Therefore, Conjecture \ref{conj Adams ABV refined}(5) implies that $\theta_{-\alpha}^\pm(\pi)\in\Pi_{\phi_{\psi^{\max}(\pi)_\alpha}}^{\ABV}.$ Finally, by Theorem \ref{thm Vogan}, it follows that $\theta_{-\alpha}^\pm(\pi)\in\Pi_{(\psi^{\max}(\pi))_\alpha}.$ That is, Conjecture \ref{conj Adams ABV refined}(5) implies Conjecture \ref{conj Adams refined}(4).
\end{proof}

Note that the analogue of Conjecture \ref{conj Adams ABV refined}(5) is not conjectured in Conjecture \ref{conj Adams refined}. This is because the closure order $\geq_C$ is natural in the setting of ABV-packets but less so for local Arthur packets. However,  Conjecture \ref{conj Adams ABV refined}(5) and Theorem \ref{thm Vogan} imply the following analogous conjecture for local Arthur packets.

\begin{conj}\label{conj Adams closure}
    Suppose that $\pi\in\Pi_\psi\cap\Pi_{\psi'}$ for some $\psi,\psi'\in\Psi^+(G)$ with $\psi\geq_C\psi'.$ If $\theta_{-\alpha}^\pm(\pi)\in\Pi_{\psi'_\alpha}$, then $\theta_{-\alpha}^\pm(\pi)\in\Pi_{\psi_\alpha}$.
\end{conj}

For symplectic-even orthogonal dual pairs, Theorem \ref{thm Adams refined} states that Conjecture \ref{conj Adams refined} is true. However, Conjecture \ref{conj Adams closure} remains open in all cases including symplectic-even orthogonal dual pairs. Indeed, the argument in \cite{Haz24} uses an ordering $\geq_O$ on $\Psi(\pi)$ which implies $\geq_C$ (\cite[Theorem 4.5(1)]{HLLZ25}) but the reverse is not true (see \cite[Example 5.9(2)]{HLLZ25}). Therefore, \cite[Theorem 1.3]{Haz24} only provides partial evidence for Conjecture \ref{conj Adams closure}.

\begin{rmk}
    With the above discussion in mind, it is possible that Conjecture \ref{conj Adams ABV refined}(5) is false. In this case, one would expect that there exists a stronger partial order, analogous to $\geq_O$, on $\Phi(\pi)$ which would replace $\geq_C$ in Conjecture \ref{conj Adams ABV refined}(5). This order should replace $\geq_C$ in Conjecture \ref{conj Adams closure} as well. However, $\geq_C$ is the most natural partial ordering on $\Phi(\pi)$ and at the time of writing, no counter-example to Conjecture \ref{conj Adams ABV refined}(5) or Conjecture \ref{conj Adams closure} is known to the author.
\end{rmk}

We now turn towards verifying an implication of Conjecture \ref{conj Adams ABV refined}. Specifically, we verify that Conjecture \ref{conj Adams ABV refined}(4) holds on the going-down tower. This follows immediately from the following result.

\begin{lemma}\label{lemma going down Adams ABV}
    For any suitable positive integer $\alpha$, we have that $\theta_{-\alpha}^-(\pi)\in\Pi_{(\phi_\pi)_\alpha}^{\ABV}$.
\end{lemma}
\begin{proof}
    By \cite{AG17a, BH21}, we have $\phi_{\theta_{-\alpha}^-(\pi)}=(\phi_\pi)_\alpha.$ Thus we have that $\theta_{-\alpha}^-(\pi)\in\Pi_{\phi_{\theta_{-\alpha}^-(\pi)}}.$ From Proposition \ref{prop Lpacket in ABV}, it follows that $\theta_{-\alpha}^-(\pi)\in\Pi_{\phi_{\theta_{-\alpha}^-(\pi)}}^\ABV=\Pi_{(\phi_\pi)_\alpha}^{\ABV}.$
\end{proof}

The above lemma entirely resolves the failure of the Adams conjecture for local Arthur packets. Indeed, recall Conjecture \ref{conj stable Arthur}. Let $\alpha_0$ be the minimum among all positive integers $\alpha$ such that $\theta_{-\alpha}^-(\pi)\in\Pi_{(\psi^{\max}(\pi))_\alpha}$. If $\alpha_0\geq 3,$ then Conjecture \ref{conj stable Arthur} predicts that $\theta_{-(\alpha_0-2)}^-(\pi)$ is not of Arthur type. In this case, we must consider $\ABV$-packets instead of local Arthur packets. Conjecture \ref{conj Adams ABV refined}(5) suggests that we should move to some $\phi\in\Phi(\pi)$ for which $\phi\geq_C\phi_{\psi^{\max}(\pi)}$ and check if the Adams conjecture for $\ABV$-packets holds for $\phi$. Lemma \ref{lemma going down Adams ABV} says that we may always do this. Indeed, one may take $\phi=\phi_\pi.$

Curiously, Conjecture \ref{conj Adams ABV refined}(2) predicts that the analogue of Lemma \ref{lemma going down Adams ABV} should also hold on the going-up tower. However, the above proof does not verify this. Indeed, this is primarily a consequence of \cite[Theorem 4.5]{AG17a} (and \cite[Theorem 6.8]{BH22}); see also \cite[p. 558]{Sak17}. More specifically, it is possible that $\phi_{\theta_{-\alpha}^+(\pi)}\neq(\phi_\pi)_\alpha.$

We end this section by giving some examples which hint at a possible relationship between Conjecture \ref{conj stable Arthur} and Conjecture \ref{conj Adams ABV refined}. First, we fix some notation.

\begin{defn}
    Let $\pi\in\Pi_{\psi}$ for some local Arthur parameter $\psi\in\Psi(G).$ We define
    \[
    d^\pm(\pi,\psi):=\min\{\mathrm{suitable} \ \alpha>0 \  | \ \theta_{-\alpha}^\pm(\pi)\in\Pi_{\psi_\alpha}\}.
    \]
    Similarly, let $\pi\in\Pi_{\phi}^\ABV$ for some $\phi\in\Phi(G).$ We define
    \[
    d^\pm(\pi,\phi):=\min\{\mathrm{suitable} \ \alpha>0 \  | \ \theta_{-\alpha}^\pm(\pi)\in\Pi_{\phi_\alpha}^\ABV\}.
    \]
\end{defn}

If $\phi=\phi_\psi$ for some $\psi\in\Psi(G)$, then Theorem \ref{thm Vogan} shows that $d^\pm(\pi,\psi)=d^\pm(\pi,\phi_\psi).$ Conjectures \ref{conj Adams refined}(2) and \ref{conj Adams ABV refined}(2) essentially become $d^+(\pi,\psi)=m^{+,\alpha}(\pi)$ and $d^+(\pi,\phi)=m^{+,\alpha}(\pi)$, respectively. Assuming Conjectures \ref{conj Adams refined}(3) and \ref{conj Adams ABV refined}(3) would make the previous sentence precise. However, we wish to pay particular attention to Conjecture \ref{conj Adams refined}(4). It states that for any $\psi\in\Psi(\pi),$ we have
\[
d^-(\pi,\psi^{\max}(\pi))\leq d^-(\pi,\psi),
\]
i.e., the Adams conjecture for local Arthur packets holds in its greatest generality for $\psi^{\max}(\pi).$ Consequently, it is desirable to understand how to compute $d^-(\pi,\psi^{\max}(\pi))$ (especially as it is the conjectural lower bound for determining when the theta lift is unitary, see Remark \ref{rmk nonunitary}). We suspect that it is related to determining when $\phi_\alpha$ is of Arthur type for some $\phi\in\Phi(\pi)$ (this is not entirely correct as written; one would need to focus on the ``$\chi_W$-part'' of $\phi_\alpha$). 

Let $\phi\in\Phi(G)$ and $\pi\in\Pi_\phi^\ABV$. By \cite{AG17a, BH21},  for any suitable positive integer $\alpha$, we have that $\phi_{\theta_{-\alpha}^-(\pi)}=(\phi_\pi)_\alpha$. This observation is useful in the following examples.

\begin{exmp}\label{exmp curious}
    This example is \cite[Example 6.1]{Haz24}. There is a representation $\pi$ of $\Sp_{10}(F)$ of Arthur type with $L$-parameter
    \[
    \phi_\pi=|\cdot|^3\otimes S_1 + |\cdot|^{-3}\otimes S_1+\mathbbm{1}_{W_F}\otimes S_1+\mathbbm{1}_{W_F}\otimes S_3+\chi_V\otimes S_5
    \]
    and $d^-(\pi,\psi^{\max}(\pi))=5,$ where 
    \[
    \psi^{\max}(\pi)=\mathbbm{1}_{W_F}\otimes S_1\otimes S_7+\mathbbm{1}_{W_F}\otimes S_3\otimes S_1 +\mathbbm{1}_{W_F}\otimes S_1\otimes S_1.
    \]
    
    Note that $\phi_\pi$ is not of Arthur type (and $\chi_V=\chi_W=\mathbbm{1}_{W_F}$), but
    \begin{align*}
    (\phi_\pi)_5&=|\cdot|^3\otimes S_1+|\cdot|^2\otimes S_1+|\cdot|^1\otimes S_1+|\cdot|^{-1}\otimes S_1+|\cdot|^{-2}\otimes S_1  \\
    &+ |\cdot|^{-3}\otimes S_1 +\mathbbm{1}_{W_F}\otimes S_1+\mathbbm{1}_{W_F}\otimes S_1+\mathbbm{1}_{W_F}\otimes S_3+\mathbbm{1}_{W_F}\otimes S_5
    \end{align*}
    is of Arthur type. Indeed, we have $(\phi_\pi)_5=\phi_\psi$ where
    \[
    \psi=\mathbbm{1}_{W_F}\otimes S_1\otimes S_7+\mathbbm{1}_{W_F}\otimes S_1\otimes S_1+\mathbbm{1}_{W_F}\otimes S_3\otimes S_1+\mathbbm{1}_{W_F}\otimes S_5\otimes S_1.
    \]
    Furthermore, $(\phi_\pi)_\alpha$ is not of Arthur type for any $\alpha\neq 5.$ In other words, we have $(\phi_\pi)_\alpha$ is of Arthur type if and only if $\alpha=d^-(\pi,\psi^{\max}(\pi)).$ Equivalently, from \cite{HLL26} (see Remark \ref{rmk nonunitary}), it follows that $\theta_{-\alpha}^-(\pi)$ is unitary for $\alpha>0$ if and only if $\alpha\geq5=d^-(\pi,\psi^{\max}(\pi)).$ In particular, $\theta_{-3}^-(\pi)$ and $\theta_{-1}^-(\pi)$ are not unitary.
\end{exmp}

We remark that computing $d^-(\pi,\psi^{\max}(\pi))$ is done algorithmically which limits its theoretical use. On the other hand, computing whether $(\phi_\pi)_\alpha$ (or generally $\phi_\alpha$) is of Arthur type is incredibly simple. Having a relation between the two would be desirable in light of Conjecture \ref{conj stable Arthur} and Remark \ref{rmk nonunitary}.  Here is another example.

\begin{exmp}\label{exmp curious 2}
    There is a unique representation $\pi$ of $\Sp_{10}(F)$ of Arthur type with $L$-parameter
    \[
    \phi_\pi=|\cdot|^\frac{3}{2}\otimes S_2 +|\cdot|^\frac{-3}{2}\otimes S_2+\mathbbm{1}_{W_F}\otimes S_1+\mathbbm{1}_{W_F}\otimes S_3+\mathbbm{1}_{W_F}\otimes S_3
    \]
    and satisfying
    \[
    \psi^{\max}(\pi)=\mathbbm{1}_{W_F}\otimes S_2\otimes S_4+\mathbbm{1}_{W_F}\otimes S_3\otimes S_1.
    \]
    Note that $(\phi_\pi)_\alpha$ is never of Arthur type. However, $d^-(\pi,\psi^{\max}(\pi))=1$ and so we would not expect a relationship with $(\phi_\pi)_\alpha$ being of Arthur type.
\end{exmp}

\section{The Adams conjecture for general linear groups}\label{sec The Adams conjecture for general linear groups}

For this section, we focus on the case that $G=G_n=\GL_n(F)$ and $H=H_m=\GL_m(F)$, where $\alpha=m-n\geq 0.$ In this setting, $c$ is trivial and the analogues of the characters $\chi_W$ and $\chi_V$ are the trivial characters. We remark that the local Langlands correspondence is known for general linear groups (\cite{HT01, Hen00, Sch13}). The pair $(\GL_n(F),\GL_m(F))$ forms a reductive dual pair of type II. Consequently, there is a local theta correspondence $\theta_{-\alpha}(\pi)$ from $G$ to $H$. Before recalling its precise description, we recall some representation theory.

We fix $B_n$ to be the Borel subgroup of $G_n$ consisting of upper triangular matrices. Consider a parabolic subgroup $P$ of $G_n$ with Levi decomposition $P=MN$ where $M$ is a Levi subgroup isomorphic to $G_{n_1}\times\cdots \times G_{n_r},$ where $n_1+\cdots+n_r=n.$ For $\pi_i\in\Pi(G_{n_i}),$ we denote the normalized parabolic induction by
\[
\mathrm{Ind}_P^{G_n}(\pi_1\otimes\cdots\otimes\pi_r)=\pi_1\times\cdots\times\pi_r.
\]
Given $\pi\in\Pi(G_n)$, we let $\pi^\vee$ denote its contragredient.

The Langlands classification for $G_n$ was established by Zelevinsky using segments (\cite{Zel80}) which we recall here. A segment $[x,y]_\rho$ is the set $\{\rho|\cdot|^x,\dots,\rho|\cdot|^y\}$, where
$\rho$ denotes an irreducible unitary supercuspidal representation of $\GL_k(F)$  and $x,y\in\mathbb{R}$ with $x-y\in\mathbb{Z}_{\geq 0}$. The parabolic induction
\[
\rho|\cdot|^x\times\cdots\times\rho|\cdot|^y
\]
possesses a unique irreducible subrepresentation, which we denote by $\Delta_\rho[x,y]$. Moreover, $\Delta_\rho[x,y]$ is essentially square integrable. Furthermore, any irreducible essentially square-integrable representation of $\GL_n(F)$ arises in this way.

The Langlands classification is given in terms of essentially square-integrable representations which we state following \cite[\S6]{Min08}. For $i=1,\dots,r$, let $\pi_i\in\Pi(G_{n_i})$ be essentially square-integrable, i.e., $\pi_i=\Delta_{\rho_i}[x_i,y_i]$. Then letting $\alpha_i=-\frac{x_i+y_i}{2}$, we have that $\pi_i|\cdot|^{\alpha_i}$ is square-integrable. Let $\sigma$ be a permutation of $\{1,\dots,r\}$ such that $\alpha_{\sigma(i)}\leq\alpha_{\sigma(j)}$ if $i<j.$ Then the induced representation
\[
\pi_{\sigma(1)}\times\cdots\times\pi_{\sigma(r)}
\]
has a unique irreducible quotient $\pi$ known as the Langlands quotient. In this setting, we write
\[
\pi=L(\pi_1,\dots,\pi_r).
\]
Moreover, any $\pi\in\Pi(G_n)$ can be realized as such a Langlands quotient. In the above situation, we write 
\[
M(\pi)=\pi_{\sigma(1)}\times\cdots\times\pi_{\sigma(r)}
\]
and say that $M(\pi)$ is the standard module of $\pi.$

Since $L$-packets of $\GL_n(F)$ and $\GL_m(F)$ are singletons, we may take the following theorem of M{\'i}nguez to be our definition of the local theta correspondence in this setting.

\begin{thm}[{\cite[Theorem 1]{Min08}}]\label{thm theta GL}
    Suppose that $n\leq m$ and $\pi\in\Pi(\GL_n(F))$ is the unique irreducible quotient of 
    \[
    M(\pi)=\tau_1\times\cdots\times\tau_r.
    \] 
    Then $\theta_{-\alpha}(\pi)=L(|\cdot|^{-\frac{m-n-1}{2}},\dots,|\cdot|^{\frac{m-n-1}{2}},\tau^\vee_1,\dots,\tau^\vee_r)$. 
\end{thm}

Theorem \ref{thm theta GL} gives the following corollary immediately.
\begin{cor}\label{cor theta Lpar}
    Let $\pi\in\Pi(\GL_n(F)).$ Then \[\phi_{\theta_{-\alpha}(\pi)}=\phi_\pi^\vee\oplus \left(\oplus_{i=0}^{\alpha-1} |\cdot|^{\frac{\alpha-1}{2}-i}\otimes S_1\right).\]
    In particular, if $\alpha=0,$ then $\phi_{\theta_{0}(\pi)}=\phi_\pi^\vee$ and  $\theta_{0}(\pi)=\pi^\vee$ is the contragredient of $\pi.$
\end{cor}

One difference between the local theta correspondence for dual pairs of classical groups and the local theta correspondence for general linear groups is that there is only one tower. That is, we do not have the concept of a going-up or going-down tower in this setting. Also if $\psi$ is a local Arthur parameter of $\GL_n(F),$ then $\Pi_\psi=\Pi_{\phi_\psi}$. We immediately obtain the Adams conjecture for $\GL_n(F).$ We remark that this result is already well-understood (see also \cite[Theorem 6.7]{Ada89}), but we include a proof for completeness.

\begin{lemma}\label{lemma Adams Arthur GL}
    Let $\pi\in\Pi_\psi$ for some local Arthur parameter $\psi$ of $\GL_n(F).$ Then
    $\theta_{-\alpha}(\pi)\in\Pi_{\psi_\alpha}$ for any $\alpha\in\mathbb{Z}_{\geq0}.$ Here $\psi_\alpha=\psi^\vee\oplus 1_{W_F}\otimes S_1\otimes S_\alpha$.
\end{lemma}

\begin{proof}
    For any local Arthur parameter $\psi$ of $\GL_n(F),$ we have that $\Pi_\psi\cap\Pi_{\psi'}\neq\emptyset$ if and only if $\Pi_{\phi_\psi}\cap\Pi_{\phi_{\psi'}}\neq\emptyset$. Since $L$-packets are disjoint, we obtain that $\Pi_\psi\cap\Pi_{\psi'}\neq\emptyset$ if and only if $\psi=\psi'.$ Since $\pi\in\Pi_\psi=\Pi_{\phi_\psi},$ it follows that $\Psi(\pi)=\{\psi\}.$ Also, we have $\phi_\pi=\phi_\psi.$ We obtain from Corollary \ref{cor theta Lpar} that $(\phi_\pi)_\alpha=(\phi_\psi)_\alpha=\phi_{\theta_{-\alpha}(\pi)}$ and hence $\theta_{-\alpha}(\pi)\in\Pi_{\phi_{\theta_{-\alpha}(\pi)}}=\Pi_{\psi_\alpha}$ which proves the lemma.
\end{proof}

Next we state the (naive?) Adams conjecture for ABV-packets of $\GL_n(F).$ 

\begin{conj}[The (naive?) Adams Conjecture for ABV-packets of $\GL_n(F)$]\label{conj Adams ABV GL naive}
If $\pi\in\Pi_\phi^{\ABV}$ for some $\phi\in\Phi(\GL_n(F)),$ then $\theta_{-\alpha}(\pi)\in\Pi_{\phi_\alpha}^{\ABV}$.
\end{conj}

We remark on why we wrote ``naive?'' here. Recall that we do not have the concept of a going-up or going-down tower for general linear groups. If general linear groups behave like a going-up tower, then we should expect Conjecture \ref{conj Adams ABV GL naive} to hold as written. This is in contrast with Conjecture \ref{conj Adams ABV naive} which we know does fail. On the other hand, if general linear groups behave like a going-down tower, then we should expect Conjecture \ref{conj Adams ABV GL naive} to possibly fail. It is unclear which is the correct expectation currently.

Regardless of the situation, we make the following refined Adams conjecture for ABV-packets of $\GL_n(F).$ 

\begin{conj}\label{conj Adams ABV GL refined}
    Let $\pi\in\Pi_\phi^{\ABV}$ for some $\phi\in\Phi(\GL_n(F)).$
    \begin{enumerate}
        \item For $\alpha\gg 0,$ we have $\theta_{-\alpha}(\pi)\in\Pi_{\phi_\alpha}^{\ABV}$.
        \item If $\theta_{-\alpha}(\pi)\in\Pi_{\phi_\alpha}^{\ABV}$ for some $\alpha\in\mathbb{Z}_{\geq 0},$ then $\theta_{-(\alpha+1)}(\pi)\in\Pi_{\phi_{\alpha+1}}^{\ABV}$.
        \item If $\theta_{-\alpha}(\pi)\in\Pi_{\phi_\alpha}^{\ABV}$, then $\theta_{-\alpha}(\pi)\in\Pi_{(\phi_\pi)_\alpha}^{\ABV}$.
        \item Assume that $\pi\in\Pi_\phi^\ABV\cap\Pi_{\phi'}^\ABV$ with $\phi\geq_C\phi'.$ If $\theta_{-\alpha}(\pi)\in\Pi_{(\phi')_\alpha}^{\ABV}$, then $\theta_{-\alpha}(\pi)\in\Pi_{\phi_\alpha}^{\ABV}$.
    \end{enumerate}
\end{conj}

\begin{rmk}
    While we stated the above as a conjecture, we will prove both parts (1) and (3) in this article (see the below discussion).
\end{rmk}

We remark that Conjecture \ref{conj Adams ABV GL refined} is the analogue of Conjecture \ref{conj Adams ABV refined}. Indeed, Conjecture \ref{conj Adams ABV GL refined}(1, 2, 3, 4) is the analogue of Conjecture \ref{conj Adams ABV refined}(1, 3, 4, 5), respectively. The omission of the analogue of  Conjecture \ref{conj Adams ABV refined}(2) is because general linear groups only have one tower (see the above discussion). Again, we have that Conjecture \ref{conj Adams ABV GL refined}(4) implies Conjecture \ref{conj Adams ABV GL refined}(3).

We remark on why the proof of Lemma \ref{lemma Adams Arthur GL} does not generalize to ABV-packets. The argument requires that $\Pi_\psi=\Pi_{\phi_\psi}$ and hence is a singleton. Recall that by Theorem \ref{thm Vogan}, we can view ABV-packets as generalizations of local Arthur packets. However, for ABV-packets of $\GL_n(F)$, it is not true that $\Pi_\phi^\ABV=\Pi_\phi$ generally. Indeed, there is a counter-example for $\GL_{16}(F)$ (\cite{CFK22}). This makes Conjecture \ref{conj Adams ABV GL refined} nontrivial.

We have two pieces of evidence for Conjecture \ref{conj Adams ABV GL refined}. The first piece of evidence is that the analogue of Lemma \ref{lemma going down Adams ABV} holds. Indeed, Corollary \ref{cor theta Lpar} and Proposition \ref{prop Lpacket in ABV} imply that $\theta_{-\alpha}(\pi)\in\Pi_{(\phi_\pi)\alpha}^\ABV$ for any nonnegative integer $\alpha.$ This proves Conjecture \ref{conj Adams ABV GL refined}(3) in full generality.

The second piece of evidence is more substantial. We confirm Conjecture \ref{conj Adams ABV GL refined}(1) in full generality (Theorem \ref{thm GL main thm}). The majority of the remainder of this article is devoted to this verification.

\subsection{Representation theory} 

In this subsection, we continue to focus on the case $G_n=\GL_n(F)$.
Let $W_F$ be the Weil group associated to $F$ and $\widehat{G}_n(\C)=\GL_n(\C)$ be the complex dual group of $G_n.$ Since $G_n$ is split, an $L$-parameter of $G_n$ may be regarded as a $\widehat{G}_n(\C)$-conjugacy class of an admissible homomorphism $\phi:W_F\times\SL_2(\C)\rightarrow\widehat{G}_n(\C)$ (\cite[\S8]{Bor79}). Let $\Phi(G_n)$ denote the set of $L$-parameters of $G_n.$ We do not distinguish a representative $\phi$ from its conjugacy class. The local Langlands correspondence for $G_n$ is well-understood (\cite{HT01, Hen00, Sch13}). One consequence is that there is a bijection $rec:\Pi(G_n)\rightarrow\Phi(G_n).$ The $L$-packet attached to $\phi$ is $\Pi_\phi:=rec^{-1}(\phi).$ Since the map is a bijection, the $L$-packet is a singleton. We write $\Pi_{\phi}=\{\pi_\phi\}.$ Conversely, for $\pi\in\Pi(G_n)$, we let $\phi_\pi:=rec(\pi)$ denote the $L$-parameter of $\pi.$

Let $\lambda$ be an infinitesimal character of $G_n.$ Recall that we have
\[
\Phi_\lambda(G_n)=\{\phi\in\Phi(G_n) \ | \ \lambda_\phi=\lambda\}
\]
and
\[
\Pi_\lambda(G_n)=\{\pi\in\Pi(G_n) \ | \ \lambda_{\phi_\pi}=\lambda\}
\]
and that both of these sets are finite. The Grothendieck group of finite length representations of $G_n$ with infinitesimal parameter $\lambda$ is denoted by $K\Pi_\lambda(G_n).$ Given $\pi\in\Pi_\lambda(G_n),$ we let $[\pi]$ denote its image in $K\Pi_\lambda(G_n).$ The set $B_{\Pi_\lambda}=\{[\pi] \ | \ \pi\in\Pi_\lambda(G_n)\}$ forms a $\mathbb{Z}$-basis for $K\Pi_\lambda(G_n).$

\subsection{Perverse Sheaves}\label{sec perverse sheaves}  We continue to assume that $G_n=\GL_n(F).$ We recall some notation from \S\ref{sec ABV-packets}. Let $\lambda$ be an infinitesimal character of $G_n$. The group $H_\lambda$ acts on the Vogan variety $V_\lambda$ with finitely many orbits and we let $C_\lambda(G_n)$ denote the collection of these orbits. These orbits are in bijection with $\Phi_\lambda(G_n)$ (\cite[Proposition 4.2.2]{CFMMX22}). For $\phi\in\Phi_\lambda(G_n),$ we let $C_\phi\in C_\lambda(G_n)$ denote the corresponding orbit. Through the orbit closure, we defined a partial order $\geq_C$ on $\Phi_\lambda(G_n)$ (Definition \ref{def C ordering}).

We let $D_{H_\lambda}(V_\lambda)$ denote the $H_\lambda$-equivariant derived category of $\ell$-adic sheaves on $V_\lambda$ and $\Per_{H_\lambda}(V_\lambda)$ denote the category of $H_\lambda$-equivariant perverse sheaves on $V_\lambda$ (see \cite{Ach21}).  Vogan's perspective on the local Langlands correspondence (\cite{Vog93}) gives a bijection between $\Pi_\lambda(G_n)$ and the simple objects in $\Per_{H_\lambda}(V_\lambda)$ (up to isomorphism). For $\pi\in\Pi_\lambda(G_n)$, we write $\mathcal{P}(\pi)$ for the corresponding simple perverse sheaf. For $G_n$, it is simple to describe these objects. Namely, $\mathcal{P}(\pi)=\mathcal{IC}(\mathbbm{1}_{C_{\phi_\pi}})$, where $\mathbbm{1}_{C_{\phi_\pi}}$ denotes the trivial local system on $C_{\phi_\pi}$ and $\mathcal{IC}(\cdot)$ denotes the intersection cohomology complex.

Let $K\Per_\lambda(G_n)$ denote the Grothendieck group of $\Per_{H_\lambda}(V_\lambda).$ Given $\mathcal{F}\in\Per_{H_\lambda}(V_\lambda)$, we let $[\mathcal{F}]$ denote its image in $K\Per_\lambda(G_n)$. Vogan's perspective on the Langlands classification shows that $K\Per_\lambda(G_n)$ has a $\mathbb{Z}$-basis given by $B_{\Per_\lambda}=\{\mathcal{IC}(\mathbbm{1}_{C}) \ | \ C\in C_\lambda(G_n)\}.$

Now, we introduce a perfect pairing between the Grothendieck groups above. We define
\[
\langle\cdot,\cdot\rangle:K\Pi_\lambda(G_n)\times K\Per_\lambda(G_n)\rightarrow\Z,
\]
by defining it on the basis $B_{\Pi_\lambda}\times B_{\Per_\lambda}$ via
\begin{equation}\label{eqn pairing}
\langle[\pi],[\mathcal{F}]\rangle=\left\{\begin{array}{cc}
    (-1)^{d(\pi)} &  \mathrm{if} \ \mathcal{F}=\mathcal{P}(\pi),   \\
    0 & \mathrm{otherwise,} 
\end{array}\right.
\end{equation}
where $d(\pi):=\dim C_{\phi_\pi},$ and extending linearly.

For $C\in C_\lambda(G_n)$, Cunningham et al. attach an element $\eta_{C}^\Evs\in K\Pi_\lambda(G_n)$ (\cite[\S8.4]{CFMMX22}). In our setting, we have
\[
\eta_{C}:=\eta^{\Evs}_{C}=(-1)^{d(C)}\sum_{\pi\in\Pi_\lambda(G_n)}(-1)^{d(\pi)}\rank(\Evs_C(\mathcal{P}(\pi))) [\pi],
\]
where $d(C):=\dim C$ and $\Evs$ is the functor on perverse sheaves defined in \cite[\S7.9]{CFMMX22}. See also Equation \eqref{eqn eta dist general}.
Let $\phi\in\Phi_\lambda(G_n).$ We set $\eta_{\phi}=\eta^\Evs_{C_{\phi}}.$ Recall that
\[
\Pi_\phi^{\ABV}=\{\pi\in\Pi_\lambda(G_n) \ | \ \Evs_{C_\phi}(\mathcal{P}(\pi))\neq 0\}.
\]

Thus, it follows that we may use the pairing of the Grothendieck groups and $\eta_{\phi}$ to determine $\Pi_\phi^{\ABV}.$
\begin{lemma}[{\cite[Proposition 1.6]{CR26}}]\label{lemma in ABV}
    We have that $\pi\in\Pi_\phi^{\ABV}$ if and only if 
    \[
    \langle \eta_{\phi}, [\mathcal{P}(\pi)] \rangle \neq 0.
    \]
\end{lemma}

We remark that the above lemma follows simply from observing that \cite[Proposition 1.6]{CR26} holds for a general $L$-parameter, rather than an Arthur parameter.

Let $\lambda=\lambda_\phi$ and $\lambda^\vee=\lambda_{\phi^\vee}.$ We relate $V_\lambda$ and $V_{\lambda^\vee}$ using the framework introduced in \cite[Section 10.2.1]{CFMMX22}. This framework will be used extensively in later sections when we discuss conormal bundles. Along the way, we will also show that the contragredient preserves ABV-packets for $\GL_n(F).$ In general, we have that $\phi^\vee(w,x)={}^t\phi(w,x)^{-1}$ for $w\in W_F$ and $x\in\SL_2(\BC).$

First, by \cite[Theorem 5.1.1]{CFMMX22}, we may assume that $\lambda$ is unramified, i.e., trivial on $I_F$, and $\chi(\lambda(\Fr))\in\mathbb{R}_{>0}$ for any character $\chi:\widehat{T}\rightarrow\GL_1(\BC)$, where $\widehat{T}$ is any torus in $\GL_n(\BC)$ containing $\lambda(\Fr).$ Consequently, we may write
\[
\lambda=m_1|\cdot|^{x_1}+m_2|\cdot|^{x_2}+\cdots+m_r|\cdot|^{x_r}
\]
where $m_i\in\mathbb{Z}_{\geq 1}$ denotes the multiplicity and $x_i\in\mathbb{R}$ with $x_i>x_{i+1}$ for $i=1,\dots,r-1.$ Since $y\in V_\lambda$ if and only if $\Ad(\lambda(\Fr))y=q_Fy,$ we may assume that $x_i-1=x_{i+1}$ for $i=1,2,\dots,r-1$ (otherwise the Vogan variety decomposes as a product of such Vogan varieties). For $i=1,\dots,r,$ let $E_i$ denote the $q_F^{x_i}$-eigenspace of $\lambda(\Fr)$. We have $m_i=\dim(E_i).$ Furthermore, we have that
\begin{equation}\label{eqn eigenspace decomp}
V_\lambda\cong\Hom(E_1,E_2)\times\Hom(E_2,E_3)\times\cdots\times\Hom(E_{r-1},E_r).    
\end{equation}
In this setting, we have
\[
H_\lambda\cong\GL(E_1)\times\GL(E_2)\times\cdots\times\GL(E_r).
\]
Let
\[
y=(y_1,\dots,y_{r-1})\in\Hom(E_1,E_2)\times\cdots\times\Hom(E_{r-1},E_r)\cong V_\lambda.
\]
The $H_\lambda$-orbit of $y$ is the set of \[z=(z_1,\dots,z_{r-1})\in\Hom(E_1,E_2)\times\cdots\times\Hom(E_{r-1},E_r)\] such that $\rank(y_j\circ y_{j-1}\circ\cdots\circ y_i)=\rank(z_j\circ z_{j-1}\circ\cdots\circ z_i)$ for any $1\leq i\leq j\leq r-1.$

Now, note that $\lambda^\vee=m_1|\cdot|^{-x_1}+m_2|\cdot|^{-x_2}+\cdots+m_r|\cdot|^{-x_r}.$ For $i=1,\dots,r,$ let $w_i=-x_{r-i+1}$ and $F_i$ denote the $q_F^{w_i}$-eigenspace of $\lambda^\vee(\Fr)$. Analogously to the above discussion, we have
\[
V_{\lambda^\vee}\cong\Hom(F_1,F_2)\times\Hom(F_2,F_3)\times\cdots\times\Hom(F_{r-1},F_r).
\]
Of course, for $i=1,\dots,r,$ we also have $F_i\cong E_{r-i+1}$ and thus
\[
V_{\lambda^\vee}\cong\Hom(E_r,E_{r-1})\times\Hom(E_{r-1},E_{r-2})\times\cdots\times\Hom(E_{2},E_1).
\]
Similarly, we have 
\[
H_{\lambda^\vee}\cong\GL(E_r)\times\GL(E_{r-1})\times\cdots\times\GL(E_1).
\]
This is a reflection of $\phi^\vee={}^t\phi^{-1}$ on the Vogan variety. Indeed, let
\[
y=(y_1,\dots,y_{r-1})\in\Hom(E_1,E_2)\times\cdots\times\Hom(E_{r-1},E_r)\cong V_\lambda.
\]
Define \[y^\vee:=({}^ty_{r-1},\dots,{}^ty_{1})\in \Hom(E_r,E_{r-1})\times\cdots\times\Hom(E_{2},E_1).\]
Similarly, let 
\[
h=(h_1,\dots,h_{r})\in\GL(E_1)\times\cdots\times\GL(E_r)\cong H_\lambda.
\]
and
define \[h^\vee:=({}^th_{r}^{-1},\dots,{}^th_{1}^{-1})\in \GL(E_r)\times\cdots\times\GL(E_1).\]
The maps $y\mapsto y^\vee$ and $h\mapsto h^\vee$ induce an $H_\lambda$-equivariant isomorphism $V_\lambda\cong V_{\lambda^\vee}.$ Furthermore, it sends the orbit $C_\phi$ to $C_{\phi^\vee}$. Indeed, this follows from simply computing the ranks.

In summary, we have an equivariant isomorphism $V_\lambda\cong V_{\lambda^\vee}$ which sends $C_\phi$ to $C_{\phi^\vee}.$ Since the vanishing cycles functor is functorial for equivariant isomorphisms, we obtain the following lemma.

\begin{lemma}\label{lemma contragredient ABV}
    Suppose that $\pi\in\Pi_{\phi}^\ABV.$ Then $\pi^\vee\in\Pi_{\phi^\vee}^\ABV.$
\end{lemma}

We note that this is expected more generally; however, the contragredient may permute the elements in an $L$-packet (see \cite{Kal13}). For $\GL_n(F)$, we avoided this issue as the $L$-packets are all singletons. Here is a simple example illustrating the above ideas.

\begin{exmp}
    Suppose that $\phi=|\cdot|^{\frac{3}{2}}+|\cdot|^{\frac{1}{2}}.$ Then $\phi^\vee=|\cdot|^{-\frac{1}{2}}+|\cdot|^{-\frac{3}{2}}.$
    We have that 
    \[
    V_{\lambda_\phi}=\left\{\begin{pmatrix}
        0 & y \\
        0 & 0 
    \end{pmatrix} \ | \ y\in\BC\right\}=V_{\lambda_{\phi^\vee}}.
    \]
    The eigenvalues of $\lambda(\Fr)$ are $q_F^{x_i}$ where $x_1=\frac{3}{2}$ and $x_2=\frac{1}{2}.$ We have that the eigenspaces of both eigenvalues are 1-dimensional and hence
    \begin{align*}
    V_{\lambda}&\cong\Hom(\mathbb{C},\mathbb{C}) \\
    \begin{pmatrix}
        0 & y \\
        0 & 0 
    \end{pmatrix}&\mapsto y,
    \end{align*}
    where we consider $y\in\BC$ as the linear transformation defined by $y(z)=yz$ for any $z\in\BC.$ 

    Similarly, we have that
    \[
    V_{\lambda_{\phi^\vee}}=\left\{\begin{pmatrix}
        0 & y \\
        0 & 0 
    \end{pmatrix} \ | \ y\in\BC\right\}. 
    \]
    The eigenvalues of $\lambda^\vee(\Fr)$ are $q_F^{w_i}$ where $w_1=-\frac{1}{2}=-x_2$ and $w_2=-\frac{3}{2}=-x_1.$ We have that the eigenspaces of both eigenvalues are 1-dimensional and hence
    \begin{align*}
    V_{\lambda^\vee}&\cong\Hom(\mathbb{C},\mathbb{C}) \\
    \begin{pmatrix}
        0 & y \\
        0 & 0 
    \end{pmatrix}&\mapsto y,
    \end{align*}
    where we consider $y\in\BC$ as the linear transformation defined by $y(z)=yz$ for any $z\in\BC.$ 
    The map $y\mapsto y^\vee$ is simply the identity map and hence $V_{\lambda_{\phi}}\cong V_{\lambda_{\phi^\vee}}$ is also the identity map.

    Furthermore, the $L$-parameter $\phi$ corresponds to the $0$-orbit in $V_{\lambda_\phi}$. This corresponds to $y=0$. We have $0^\vee=0$ which corresponds to the $L$-parameter $\phi^\vee$.
\end{exmp}

Our next goal is to state a fixed point formula (Theorem \ref{thm GL fixed point formula}) which will be the key step in our proof of Theorem \ref{thm GL main thm}.

For $i=1,\dots,r$, let $\phi_i\in\Phi(G_{n_i})$, and $n=n_1+\dots+n_r.$ We set $G^\times:=G_{n_1}\times\cdots\times G_{n_r}$ and $\phi^\times=\phi_{1} \times \cdots\times\phi_r$. Note that $\widehat{G}^\times(\C)=\widehat{G}_{n_1}(C)\times\cdots\times\widehat{G}_{n_r}(\C)$ and $\Pi_{\phi^\times}=\Pi_{\phi_1}\times\cdots\times\Pi_{\phi_r}.$ We also let $\lambda^\times=\lambda_1\times\cdots\times\lambda_r$ be the corresponding infinitesimal parameter, where $\lambda_i=\lambda_{\phi_i}$. Furthermore, we have that the Vogan variety is $V_{\lambda^\times}=V_{\lambda_1}\times\cdots\times V_{\lambda_r}$ and $H_{\lambda^\times}=H_{\lambda_1}\times\cdots\times H_{\lambda_r}$. There is an action of $H_{\lambda^\times}$ on $V_{\lambda^\times}$  in the obvious manner. Alternatively, these could be directly computed from the definitions in \cite[\S4]{CFMMX22}.

We let $\phi=\phi_1+\dots+\phi_r\in\Phi(G_n)$ and $\lambda=\lambda_\phi.$
Let $s\in\widehat{G}_n(\C)$ be of finite order (and hence semi-simple) such that $Z_{\widehat{G}_n(\C)}(s)\cong\widehat{G}^\times$. The resulting inclusion $\widehat{G}^\times \hookrightarrow \widehat{G}_n(\C)$ induces inclusions $H_{\lambda^\times}\hookrightarrow H_{\lambda}$ and
\[
\varepsilon: V_{\lambda^\times}\hookrightarrow V_{\lambda}
\]
which is equivariant for the action by $H_{\lambda^\times}.$
Indeed, we have that 
\[V_{\lambda^\times}=V_{\lambda}^s:=\{x\in V_\lambda \ | \ \Ad(s)x=x\}.
\]

Let $\varepsilon^*:\mathrm{D}_{H_{\lambda}}(V_{\lambda})\rightarrow\mathrm{D}_{H_{\lambda^\times}}(V_{\lambda^\times})$ denote the equivariant restriction functor for the equivariant derived categories. As a shorthand, we write
\[
\mathcal{F}|_{V_{\lambda^\times}}:=\varepsilon^*\mathcal{F}.
\]
We note that $\varepsilon^*$ is an exact functor, but does not preserve perverse sheaves. 

We define a special case of endoscopic lifting (see \cite[Definition 26.18]{ABV92} or \cite[\S4]{CR26}) to be the linear transformation
\[\Lift_{G^\times}^{G_n}: K\Pi_{\lambda^\times}(G^\times)\rightarrow K\Pi_{\lambda}(G_n)\]
defined by
\[
\langle \Lift_{G^\times}^{G_n}[\pi],[\mathcal{F}] \rangle_\lambda = \langle [\pi], [\varepsilon^*\mathcal{F}]\rangle_{\lambda^\times}.
\]
In this setting, the endoscopic lifting is simple to describe. 

\begin{prop}[{\cite[Proposition 4.5]{CR26}}]\label{prop lift}
    We continue with the above notation. Let $[\pi]\in K\Pi_{\lambda^\times}(G^\times)$ and $P$ be the standard parabolic subgroup of $G_n$ whose Levi subgroup is isomorphic to $G^\times.$ Then \[
    \Lift_{G^\times}^{G_n}[\pi]=[\mathrm{Ind}_{P}^{G_n}\pi].
    \]
\end{prop}

We note that the above proposition is proved assuming the Kazhdan-Lusztig hypothesis which is proved in \cite{Bal25} (see also \cite{Sol25}).

An equation of the form $\langle\eta_\phi,[\mathcal{F}]\rangle_\lambda=\langle\eta_{\phi^\times},[\mathcal{F}|_{V_{\lambda^\times}}]\rangle_{\lambda^\times}$ is called a fixed point formula as it is usually obtained from a Lefschetz fixed point formula, e.g., \cite[Theorem 25.8]{ABV92}. Note that it is equivalent to $\Lift_{G^\times}^{G_n}(\eta_{\phi^\times})=\eta_\phi.$

Let $\phi\in\Phi(G_n).$ Recall that
\[
\phi_\alpha=\phi^\vee\oplus\left(\bigoplus_{i=0}^{\alpha-1} |\cdot|^{\frac{\alpha-1}{2}-i}\otimes S_1\right).
\]
Let $\phi^\alpha=\bigoplus_{i=0}^{\alpha-1} |\cdot|^{\frac{\alpha-1}{2}-i}\otimes S_1.$ Then $\phi_\alpha=\phi^\vee+\phi^\alpha$ and we let $\phi^\times=\phi^\vee\times\phi^\alpha.$ Note that $\phi^\vee\in\Phi(G_n),$ $\phi^\alpha\in\Phi(G_\alpha)$, and $\phi_\alpha\in\Phi(G_m).$ Let $\lambda_\alpha=\lambda_{\phi_\alpha}$ and $\lambda^\times=\lambda_{\phi^\vee}\times\lambda_{\phi^\alpha}.$ Per the above discussion, we have inclusions $H_{\lambda^\times}\hookrightarrow H_{\lambda_\alpha}$ and
\[
\varepsilon: V_{\lambda^\times}\hookrightarrow V_{\lambda_\alpha}
\]
which is equivariant for the action by $H_{\lambda^\times}.$ We work towards showing that for any $[\mathcal{F}]\in K\Per_{\lambda_\alpha}(G_m),$ we have (Theorem \ref{thm GL fixed point formula})
    \begin{equation}\label{eqn GL fixed point formula}
    \langle\eta_{\phi_\alpha},[\mathcal{F}]\rangle_{\lambda_\alpha}=\langle\eta_{\phi^\times},[\mathcal{F}|_{V_{\lambda^\times}}]\rangle_{\lambda^\times}.
    \end{equation}

\subsection{Conormal bundles}\label{sec Conormal bundles} In this section, we study conormal bundles. In particular, understanding the regular part of the conormal bundle is crucial for applying the fixed point formula (Theorem \ref{thm GL fixed point formula}) later. The main result of this section is to relate regular conormal vectors in our specific case (Proposition \ref{prop regular conorm}). We continue to assume throughout this section that $G_n=\GL_n(F)$.

Let $\phi\in\Phi(G_n)$ and $\lambda=\lambda_\phi.$ We let $V_\lambda^*$ denote the dual Vogan variety to $V_\lambda.$
By considering
\[
{}^tV_{\lambda}:=\{x\in\mathrm{Lie}(K_\lambda) \ | \ \mathrm{Ad}(\lambda(\Fr))x=q_F^{-1}x\}
\]
we identify the dual Vogan variety $V_\lambda^*\cong {}^tV_{\lambda}$ hereinafter (\cite[Proposition 6.2.1]{CFMMX22}).

The conormal bundle is denoted by
\[
\Lambda_\lambda:=\{(x,y)\in V_\lambda\times V^*_\lambda \ | \ [x,y]=0\},
\]
where $[\cdot,\cdot]$ denotes the Lie bracket (\cite[Proposition 6.2]{CFMMX22}).

We define
\[
\Lambda_{C_\phi}:=\{(x,y)\in C_\phi\times V^*_\lambda \ | \ [x,y]=0\}.
\]
For an $H_\lambda$-orbit $B$ of $V_\lambda^*$, we consider
\[
\Lambda_{B}:=\{(y,x)\in B\times V_\lambda \ | \ [y,x]=0\}.
\]
By \cite[Lemma 6.5]{CFMMX22}, there exists a unique $H_\lambda$-orbit of $V_\lambda^*$, denoted $(C_{\phi})^*=C_\phi^*$ such that 
\[
\overline{\Lambda}_{C_\phi}\cong\overline{\Lambda}_{(C_{\phi})^*}.
\]
We say that $(C_{\phi})^*$ is the dual orbit to $C_\phi.$ The $H_\lambda$ orbits of $V_\lambda^*$ are also in bijection with $\Phi_\lambda(G_n).$ We let $\hat{\phi}$ be the $L$-parameter (called the Pyasetskii dual) corresponding to $(C_{\phi})^*.$ We define the regular part of the conormal bundle of $C_\phi$ to be
\[
\Lambda_{C_{\phi}}^{reg}:=\Lambda_{C_\phi} \setminus \bigcup_{\substack{C' \\ C_{\phi}\subsetneq \overline{C'}}} \overline{\Lambda}_{C'}.
\]

Consider the $L$-parameter $\phi^{\alpha}=\oplus_{i=0}^{\alpha-1} \chi_W|\cdot|^{\frac{\alpha-1}{2}-i}\otimes S_1$ which corresponds to the $0$-orbit in $V_{\lambda_{\phi^{\alpha}}}.$ Let $x_{\phi^{\alpha}}=0$ and so $\Lambda_{x_{\phi^{\alpha}}}=\Lambda_{C_{\phi^{\alpha}}}\cong V_{\lambda_{\phi^{\alpha}}}^*$. It follows that $\Lambda_{C_{\phi^{\alpha}}}^{reg}$ is the set of $(0,y)\in \Lambda$ where $y\in {}^tC_{\chi_W\otimes S_\alpha}$ (this is the unique open orbit in $V_{\lambda_{\phi^{\alpha}}}^*$). In particular, this set is nonempty. Let $(x_{\phi^{\alpha}},y_{\phi^{\alpha}})\in \Lambda_{C_{\phi^{\alpha}}}^{reg}$ be arbitrary.
For $\phi^\vee,$ we let $(x_{\phi^\vee},y_{\phi^\vee})\in \Lambda_{C_\phi^\vee}^{reg}$ also be arbitrary. 

Recall that $\phi_\alpha=\phi^\vee+\phi^\alpha.$ Let $\phi^\times=\phi^\vee\times\phi^\alpha$ and  $\lambda^\times=\lambda_{\phi^\times}$. We have that $V_{\lambda^\times}=V_{\lambda_{\phi^\vee}}\times V_{\lambda_{\phi^\alpha}}$. Consider the embeddings $\varepsilon:V_{\lambda^\times}\hookrightarrow V_{\lambda_{\phi_\alpha}}$ and ${}^t\varepsilon:V_{\lambda^\times}^*\hookrightarrow V_{\lambda_{\phi_\alpha}}^*$. We also consider $\varepsilon'=\varepsilon\times {}^t\varepsilon$.

To establish the fixed point formula \eqref{eqn GL fixed point formula}, we must find $(x,y)\in\Lambda_{C_{\phi^\times}}^{reg}$ such that $\varepsilon'(x,y)\in \Lambda_{C_{\phi_\alpha}}^{reg}.$
Next, we provide two running examples which will explain why we require $\alpha\gg 0$ later. First is the example where our strategy will succeed.

\begin{exmp}\label{exmp GL works 1} Let $G=\GL_2(F),$ $\alpha=2,$
 $\phi=\phi^\vee=|\cdot|^{\frac{1}{2}}+|\cdot|^{\frac{-1}{2}}$, and $\phi^\alpha=\phi$ (so our theta lift is from $\GL_2(F)$ to $\GL_4(F)).$ Let $\lambda=\lambda_\phi.$ The Vogan varieties for $V_\lambda=V_{\lambda_{\phi^\vee}}=V_{\lambda_{\phi^\alpha}}$ are given by
\[
V_{\lambda}=\left(\begin{matrix}
    0 & x \\
    0 & 0
\end{matrix} \right).
\]
We have $H_\lambda=\GL_1(\BC)\times\GL_1(\BC)$ with action given by $(h_1,h_2)\cdot x\mapsto\frac{h_1}{h_2}x$. There are 2 orbits, the $0$-orbit and the open orbit ($x\neq0$). We let $C$ denote the $0$-orbit which corresponds to $\phi^\vee=\phi^\alpha$. Also, we have
\[
V_{\lambda}^*=\left(\begin{matrix}
    0 & 0 \\
    y & 0
\end{matrix} \right)
\]
with the action of $H_{\lambda}$ given by $(h_1,h_2)\mapsto \frac{h_2}{h_1}y$. A choice of $(x,y)\in\Lambda_{C}^{reg}$ is given by $x=0, y=\left(\begin{matrix}
    0 & 0 \\
    1 & 0
\end{matrix} \right).$ Now $\phi_\alpha=2\phi$ (note that $\alpha=2$) and it corresponds to the $0$-orbit $C_\alpha$ in 
\[
V_{\lambda_{\phi_\alpha}}=\left\{\left(\begin{matrix}
    0 & x \\
    0 & 0
\end{matrix} \right) \ | \ x\in\mathrm{Mat}_{2\times 2}(\mathbb{C})\right\}.
\]
We have $H_{\lambda_{\phi_\alpha}}=\GL_2(\BC)\times \GL_2(\BC)$ with action given by \[(a,b)\cdot\left(\begin{matrix}
    0 & x \\
    0 & 0
\end{matrix} \right)\mapsto\left(\begin{matrix}
    0 & axb^{-1} \\
    0 & 0
\end{matrix} \right).\]
We have that
\[
V_{\lambda_{\phi_\alpha}}^*=\left\{\left(\begin{matrix}
    0 & 0 \\
    y & 0
\end{matrix} \right) \ | \ y\in\mathrm{Mat}_{2\times 2}(\mathbb{C})\right\}.
\]
We have $(x_{\phi_\alpha},y_{\phi_\alpha})\in\Lambda_{C_\alpha}^{reg}$ where $x_{\phi_\alpha}=0$ and  
\[y_{\phi_\alpha}=\left(\begin{matrix}
    0 & 0 \\
    \left(\begin{matrix}
    1 & 0 \\
    0 & 1
\end{matrix} \right) & 0
\end{matrix} \right).\]
The embedding of $V_{\lambda_{\phi^\vee}}\times V_{\lambda_{\phi^\alpha}}$ into $V_{\lambda_{\phi_\alpha}}$ is given by
\[
\left(\left(\begin{matrix}
    0 & x_1 \\
    0 & 0
\end{matrix} \right),\left(\begin{matrix}
    0 & x_2 \\
    0 & 0
\end{matrix} \right)\right)\mapsto\left(\begin{matrix}
    0 & x \\
    0 & 0
\end{matrix} \right),
\]
where
\[
x=\left(\begin{matrix}
    x_1 & 0 \\
    0 & x_2
\end{matrix} \right).
\]
The embedding of $V_{\lambda_\phi}^*\times V_{\lambda_{\phi^\alpha}}^*$ into $V_{\lambda_{\phi_\alpha}}^*$ is given similarly by taking the transpose. Thus we see that the image of $((x_\phi,y_\phi),(x_{\phi},y_{\phi}))\in \Lambda_{C}^{reg}\times \Lambda_{C}^{reg}$ is precisely $(x_{\phi_\alpha},y_{\phi_\alpha})\in\Lambda_{C_\alpha}^{reg}$ as desired.
\end{exmp}

The next example is where we see that the condition $\alpha\gg0$ will be needed in our strategy.

\begin{exmp}\label{exmp GL fails 1}
    Let $G=\GL_2(F)$, $\alpha=2,$ $\phi=\phi^\vee=|\cdot|^{\frac{3}{2}}+|\cdot|^{\frac{-3}{2}},$ and $\phi^\alpha=|\cdot|^{\frac{1}{2}}+|\cdot|^{\frac{-1}{2}}$. The geometry for the Vogan variety of $\phi^\alpha$ is the same as in Example \ref{exmp GL works 1}. Let $\lambda=\lambda_\phi=\lambda_{\phi^\vee}.$ Then $V_\lambda=\left\{\begin{pmatrix}
        0 & 0 \\ 0 & 0
    \end{pmatrix}\right\}.$ The group $H_\lambda$ is isomorphic to $\GL_1(\BC)\times\GL_1(\BC)$, but the action is trivial. Consequently, we have that $\Lambda_{C_{\phi^\vee}}=\Lambda_{C_{\phi^\vee}}^{reg}$ is the singleton $(x_{\phi^\vee},y_{\phi^\vee})$ where $x_{\phi^\vee}=y_{\phi^\vee}=\begin{pmatrix}
        0 & 0 \\ 0 & 0
    \end{pmatrix}.$

    On the other hand $\phi_\alpha=\phi^\vee+\phi^\alpha=|\cdot|^{\frac{3}{2}}+|\cdot|^{\frac{1}{2}}+|\cdot|^{\frac{-1}{2}}+|\cdot|^{\frac{-3}{2}}.$ Let $\lambda_\alpha=\lambda_{\phi_\alpha}.$ The Vogan variety is 
    \[
    V_{\lambda_\alpha}=\left\{\begin{pmatrix}
      0 & a & &  \\
     &  0 & b &  \\
     &  & 0 & c  \\
     & & & 0  \\
\end{pmatrix} \ | \ a,b,c\in\mathbb{C}  \right\}.
    \]
    The group $H_{\lambda_\alpha}$ is the standard torus of $\GL_4(\BC),$ i.e. it is isomorphic to $\GL_1(\BC)\times\GL_1(\BC)\times \GL_1(\BC)\times\GL_1(\BC)$, and its action is given by the usual simple roots. 
    The $0$-orbit in $V_{\lambda_\alpha}$ corresponds to $\phi_\alpha.$ 
    
    However, the dual orbit of $C_{\phi_\alpha}$ is the unique open orbit corresponding to the tempered parameter $S_4.$ That is,
    \[
    C_{\lambda_\alpha}^*=\left\{\begin{pmatrix}
      0 &  & &  \\
     a &  0 &  &  \\
     & b & 0 &   \\
     & & c & 0  \\
\end{pmatrix} \ | \ a,b,c\in\mathbb{C}^\times  \right\}.
    \]
    Note that by \cite[Lemma 6.4.2]{CFMMX22}, we have $\Lambda_{C_{\phi_\alpha}}^{reg}\subseteq C_{\phi_\alpha} \times C_{\phi_\alpha}^*.$ The embedding $\varepsilon:V_\lambda\times V_{\lambda_{\phi^\alpha}}\hookrightarrow V_{\lambda_\alpha}$
    is given by
    \[
    \left(\begin{pmatrix}
        0 & 0 \\ 0 & 0
    \end{pmatrix}, \begin{pmatrix}
        0 & x \\ 0 & 0
    \end{pmatrix} \right)\mapsto \begin{pmatrix}
      0 &  & &  \\
     &  0 & x &  \\
     &  & 0 &   \\
     & & & 0  \\
\end{pmatrix}.
    \]
    The map ${}^t\varepsilon:V_\lambda^*\times V_{\lambda_{\phi^\alpha}}^*\hookrightarrow V_{\lambda_\alpha}^*$ is given by taking the transpose of these matrices. Consequently, the image of ${}^t\varepsilon$ does not intersect with $C_{\lambda_\alpha}^*.$ That is, in contrast with Example \ref{exmp GL works 1}, there does not exist an element $(x,y)\in\Lambda_{C_{\phi^\times}}^{reg}$ such that $\varepsilon'(x,y)\in\Lambda_{C_{\phi_\alpha}}^{reg}.$ The reason for this is that $\alpha=2$ is too small. We will discuss this example more in Example \ref{exmp GL fails 2}. 
\end{exmp}

To differentiate between these examples, we introduce the following definition. 

\begin{defn}
    Write
    \[
    \phi^\vee=\bigoplus_{i=1}^r |\cdot|^{x_i}\otimes S_{a_i}  \oplus\left(\bigoplus_{i=r+1}^k\rho_i\otimes S_{a_i}\right),
    \]
    where $\rho_i\neq|\cdot|^x$ for any $x\in\mathbb{R}.$ We define the set of trivial exponents of $\phi^\vee$ to be
    \[
    \exp_{\mathbbm{1}_{W_F}}(\phi^\vee)=\bigcup_{i=1}^r\{\frac{a_i-1}{2}+x_i,\frac{a_i-1}{2}-1+x_i, \dots,\frac{1-a_i}{2}+x_i\}.
    \]
    Let $\beta\in\frac{1}{2}\mathbb{Z}$, (later we take $\beta=\frac{\alpha-1}{2}$).
    We define the set of trivial $\beta$-exponents of $\phi^\vee$ to be $\exp^\beta_{\mathbbm{1}_{W_F}}(\phi^\vee)=\exp_{\mathbbm{1}_{W_F}}(\phi^\vee)\cap(\beta+\Z).$ Finally, we let $m^\beta_{\mathbbm{1}_{W_F}}(\phi^\vee)=\max\{|x| \ | \ x \in \exp^\beta_{\mathbbm{1}_{W_F}}(\phi^\vee)\}.$ 
\end{defn}

\begin{rmk}
    Let $\beta=\frac{\alpha-1}{2}.$ It is possible that $\exp^\beta_{\mathbbm{1}_{W_F}}$ is empty. In this case, $\exp_{\mathbbm{1}_{W_F}}(\phi^\vee)\cap(\beta+\Z)=\emptyset$ and, by convention, we  write $\beta\geq m^\beta_{\mathbbm{1}_{W_F}}(\phi^\vee).$
\end{rmk}

We recall the previous examples.

\begin{exmp}\label{exmp GL works 2}
    We continue Example \ref{exmp GL works 1}. In this case, $\alpha=2$ and so we have $\beta=\frac{\alpha-1}{2}=\frac{1}{2}\in\frac{1}{2}+\mathbb{Z}.$ We have
    \[
    \exp_{\mathbbm{1}_{W_F}}(\phi^\vee)=\exp^\beta_{\mathbbm{1}_{W_F}}(\phi^\vee)=\{\frac{1}{2},-\frac{1}{2}\}
    \]
    and
    $m^\beta_{\mathbbm{1}_{W_F}}(\phi^\vee)=\frac{1}{2}.$ In particular $\frac{\alpha-1}{2}\geq m^\beta_{\mathbbm{1}_{W_F}}(\phi^\vee).$
\end{exmp}

\begin{exmp}\label{exmp GL fails 2}
     We continue Example \ref{exmp GL fails 1}. In this case, $\alpha=2$ and so we have $\beta=\frac{\alpha-1}{2}=\frac{1}{2}\in\frac{1}{2}+\mathbb{Z}.$ We have
    \[
    \exp_{\mathbbm{1}_{W_F}}(\phi^\vee)=\exp^\beta_{\mathbbm{1}_{W_F}}(\phi^\vee)=\{\frac{3}{2},-\frac{3}{2}\}
    \]
    and
    $m^\beta_{\mathbbm{1}_{W_F}}(\phi^\vee)=\frac{3}{2}.$ In contrast with Example \ref{exmp GL works 2}, we have $\frac{\alpha-1}{2}=\frac{1}{2}< m^\beta_{\mathbbm{1}_{W_F}}(\phi^\vee).$
\end{exmp}

The main goal of the next section is to prove the following proposition.

\begin{prop}\label{prop regular conorm} Let $\beta=\frac{\alpha-1}{2} \geq m^\beta_{\mathbbm{1}_{W_F}}(\phi^\vee)$ and \[((x_{\phi^\vee},y_{\phi^\vee}),(x_{\phi^\alpha},y_{\phi^\alpha}))\in \Lambda_{C_{\phi^\times}}^{reg}.\]
Then
    $\varepsilon'((x_{\phi^\vee},y_{\phi^\vee}),(x_{\phi^\alpha},y_{\phi^\alpha}))\in\Lambda_{C_{\phi_\alpha}}^{reg}.$
\end{prop}

\subsection{\texorpdfstring{Proof of Proposition \ref{prop regular conorm}}{}} The proof of Proposition \ref{prop regular conorm} will occupy the entirety of this section. We recall some notation. We have $\phi\in\Phi(G_n)$, $\lambda=\lambda_\phi$, $\phi_\alpha=\phi^\vee+\phi^\alpha$, $\phi^\times=\phi^\vee\times\phi^\alpha$ and  $\lambda^\times=\lambda_{\phi^\times}$. We also have that $V_{\lambda^\times}=V_{\lambda_{\phi^\vee}}\times V_{\lambda_{\phi^\alpha}}$and the embeddings $\varepsilon:V_{\lambda^\times}\hookrightarrow V_{\lambda_{\phi_\alpha}}$ and ${}^t\varepsilon:V_{\lambda^\times}^*\hookrightarrow V_{\lambda_{\phi_\alpha}}^*$. We also consider $\varepsilon'=\varepsilon\times {}^t\varepsilon$.

We make use of the decomposition of $V_{\lambda_\alpha}$ from Equation \eqref{eqn eigenspace decomp} which we recall here with slightly modified notation. Again, by \cite[Theorem 5.1.1]{CFMMX22}, we may assume that $\lambda_\alpha$ is unramified and $\chi(\lambda_\alpha(\Fr))\in\mathbb{R}_{>0}$ for any character $\chi:\widehat{T}\rightarrow\GL_1(\BC)$, where $\widehat{T}$ is any torus in $\GL_m(\BC)$ containing $\lambda_\alpha(\Fr).$ Consequently, we  write
\[
\lambda_\alpha=m_1|\cdot|^{e_1}+m_2|\cdot|^{e_2}+\cdots+m_\alpha|\cdot|^{e_\alpha}
\]
where $m_i\in\mathbb{Z}_{\geq 1}$ denotes the multiplicity and $e_i:=\beta-i+1$ for $i=1,\dots,\alpha$, where $\beta=\frac{\alpha-1}{2}$ (note that it suffices to consider this situation since the Vogan variety decomposes as a direct product of Vogan varieties based on the $\mathbb{Z}$-cosets).  For $i=1,\dots,\alpha,$ we let $E_i$ denote the $q_F^{e_i}$-eigenspace of $\lambda_\alpha(\Fr)$. We have $m_i=\dim(E_i).$ Furthermore, we have that
\begin{equation*}
V_{\lambda_\alpha}\cong\Hom(E_1,E_2)\times\Hom(E_2,E_3)\times\cdots\times\Hom(E_{\alpha-1},E_\alpha).    
\end{equation*}
In this setting, we have
\[
H_{\lambda_\alpha}\cong\GL(E_1)\times\GL(E_2)\times\cdots\times\GL(E_\alpha).
\]

We have similar decompositions for 
\[
V_{\lambda}\cong\Hom(D_1,D_2)\times\Hom(D_2,D_3)\times\cdots\times\Hom(D_{\alpha-1},D_\alpha),  
\]
where $D_i$ denotes the $q_F^{e_i}$-eigenspace for $\lambda_{\phi}(\Fr)$ (which may be $0$) and
\[V_{\lambda_{\phi^\alpha}}\cong\Hom(L_1,L_2)\times\cdots\times\Hom(L_{\alpha-1},L_\alpha)\cong\mathbb{C}^{\alpha-1}.\]
Indeed, note that the $q_F^{e_i}$-eigenspace for $\lambda_{\phi^\alpha}(\Fr)$ is a one-dimensional line
\[
L_i=\mathbb{C}l_i.
\]
Thus we have
\[
E_i=D_i^\vee\oplus L_i,
\]
where $D_i^\vee$ is identified with the $q_F^{e_i}$-eigenspace of $\lambda_{\phi^\vee}(\Fr).$ Note that $D_i^\vee\cong D_{\alpha+1-i}^*$.

Let $x=(x_i)_{i=1}^{\alpha-1}\in \bigoplus_{i=1}^{\alpha-1}\Hom(E_i,E_{i+1})\cong V_{\lambda_\alpha}$ and $y=(y_i)_{i=1}^{\alpha-1}\in \bigoplus_{i=1}^{\alpha-1}\Hom(E_{i+1},E_{i})\cong V_{\lambda_\alpha}^*$. In addition, we set $x_0=y_0=x_\alpha=y_\alpha=0.$ We have 
\begin{equation}\label{eqn obs}
[x,y]=0 \ \text{if and only if} \ x_{i-1}y_{i-1}=y_ix_i \ \text{for any}\ 1\leq i\leq \alpha.    
\end{equation}

Consider $((x_{\phi^\vee},y_{\phi^\vee}),(x_{\phi^\alpha},y_{\phi^\alpha}))\in \Lambda_{C_{\phi^\times}}^{reg}$ and write $x_{\phi^\vee}=(x_{\phi^\vee,i})_{i=1}^\alpha$, $y_{\phi^\vee}=(y_{\phi^\vee,i})_{i=1}^\alpha$, and similarly for
$x_{\phi^\alpha}$ and $y_{\phi^\alpha}.$ Note that $x_{\phi^\alpha,i}=0$ and $y_{\phi^\alpha,i}\neq0$ for any $i=1,\dots,\alpha-1$ since $C_{\phi^\alpha}$ is the $0$-orbit. Rescaling if necessary, we assume that $y_{\phi^\alpha,i} l_{i+1}=l_{i}$. Consequently, in these coordinates we have
\[
\epsilon'((x_{\phi^\vee},y_{\phi^\vee}),(x_{\phi^\alpha},y_{\phi^\alpha}))=(x_{\phi^\vee}\oplus 0,y_{\phi^\vee}\oplus (l_{i+1}\mapsto l_i)_{i=1}^{\alpha-1})
\]

Let $\Lambda^\alpha:=\{(x,y)\in\Lambda_{\lambda_\alpha} \ | y_i l_{i+1}=l_i \ \mathrm{for} \ i=1,\dots,\alpha-1\}$. Here, we write $y=(y_i)_{i=1}^\alpha$.
\begin{lemma}
    We have an isomorphism $\Lambda^\alpha\cong \Lambda_{\lambda_{\phi^\vee}}\times\mathbb{C}^t$, where $t=\sum_{i=2}^\alpha\dim D_i^\vee$.
\end{lemma}

\begin{proof}
    Let $(x,y)\in \Lambda^\alpha$ and write $x=(x_i)_{i=1}^\alpha$ and similarly for $y$. Since $y_i l_{i+1}=l_i$, we have $x_i l_i=0$ for any $i$. Indeed, from Equation \eqref{eqn obs}, we obtain
    \[
    x_{\alpha-1}l_{\alpha-1}=x_{\alpha-1}y_{\alpha-1} l_\alpha= y_\alpha x_\alpha l_\alpha=0
    \]
    since $x_\alpha=0$ by convention. Next, we obtain 
    \[
    x_{i-1}l_{i-1}=x_{i-1}y_{i-1} l_i= y_i x_i l_i=0
    \]
    for any $i$.

    Now let $d\in D_i^\vee$ and write
    \[
    x_id=a_i d+b_i(d)l_{i+1}, \ y_{i-1}d=p_{i-1} d+q_{i-1}(d)l_{i-1},
    \]
    where $a_i\in\Hom(D_i^\vee,D_{i+1}^\vee)$, $b_i\in\Hom(D_{i}^\vee,\mathbb{C}), q_i\in\Hom(D_{i+1}^\vee,\mathbb{C})$ and $p_i\in\Hom(D_{i+1}^\vee, D_{i}^\vee)$. We show that the map $(x,y)\mapsto((a,p),q)$ is an isomorphism.
    
    Recall that $x_il_i=0$ for any $i$. Thus for $d\in D_{i}^\vee$, we have
    \begin{align*}
    x_{i-1}y_{i-1}d=x_{i-1}p_{i-1}d+x_{i-1}q_{i-1}(d)l_{i-1}&=x_{i-1}p_{i-1}d \\
    &=a_{i-1}p_{i-1}d+b_{i-1}(p_{i-1}d)l_{i}
    \end{align*}
    and
    \begin{align*}
    y_ix_id=y_ia_id+y_ib_i(d)l_{i+1}=p_{i}a_id+q_i(a_id)l_{i}+b_i(d)l_{i}.
    \end{align*}
    From Equation \eqref{eqn obs}, comparing the projections to $D_i^\vee$ shows that
    $a_{i-1}p_{i-1}=p_ia_i$ for any $i$. That is, we have $(a,p)\in\Lambda_{\lambda_{\phi^\vee}}$.
    On the other hand, for $i=1,\dots,\alpha-1$, by comparing the projections to $L_{i}$, we obtain
    \[
    b_i=b_{i-1}p_{i-1}-q_ia_i,
    \]
    where we set $b_0p_0:=0$. Note that if
    $b_{\alpha-1}p_{\alpha-1}=0,$ then $b=(b_i)$ is entirely determined by $a_i$, $p_i$, and $q_i$ and consequently the map $(x,y)\mapsto((a,p),q)$ is an isomorphism.

    We show that $b_{\alpha-1}p_{\alpha-1}=0.$ Set $R_i:=p_ip_{i+1}\cdots p_{\alpha-1}:D_\alpha^\vee\rightarrow D_{i}^\vee$. Then $R_{i-1}=p_{i-1}R_i$ and from the above formula for $b_i$, we have
    \[
    b_iR_i=b_{i-1}R_{i-1}-q_ia_i R_i.
    \]
    We obtain
    \[
    b_{\alpha-1}p_{\alpha-1}=b_{\alpha-1}R_{\alpha-1}=-\sum_{i=1}^{\alpha-1}q_ia_iR_i.
    \]
    Recall that $a_{i-1}p_{i-1}=p_i a_i$ for any $i$.
    In particular, we have that $0=a_{\alpha-1} p_{\alpha-1}=a_{\alpha-1} R_{\alpha-1}$ and so $a_iR_i=a_ip_iR_{i+1}=p_{i+1}a_{i+1}R_{i+1}=0$ for any $i.$ Therefore, $b_{\alpha-1}p_{\alpha-1}=0$ which completes the proof of the lemma.
\end{proof}

Let $(x,y):=\epsilon'((x_{\phi^\vee},y_{\phi^\vee}),(x_{\phi^\alpha},y_{\phi^\alpha}))\in\Lambda_{\lambda_\alpha}$ and write $x=(x_i)_{i=1}^\alpha$ and $y=(y_i)_{i=1}^\alpha$. Note that $(x,y)\in  \Lambda^\alpha$ and under the isomorphism of the above lemma, we have $(x,y)\mapsto((x_{\phi^\vee},y_{\phi^\vee}),0)\in \Lambda_{\lambda_{\phi^\vee}}\times\mathbb{C}^t$.

Let $U\subseteq \Lambda_{\lambda_\alpha}\times E_\alpha$ be the open subset of triples $(x',y',u)$ satisfying that $u_i:=y_i'\cdots y_{\alpha-1}'u$ has nonzero $L_i$-component for $i<\alpha$ and also that $u_\alpha:=u$ has nonzero $L_\alpha$-component. Write $u_i=d_i+t_il_i$ for some $t_i\in\mathbb{C}^\times.$ Define $g_i(u)\in\GL(E_i)$ to be the identity map on $D_i^\vee$ and sends $l_i$ to $\frac{l_i-d_i}{t_i}$. Note that $g_i(u)u_i=l_i$. Set $g(u):=(g_i)$ and define
\[
f:U\rightarrow\Lambda^\alpha \times\prod_{i=1}^\alpha(D_i^\vee\times\mathbb{C}^\times),
\]
via
\[
f(x',y',u)=(g(u)(x',y'), (d_i,t_i)_{i=1}^\alpha).
\]
Note that $f(U)\subseteq \Lambda^\alpha \times\prod_{i=1}^\alpha(D_i^\vee\times\mathbb{C}^\times)$ since $\Lambda_{\lambda_\alpha}$ is $H_{\lambda_\alpha}$-stable and \[(g(u) y_i')l_{i+1}=g_i(u)y_i'(g_{i+1}(u)^{-1}l_{i+1})=g_i(u)y_i'u_{i+1}=g_i(u)u_i=l_i.\]

\begin{lemma}
    The morphism $f$ is an isomorphism.
\end{lemma}

\begin{proof}
    We exhibit an inverse morphism. Let $((x'',y''),(d_i,t_i)_{i=1}^\alpha)\in \Lambda^\alpha \times\prod_{i=1}^\alpha(D_i^\vee\times\mathbb{C}^\times).$ Set $w_i:=d_i+t_il_i$ and define $g(w)$ analogously as above. Consider
    \[
    h:\Lambda^\alpha \times\prod_{i=1}^\alpha(D_i^\vee\times\mathbb{C}^\times) \rightarrow U,
    \]
    defined by
    \[
    h((x'',y''),(d_i,t_i)_{i=1}^\alpha):=(g(w)^{-1}(x'',y''), w_\alpha).
    \]
    It follows that $h\circ f =\mathrm{id}$ and $f\circ h =\mathrm{id}$. Indeed, we note that since $y''_il_{i+1}=l_i$, setting $(x',y'):=g(w)^{-1}(x'',y'')$,   we have 
    \[
    y_i'w_{i+1}=g_i(w)^{-1}y_i''g_{i+1}(w)w_{i+1}=g_i(w)^{-1}l_i=w_i.
    \]
    Therefore $f$ is an isomorphism.
\end{proof}

The above two lemmas give
\begin{equation}\label{eqn U iso}
U\cong \Lambda_{\lambda_{\phi^\vee}}\times\mathbb{C}^t\times\prod_{i=1}^\alpha(D_i^\vee\times\mathbb{C}^\times).
\end{equation}
Note that both $\mathbb{C}^t$ and $\prod_{i=1}^\alpha(D_i^\vee\times\mathbb{C}^\times)$ are irreducible varieties.

Set $(x,y):=\epsilon'((x_{\phi^\vee},y_{\phi^\vee}),(x_{\phi^\alpha},y_{\phi^\alpha}))=(x_{\phi^\vee}\oplus 0,y_{\phi^\vee}\oplus (l_{i+1}\mapsto l_i)_{i=1}^{\alpha-1})$ and let $p=(x,y,l_\alpha)\in\Lambda_{\lambda_\alpha}\times E_\alpha$. Then $p\in U$ (since each $t_i=1$) and $f(p)=((x_{\phi^\vee},y_{\phi^\vee},0),(0,1)_{i=1}^\alpha).$ Since $U$ is open in $\Lambda_{\lambda_\alpha}\times E_\alpha$, the number of irreducible components of $\Lambda_{\lambda_\alpha}\times E_\alpha$ which contain $p$ is the same as the number of irreducible components of $U$ which contain $p.$ Furthermore, any irreducible component of $\Lambda_{\lambda_\alpha}\times E_\alpha$ is of the form $Z\times E_\alpha$ for some irreducible component $Z$ of $\Lambda_{\lambda_\alpha}$.

On the other hand, by Equation \eqref{eqn U iso} and the observation that both $\mathbb{C}^t$ and $\prod_{i=1}^\alpha(D_i^\vee\times\mathbb{C}^\times)$ are irreducible varieties, we obtain that the number of irreducible components of $U$ containing $p$ is the same as the number of irreducible components of $\Lambda_{\lambda_{\phi^\vee}}$ which contain $(x_{\phi^\vee},y_{\phi^\vee})$. However, $(x_{\phi^\vee},y_{\phi^\vee})\in\Lambda_{C_\phi^\vee}^{reg}$ by assumption and therefore lies in exactly one irreducible component. Consequently, $(x,y)$ lies in a unique irreducible component of $\Lambda_{\lambda_\alpha}$. 

Next we recall the notion of a $H_{\lambda_\alpha}$-component in $\Lambda_{\lambda_{\alpha}}$. 

\begin{defn}
    A subset $X\subseteq \Lambda_{\lambda_{\alpha}}$ is called an $H_{\lambda_\alpha}$-component if $X$ is a minimal $H_{\lambda_\alpha}$-invariant union of irreducible components.
\end{defn}
 
Adams, Barbasch, and Vogan provide a description of the $H_{\lambda_\alpha}$-components in $\Lambda_{\lambda_{\alpha}}$.

\begin{lemma}[{\cite[Lemma 19.2(b)]{ABV92}}]\label{lemma ABV 19.2b}
    Let $\lambda$ be an infinitesimal parameter of $G$ and $C$ be an orbit of the Vogan variety $V_\lambda.$ Then 
    \begin{enumerate}
        \item $\dim \Lambda_C=\dim V_\lambda,$
        \item $\Lambda_C$ is $H_\lambda$-irreducible, i.e., $H_\lambda$ permutes the irreducible components of $\Lambda_C$ transitively, and
        \item the $H_{\lambda}$-components of $\Lambda_\lambda$ are the closures $\overline{\Lambda}_{C'}$ where $C'\in C_\lambda(G)$. 
    \end{enumerate}
\end{lemma}

Now, we have that $(x,y)\in\overline{\Lambda}_{C_{\phi_\alpha}}$. Since $(x,y)$ lies in a unique irreducible component, we obtain $(x,y)\not\in\overline{\Lambda}_{C'}$ for any orbit $C'\neq C_{\phi_\alpha}$ from Lemma \ref{lemma ABV 19.2b}. Indeed, suppose that $(x,y)\in \overline{\Lambda}_{C'}$. Since $(x,y)$ lies uniquely in the irreducible component $Z$ of $\Lambda_{\lambda_\alpha}$, by Lemma \ref{lemma ABV 19.2b}(3), we obtain that $Z\subseteq\overline{\Lambda}_{C_{\phi_\alpha}}\cap \overline{\Lambda}_{C'}$. Consider $W=\cup_{h\in H_{\lambda_\alpha}}hZ$ which is an $H_{\lambda_\alpha}$-invariant union of irreducible components of $\Lambda_{\lambda_\alpha}$. We have that $W\subseteq \overline{\Lambda}_{C_{\phi_\alpha}}\cap \overline{\Lambda}_{C'}$. But since $\overline{\Lambda}_{C_{\phi_\alpha}}$ and $ \overline{\Lambda}_{C'}$ are minimal $H_{\lambda_\alpha}$-invariant unions of irreducible components, we obtain that $W=\overline{\Lambda}_{C_{\phi_\alpha}}= \overline{\Lambda}_{C'}$. Therefore, we have that $C_{\phi_\alpha}=C'.$

Consequently, we obtain
\[
(x,y)=\epsilon'((x_{\phi^\vee},y_{\phi^\vee}),(x_{\phi^\alpha},y_{\phi^\alpha}))\in \Lambda_{C_{\phi_\alpha}}^{reg}
\]
which completes the proof of Proposition \ref{prop regular conorm}. 

We remark that the hypothesis $\beta=\frac{\alpha-1}{2} \geq m^\beta_{\mathbbm{1}_{W_F}}(\phi^\vee)$ is required to ensure that we have a clean, uniform decomposition $E_i=D_i^\vee\oplus L_i$ for any $i$.

\subsection{Verification of Theorem \ref{thm GL main thm intro}}

Proposition \ref{prop regular conorm} gives the following fixed point formula. 

\begin{thm}\label{thm GL fixed point formula}
    Suppose that $\beta=\frac{\alpha-1}{2} \geq m^\beta_{\mathbbm{1}_{W_F}}(\phi^\vee)$. Then for any $[\mathcal{F}]\in K\Per_{\lambda_\alpha}(G_m),$ we have
    \[
    \langle\eta_{\phi_\alpha},[\mathcal{F}]\rangle_{\lambda_\alpha}=\langle\eta_{\phi^\times},[\mathcal{F}|_{V_{\lambda^\times}}]\rangle_{\lambda^\times}.
    \]
\end{thm}

\begin{proof}
We defer the proof of the above theorem to Theorem \ref{thm fpf appendix}.
\end{proof}

We remark that the above theorem is a generalization of the fixed point formula established by Cunningham and Ray in \cite[Proposition 3.2]{CR26}. This generalization is nontrivial and requires further technical discussion which we defer to Appendix \ref{appenedix fpf} in order to not distract from our goal of investigating the Adams conjecture for $\ABV$-packets.

At this point, we have reduced the verification of Conjecture \ref{conj Adams ABV GL refined}(1) to showing that $\langle\eta_{\phi^\times},[\mathcal{F}|_{V_{\lambda^\times}}]\rangle_{\lambda^\times}\neq 0.$ By definition, we have $\langle\eta_{\phi^\times},[\mathcal{F}|_{V_{\lambda^\times}}]\rangle_{\lambda^\times}=\langle\Lift_{G^\times}^{G_m}(\eta_{\phi^\times}),[\mathcal{F}]\rangle_{\lambda_\alpha}$ and so, by Proposition \ref{prop lift}, it suffices to show that $\mathrm{Ind}_{P}^{G_m}(\pi^\vee\boxtimes 1_{\GL_{\alpha}(F)})=\theta_{-\alpha}({\pi})$, where $P$ is the standard parabolic subgroup of $G_m$ with Levi subgroup isomorphic to $G^\times=G_n\times G_\alpha$ and $1_{\GL_{\alpha}(F)}$ denotes the trivial representation of ${\GL_{\alpha}(F)}$. We verify this for our setting next.

\begin{lemma}\label{lemma ind is irred for alpha large}
    Suppose that $\beta=\frac{\alpha-1}{2} \geq m^\beta_{\mathbbm{1}_{W_F}}(\phi^\vee)$. Then $\theta_{-\alpha}(\pi)=\pi^\vee\times 1_{\GL_{\alpha}(F)}$ for any $\pi^\vee\in\Pi_{\phi^\vee}^\ABV$.
\end{lemma}

\begin{proof}
    For the proof, we use another version of the Langlands classification due to Zelevinsky (\cite{Zel80}). Let $[a,b]_\rho$ denote a segment. Then 
    \[
    \rho|\cdot|^b\times\rho|\cdot|^{b+1}\times\cdots\times\rho|\cdot|^a
    \]
    has a unique irreducible subrepresentation which we denote by $Z_\rho[b,a]$. Furthermore, if $a_1+b_1\geq\cdots\geq a_r+b_r$, then
    \[
    Z_{\rho_1}[b_1,a_1]\times\cdots\times Z_{\rho_r}[b_r,a_r]
    \]
    has a unique irreducible subrepresentation, denoted $Z(Z_{\rho_1}[b_1,a_1],\dots, Z_{\rho_r}[b_r,a_r])$. Moreover, for any $\pi\in\Pi(G_n)$, we have $\pi=Z(Z_{\rho_1}[b_1,a_1],\dots, Z_{\rho_r}[b_r,a_r])$ for some finite collection of segments (\cite[Theorem 6.1]{Zel80}).
    
    In particular, the data for trivial representation is given by $1_{\GL_\alpha(F)}=Z_{1_{\GL_1(F)}}[\frac{1-\alpha}{2},\frac{\alpha-1}{2}]$.
    Suppose that $\pi^\vee=Z(Z_{\rho_1}[b_1,a_1],\dots, Z_{\rho_r}[b_r,a_r]).$ 
    Since $\beta=\frac{\alpha-1}{2} \geq m^\beta_{\mathbbm{1}_{W_F}}(\phi^\vee)$, the segment $[\frac{\alpha-1}{2},\frac{1-\alpha}{2}]_{1_{\GL_1(F)}}$ does not link with any of the segments $[a_1,b_1]_{\rho_1},\dots,[a_r,b_r]_{\rho_r}$. From \cite[Proposition 8.5]{Zel80}, we obtain that $\pi^\vee\times 1_{\GL_{\alpha}(F)}$ is irreducible.

     Now suppose that $\pi^\vee=L(\Delta_{\rho_1'}[x_1,y_1]\times\cdots\times\Delta_{\rho_s'}[x_s,y_s]).$
    By Theorem \ref{thm theta GL}, we have that \[\theta_{-\alpha}(\pi)=L(|\cdot|^{\frac{\alpha-1}{2}},\dots,|\cdot|^{\frac{1-\alpha}{2}},\Delta_{\rho_1'}[x_1,y_1],\dots,\Delta_{\rho_s'}[x_s,y_s]).\] 
    Since $\pi^\vee\times 1_{\GL_{\alpha}(F)}$ is irreducible and 
    \[1_{\GL_{\alpha}(F)}=L(\Delta_{1_{\GL_1(F)}}[\frac{\alpha-1}{2},\frac{\alpha-1}{2}],\dots,\Delta_{1_{\GL_1(F)}}[\frac{1-\alpha}{2},\frac{1-\alpha}{2}]),\] 
    it follows from \cite[Proposition 3.5(5)]{LM16} that $\theta_{-\alpha}(\pi)=\pi^\vee\times1_{\GL_{\alpha}(F)}$ which proves the lemma.
\end{proof}

We note that the hypothesis $\beta=\frac{\alpha-1}{2} \geq m^\beta_{\mathbbm{1}_{W_F}}(\phi^\vee)$ is again crucial in the above argument. For $\alpha$ small, the segment $[\frac{\alpha-1}{2},\frac{1-\alpha}{2}]_{1_{\GL_1(F)}}$ may link with one of the segments $[a_1,b_1]_{\rho_1},\dots,[a_r,b_r]_{\rho_r}$ and so $\pi^\vee\times 1_{\GL_{\alpha}(F)}$ may be reducible. This indeed occurs for Example \ref{exmp GL fails 1}.

With the preparations complete, we now prove our main result, i.e., we verify Conjecture \ref{conj Adams ABV GL refined}(1).

\begin{thm}\label{thm GL main thm}
    Suppose $\pi\in\Pi_{\phi}^\ABV$ and $\beta=\frac{\alpha-1}{2} \geq m^\beta_{\mathbbm{1}_{W_F}}(\phi^\vee)$. Then $\theta_{-\alpha}(\pi)\in\Pi_{\phi_\alpha}^\ABV.$
\end{thm}

\begin{proof}
    By Vogan's perspective on the local Langlands correspondence, we have that $\pi$ corresponds to the perverse sheaf $\IC(\mathbbm{1}_{C_{\phi_\pi}}).$ Similarly, from Theorem \ref{thm theta GL}, we have that $\theta_{-\alpha}(\pi)$ corresponds to the perverse sheaf $\IC(\mathbbm{1}_{C_{(\phi_\pi)_\alpha}}).$ Let $\lambda_\alpha=\lambda_{\phi_\alpha}.$ By Lemma \ref{lemma in ABV}, it is sufficient to show that
    \[
    \langle \eta_{\phi_\alpha}, [\IC(\mathbbm{1}_{C_{(\phi_\pi)_\alpha}})]\rangle_{\lambda_\alpha}\neq 0.
    \]
    For brevity, write $\mathcal{F}=\IC(\mathbbm{1}_{C_{(\phi_\pi)_\alpha}})$.
    From Theorem \ref{thm GL fixed point formula},  we obtain that 
   \[
    \langle \eta_{\phi_\alpha}, [\mathcal{F}]\rangle_{\lambda_\alpha}= \langle\eta_{\phi^\times},[\mathcal{F}|_{V_{\lambda^\times}}]\rangle_{\lambda^\times}=\langle\Lift_{G^\times}^{G_m}(\eta_{\phi^\times}),[\mathcal{F}]\rangle_{\lambda_\alpha},
    \]
    where $G^\times=G_n\times G_\alpha$.
    Using Lemma \ref{lemma contragredient ABV} (and using that the Vogan variety decomposes as a product of Vogan varieties), write
    \[
    \eta_{\phi^\times}=\sum_{\tau\in\Pi_{\phi}^\ABV} a_\tau[ \tau^\vee\boxtimes1_{\GL_\alpha(F)}],
    \]
    where $a_\tau\neq 0.$
    By Proposition \ref{prop lift}, we have
    \[
    \langle\Lift_{G^\times}^{G_m}(\eta_{\phi^\times}),[\mathcal{F}]\rangle_{\lambda_\alpha}=\langle\sum_{\tau\in\Pi_{\phi}^\ABV}a_\tau[\mathrm{Ind}_{P}^{G_m}( \tau^\vee\boxtimes1_{\GL_\alpha(F)})],[\mathcal{F}]\rangle_{\lambda_\alpha},
    \]
    where $P$ denotes the standard parabolic subgroup of $G_m$ with Levi subgroup isomorphic to $G^\times$.
    By Lemma \ref{lemma ind is irred for alpha large}, we obtain that $\mathrm{Ind}_{P}^{G_m} (\tau^\vee\boxtimes1_{\GL_\alpha(F)})=\theta_{-\alpha}(\tau)$ for any $\tau\in\Pi_{\phi}^\ABV$. By definition and Corollary \ref{cor theta Lpar}, we have 
    \[
    \langle\sum_{\tau\in\Pi_{\phi}^\ABV}a_\tau[\mathrm{Ind}_{P}^{G_m}( \tau^\vee\boxtimes1_{\GL_\alpha(F)})],[\mathcal{F}]\rangle_{\lambda_\alpha}=a_\pi(-1)^{d(\theta_{-\alpha}(\pi))}\neq 0.
    \]
    Therefore, we obtain that $\langle \eta_{\phi_\alpha}, [\IC(\mathbbm{1}_{C_{(\phi_\pi)_\alpha}})]\rangle_{\lambda_\alpha}\neq 0$ from which the theorem follows.
\end{proof}

We remark on some consequences of Theorem \ref{thm GL main thm}. First, in \cite[\S0B]{CFK22}, it is claimed that there exists a nonsingleton ABV-packet of $\GL_n(F)$ for any $n\geq 16.$ For $\GL_{16}(F)$ this is proved in \cite[Corollary 2.7]{CFK22}, but for $n\geq 17,$ no proof is explicitly given. Their outline is to simply construct a Vogan variety which is isomorphic to the Vogan variety of the nonsingleton ABV-packet of $\GL_{16}(F)$. However, we are able to obtain more complicated examples using Theorem \ref{thm GL main thm}.

\begin{cor}\label{cor nonsingleton}
    There exists a nonsingleton ABV-packet of $\GL_n(F)$ for $n=16, 18, 20$ or any $n\geq 21.$
\end{cor}

\begin{proof}
    Let $\phi_{\KS}$ be the $L$-parameter of $\GL_{16}(F)$ described in \cite[\S1B]{CFK22}. We have that $\phi_{\KS}^\vee=\phi_{\KS}$ and
    \[
    \exp_{\mathbbm{1}_{W_F}}(\phi_{\KS}^\vee)=\{2,1,0,-1,-2\}.
    \]
    Note that $\frac{x-1}{2}=2$ implies that $x=5.$ Thus we have that $\beta=\frac{\alpha-1}{2}\geq m_{\mathbbm{1}_{W_F}}^\beta(\phi_{\KS}^\vee)$ if and only if $\alpha=2,4$ or $\alpha\geq 5.$ By \cite[Corollary 2.7]{CFK22}, $\Pi_{\phi_{\KS}}^\ABV$ consists of two representations, say $\Pi_{\phi_{\KS}}^\ABV=\{\pi_1,\pi_2\}.$ By Theorem \ref{thm GL main thm}, for any $i=1,2$, we have $\theta_{-\alpha}(\pi_i)\in \Pi_{(\phi_{\KS})_\alpha}^\ABV$  for any $\alpha=2,4$ or $\alpha\geq 5.$ Therefore, there exists a nonsingleton ABV-packet of $\GL_n(F)$ for $n=16,  18, 20$ or any $n\geq 21.$ We note that the Vogan variety for $(\phi_\KS)_\alpha$ is not isomorphic to that of $\phi_\KS$.
\end{proof}

A second consequence is partial evidence for Conjecture \ref{conj Adams ABV GL refined}(2).

\begin{cor}
     Suppose $\pi\in\Pi_{\phi}^\ABV.$ Assume that $\beta=\frac{\alpha-1}{2} \geq m^\beta_{\mathbbm{1}_{W_F}}(\phi^\vee)$. Then the following hold.
     \begin{enumerate}
         \item If $\theta_{-\alpha}(\pi)\in\Pi_{\phi_\alpha}^\ABV,$ then $\theta_{-(\alpha+2)}(\pi)\in\Pi_{\phi_{\alpha+2}}^\ABV.$
         \item If $\beta+\frac{1}{2}=\frac{\alpha}{2} \geq m^{\beta+\frac{1}{2}}_{\mathbbm{1}_{W_F}}(\phi^\vee)$ and $\theta_{-\alpha}(\pi)\in\Pi_{\phi_\alpha}^\ABV,$ then $\theta_{-(\alpha+1)}(\pi)\in\Pi_{\phi_{\alpha+1}}^\ABV.$
     \end{enumerate}
\end{cor}

\begin{proof}
Both parts are immediate consequences of Theorem \ref{thm GL main thm}. We remark that the requirement $\beta+\frac{1}{2}=\frac{\alpha}{2} \geq m^{\beta+\frac{1}{2}}_{\mathbbm{1}_{W_F}}(\phi^\vee)$ is needed in Part (2) as the condition $\beta=\frac{\alpha-1}{2} \geq m^\beta_{\mathbbm{1}_{W_F}}(\phi^\vee)$ does not necessarily imply that $\beta+\frac{1}{2}=\frac{\alpha}{2} \geq m^{\beta+\frac{1}{2}}_{\mathbbm{1}_{W_F}}(\phi^\vee)$.
\end{proof}

Of course, the above corollary is not saying anything substantial as Conjecture \ref{conj Adams ABV GL refined}(1) (which is Theorem \ref{thm GL main thm}) implies Conjecture \ref{conj Adams ABV GL refined}(2) for $\alpha\gg 0.$ However, the above corollary explicates the condition $\alpha\gg 0$ from Theorem \ref{thm GL main thm}.

Next, we give some remarks on how the proof strategy behind Theorem \ref{thm GL main thm} breaks down when $\alpha$ is small, i.e., we explain why our argument does not verify Conjecture \ref{conj Adams ABV GL naive} for small $\alpha$. We already remarked after the proof of Lemma \ref{lemma ind is irred for alpha large} that this lemma can and does fail if $\beta=\frac{\alpha-1}{2} < m^\beta_{\mathbbm{1}_{W_F}}(\phi^\vee)$. The lifting of regular conormal elements may also fail in this range as we explain below.

Recall that given a general $L$-parameter $\phi$ corresponding to the orbit $C_\phi,$ we attach a dual $L$-parameter $\hat{\phi}$ corresponding to the dual orbit $(C_{\phi})^*.$ By \cite[Lemma 6.4.2]{CFMMX22}, we have $\Lambda_{C_{\phi}}^{reg}\subseteq C_{\phi} \times C_{\phi}^*.$ Recall that Proposition \ref{prop regular conorm} asserts that if $\beta=\frac{\alpha-1}{2} \geq m^\beta_{\mathbbm{1}_{W_F}}(\phi^\vee)$ and \[((x_{\phi^\vee},y_{\phi^\vee}),(x_{\phi^\alpha},y_{\phi^\alpha}))\in \Lambda_{C_{\phi^\times}}^{reg},\]
then
    \[\varepsilon'((x_{\phi^\vee},y_{\phi^\vee}),(x_{\phi^\alpha},y_{\phi^\alpha}))\in\Lambda_{C_{\phi_\alpha}}^{reg}.\]
For this to be true more generally, we must have that $C_{\phi^\times}^*=C_{\widehat{\phi^\times}}\subseteq C_{\phi_\alpha}^*=C_{\widehat{\phi_\alpha}}.$ This condition can fail if $\beta=\frac{\alpha-1}{2} < m^\beta_{\mathbbm{1}_{W_F}}(\phi^\vee)$ as we show in Example \ref{exmp GL fails 3}.

In general, the computation of $\hat{\phi}$ for general linear groups is determined by the M{\oe}glin-Waldspurger algorithm (\cite[Theoreme II.13]{MW86}).

\begin{lemma}\label{lemma dual of phi^alpha}
    We have $\widehat{\phi^\alpha}=\mathbbm{1}_{W_F}\otimes S_\alpha.$
\end{lemma}

\begin{proof}
    This follows simply from the M{\oe}glin-Waldspurger algorithm (\cite[Theoreme II.13]{MW86}). Alternatively, $\phi^\alpha$ corresponds to the $0$-orbit in $V_{\lambda^\alpha}.$ Its dual orbit is the unique open orbit in $V_{\lambda^\alpha}.$ This orbit corresponds to $\mathbbm{1}_{W_F}\otimes S_\alpha$ from which the lemma follows.
\end{proof}

Let $\beta=\frac{\alpha-1}{2}.$ We show that if $\beta\geq m^\beta_{\mathbbm{1}_{W_F}}(\phi^\vee),$ then $\widehat{\phi}_\alpha=\widehat{\phi^\vee}+\widehat{\phi^\alpha}$. We note that this also follows from Proposition \ref{prop regular conorm} and \cite[Lemma 6.4.2]{CFMMX22}.

\begin{lemma}\label{lemma dual of phi_alpha}
    Assume that $\beta=\frac{\alpha-1}{2} \geq m^\beta_{\mathbbm{1}_{W_F}}(\phi^\vee)$. Then $\widehat{\phi}_\alpha=\widehat{\phi^\vee}+\widehat{\phi^\alpha}$.
\end{lemma}

\begin{proof}
The proof is a direct consequence of the M{\oe}glin-Waldspurger algorithm (\cite[Theoreme II.13]{MW86}). Indeed, since $\beta \geq m^\beta_{\mathbbm{1}_{W_F}}(\phi^\vee)$, the first iteration of the algorithm groups the segments $[\frac{\alpha-1}{2},\frac{\alpha-1}{2}]_{1_{\GL_1(F)}},\dots,[\frac{1-\alpha}{2},\frac{1-\alpha}{2}]_{1_{\GL_1(F)}}$ into a segment $[\frac{\alpha-1}{2},\frac{1-\alpha}{2}]_{1_{\GL_1(F)}}$ which corresponds to $\widehat{\phi^\alpha}=\mathbbm{1}_{W_F}\otimes S_\alpha$ by Lemma \ref{lemma dual of phi^alpha}. The algorithm then repeats on the rest of the segments and hence computes $\widehat{\phi^\vee}.$ Therefore, $\widehat{\phi}_\alpha=\widehat{\phi^\vee}+\widehat{\phi^\alpha}.$
\end{proof}

We recall the situation for our earlier examples.

\begin{exmp}\label{exmp GL works 3}
    We continue Example \ref{exmp GL works 2}. In this case, $\widehat{\phi^\vee}=\mathbbm{1}_{W_F}\otimes S_2=\widehat{\phi^\alpha}$ and
    \[
    \widehat{\phi_\alpha}=\mathbbm{1}_{W_F}\otimes S_2+\mathbbm{1}_{W_F}\otimes S_2=\widehat{\phi^\vee}+\widehat{\phi^\alpha}
    \]
    as stated by Lemma \ref{lemma dual of phi_alpha}.
\end{exmp}

\begin{exmp}\label{exmp GL fails 3}
     We continue Example \ref{exmp GL fails 2}. In this case, $\widehat{\phi^\vee}=\phi^\vee=|\cdot|^{\frac{3}{2}}+|\cdot|^{-\frac{3}{2}}$ and $\mathbbm{1}_{W_F}\otimes S_2=\widehat{\phi^\alpha}$. By the M{\oe}glin-Waldspurger algorithm (\cite[Theoreme II.13]{MW86}), we obtain that
    \[
    \widehat{\phi_\alpha}=\mathbbm{1}_{W_F}\otimes S_4\neq \widehat{\phi^\vee}+\widehat{\phi^\alpha}.
    \]
    Indeed, $\phi_\alpha=|\cdot|^{\frac{3}{2}}+|\cdot|^{\frac{1}{2}}+|\cdot|^{-\frac{1}{2}}+|\cdot|^{-\frac{3}{2}}$ and the M{\oe}glin-Waldspurger algorithm (\cite[Theoreme II.13]{MW86}) groups the ``segments''  $[\frac{3}{2},\frac{3}{2}]_{1_{\GL_1(F)}},\dots,[\frac{-3}{2},\frac{-3}{2}]_{1_{\GL_1(F)}}$ into one segment $[\frac{3}{2},\frac{-3}{2}]_{1_{\GL_1(F)}}$ which corresponds to $\mathbbm{1}_{W_F}\otimes S_4.$ 

    In other words, Lemma \ref{lemma dual of phi_alpha} fails for this example. The reason is that a trivial exponent of $\phi^\alpha$ interacted (meaning it linked with a segment) with a trivial exponent of $\phi^\vee$ in the M{\oe}glin-Waldspurger algorithm (\cite[Theoreme II.13]{MW86}). This is why we require $\beta\geq m^\beta_{\mathbbm{1}_{W_F}}(\phi^\vee)$ in Lemma \ref{lemma dual of phi_alpha}, so that there is no interaction between the trivial exponents of $\phi^\alpha$ and $\phi^\vee.$

    Note that despite the failure of Lemma \ref{lemma dual of phi_alpha}, the Adams conjecture (Conjecture \ref{conj Adams ABV GL naive}) still holds for this example. Indeed, it is simple to check that the ABV-packets are singletons and hence agree with their $L$-packets. The Adams conjecture then follows for this example from Theorem \ref{thm theta GL}.
\end{exmp}

Note that to adapt the proof strategy from general linear groups to reductive dual pairs of type I, an analogous statement to Lemma \ref{lemma dual of phi_alpha} would need to hold. For symplectic-even orthogonal dual pairs, one could verify the analogous statement (for $\alpha$ sufficiently large) using the algorithm given in \cite{HL26}.

Of course, we remark that while our proof strategy for Conjecture \ref{conj Adams ABV GL naive} breaks down when $\beta=\frac{\alpha-1}{2} < m^\beta_{\mathbbm{1}_{W_F}}(\phi^\vee)$, this does not necessarily mean that Conjecture \ref{conj Adams ABV GL naive} is false.

\appendix

\section{Proof of the fixed point formula}\label{appenedix fpf}

The goal of this appendix is to prove the fixed point formula (Theorem \ref{thm GL fixed point formula}). For brevity, we let $G_n=\GL_n(F)$ throughout this appendix. The fixed point formula that we verify is a generalization of a fixed point formula for local Arthur parameters of $G_n$ given by Cunningham and Ray in \cite[Proposition 4.6]{CR26}. Their results are stated in terms of local Arthur parameters for which the geometry is significantly simpler. For example, both the regular and microlocal fundamental groups (recalled below) are trivial in their situation. In our situation, this is not guaranteed and is known to fail in our situation. Indeed, while we show that the microlocal fundamental group is always trivial in Lemma \ref{lemma microlocal fundamental group is trivial}, the regular fundamental group can be nontrivial as shown in \cite{CFK22}.

We recall some notation from \S\ref{sec perverse sheaves}.
Fix an infinitesimal parameter $\lambda$ of $G_n$ and let $\phi\in\Phi_\lambda(G_n).$ We let $C_\phi$ be the corresponding $H_\lambda$-orbit for $\phi$ in $V_\lambda$. We consider the pairing $\langle\cdot,\cdot\rangle:K\Pi_\lambda(G_n)\times K\Per_\lambda(G_n)\rightarrow\Z$ defined in Equation \eqref{eqn pairing}.
We further consider 
\[
\eta_\phi:=\eta_{C_{\phi}}=(-1)^{d(C_\phi)}\sum_{\pi\in\Pi_\lambda(G_n)}(-1)^{d(\pi)}\rank(\Evs_{C_\phi}(\mathcal{P}(\pi))) [\pi] \in K\Pi_\lambda(G_n)
\]
Here rank means the Euler characteristic.

We begin by generalizing \cite[Proposition 1.6]{CR26}. We remark that Cunningham and Ray's argument generalizes to any local $L$-parameter of $G_n.$ For the sake of completeness, we provide the proof. Recall that $D_{H_\lambda}(V_\lambda)$ denotes the $H_\lambda$-equivariant derived category of $\ell$-adic sheaves on $V_\lambda.$

\begin{prop}\label{prop perfect pair append}
    For any $\mathcal{F}\in D_{H_\lambda}(V_\lambda),$ we have
    \[
    \langle\eta_\phi,\mathcal{F}\rangle=(-1)^{d(C_\phi)}\rank(\Evs_{C_\phi}\mathcal{F}).
    \]
\end{prop}

\begin{proof}
    As in the proof of \cite[Proposition 1.6]{CR26}, the Grothendieck groups $K\Per_\lambda(V_\lambda)$ and $KD_\lambda(V_\lambda)$ coincide and so it is enough to prove the proposition for simple objects in $\Per_\lambda(V_\lambda)$.

    Let $\mathcal{F}\in \Per_\lambda(V_\lambda)$ be simple. From Vogan's perspective on the local Langlands correspondence (\cite{Vog93}; see \S\ref{sec perverse sheaves}), we have that $\mathcal{F}=\mathcal{P}(\pi')$ for some $\pi'\in\Pi_\lambda(G_n).$ We obtain that
    \begin{align*}
    \langle\eta_\phi,\mathcal{F}\rangle&=\langle(-1)^{d(C_\phi)}\sum_{\pi\in\Pi_\lambda(G_n)}(-1)^{d(\pi)}\rank(\Evs_{C_\phi}(\mathcal{P}(\pi))) [\pi],[\mathcal{P}(\pi')]\rangle \\
    &=(-1)^{d({C_\phi})}\sum_{\pi\in\Pi_\lambda(G_n)}(-1)^{d(\pi)}\rank(\Evs_{C_\phi}(\mathcal{P}(\pi)))\langle [\pi],[\mathcal{P}(\pi')]\rangle \\
    &=(-1)^{d({C_\phi})}\rank(\Evs_{C_\phi}(\mathcal{F})).
    \end{align*}
    where the last equality follows from the definition of $\langle\cdot,\cdot\rangle.$
\end{proof}

Next we recall the definition of the (equivariant) microlocal fundamental group from \cite[Definition 1.33]{ABV92}. Fix an orbit $C_\phi$ of $V_\lambda.$ Fix a choice of $y$ such that $(x',y)\in\Lambda_{C_\phi}$ for some $x'\in C_\phi$ and consider $\Lambda_{C_\phi, y}=\{x\in C_\phi \ | \ [x,y]=0\}.$ Given any $x\in \Lambda_{C_\phi, y}$, we consider the centralizer $Z_{H_\lambda}(x,y)$ and set $A_{y,x}=Z_{H_\lambda}(x,y) / Z_{H_\lambda}(x,y)^0.$ By \cite[Lemma 24.3]{ABV92}, this family is locally constant over most of $\Lambda_{C_\phi, y}.$ The (equivariant) microlocal fundamental group is defined to be $A_{C_\phi}^{mic}:= Z_{H_\lambda}(x,y) / Z_{H_\lambda}(x,y)^0=\pi_0(Z_{H_\lambda}(x,y))$ for generic $x\in \Lambda_{C_\phi, y}.$

Now, we consider the functor
\[
\Evs_{C_\phi}: \Per_{H_\lambda}(V_\lambda)\rightarrow\mathrm{Loc}_{H_\lambda}(\Lambda_{C_\phi}^{reg}),
\]
where $\mathrm{Loc}_{H_\lambda}(\Lambda_{C_\phi}^{reg})$ denotes the category of $H_\lambda$-equivariant local systems on $\Lambda_{C_\phi}^{reg}$. We note that the definition given in \cite{CFMMX22} uses a subset denoted by $\Lambda_{C_\phi}^{gen}$; however, we have that $\Lambda_{C_\phi}^{reg}=\Lambda_{C_\phi}^{gen}$ by \cite[Remark 11.3]{CHLLRXS26}.
The regular (equivariant) fundamental group is $A_{C_\phi}^{reg}:=\pi_{1,H_\lambda}(\Lambda_{C_\phi}^{reg},(x,y))$, where $(x,y)\in \Lambda_{C_\phi}^{reg}$ is a choice of base point. We have
\[
\mathrm{Loc}_{H_\lambda}(\Lambda_{C_\phi}^{reg})\cong \mathrm{Rep}(A_{C_\phi}^{reg}).
\]
We warn the reader that the isomorphism for $\mathrm{Loc}_{H_\lambda}(\Lambda_{C_\phi}^{reg})$ is incorrectly stated in \cite[\S8.4]{CFMMX22}. Given $(x,y)\in \Lambda_{C_\phi}^{reg},$ let $\mathcal{O}_{H_\lambda}(x,y)$ denote its corresponding $H_\lambda$-orbit. We have
\[
\mathrm{Loc}_{H_\lambda}(\mathcal{O}_{H_\lambda}(x,y))\cong \mathrm{Rep}(A_{y,x}).
\]
Fix $(x,y)\in \Lambda_{C_\phi}^{reg}.$
Given $s\in Z_{H_\lambda}(x,y),$ we let $a_s\in A_{y,x}$ denote its image. The restriction map $\mathrm{Loc}_{H_\lambda}(\Lambda_{C_\phi}^{reg})\rightarrow \mathrm{Loc}_{H_\lambda}(\mathcal{O}_{H_\lambda}(x,y))$ induces a map $\mathrm{Rep}(A^{reg}_\phi)\rightarrow \mathrm{Rep}(A_{C_\phi}^{mic})$.

Given $s\in Z_{H_\lambda}(x,y)$ for $(x,y)\in \Lambda_{C_\phi}^{reg}$, we let $a_s\in A_{C_\phi}^{mic}$ denote its image. We consider the distribution
\begin{equation}\label{eqn eta dist general}
\eta_{\phi,s}:=(-1)^{d(C)}\sum_{\pi\in\Pi_\lambda(G_n)}(-1)^{d(\pi)}
\mathrm{trace}(a_s,\Evs_C(\mathcal{P}(\pi))) [\pi].
\end{equation}
Here, we have identified $\Evs_C(\mathcal{P}(\pi))$ with a representation of $A_{C_\phi}^{mic}$ via $\mathrm{Loc}_{H_\lambda}(\Lambda_{C_\phi}^{reg})\rightarrow \mathrm{Loc}_{H_\lambda}(\mathcal{O}_{H_\lambda}(x,y))\cong \mathrm{Rep}(A_{C_\phi}^{mic})$, where these are the maps discussed above.  See \cite[\S12]{CHLLRXS26} for further details. We note that the normalization of the $\Evs$ functor is trivial, i.e., the functor $\mathrm{NEvs}$ defined in \cite[\S7.10]{CFMMX22} agrees with $\Evs$.

We verify that the microlocal fundamental group for general linear groups is trivial below, i.e., we always have that $a_s=1$ and hence $\eta_{\phi,s}=\eta_\phi.$

\begin{lemma}\label{lemma microlocal fundamental group is trivial}
    Let $\phi\in\Phi(\GL_n(F))$. Then $A_{C_{\phi}}^{\mic}$ is trivial.
\end{lemma}

\begin{proof}
    Let $x\in V_\lambda$ and $y\in V_\lambda^*$ where $\lambda=\lambda_\phi.$ Then
    \[
    Z_{H_{\lambda}}(x,y)=Z_{\GL_n(\mathbb{C})}(\lambda(W_F)\cup\{x,y\})
    \]
    is a connected algebraic group from which the lemma follows.
\end{proof}

We recall further notation from \S\ref{sec perverse sheaves}. 
For $i=1,\dots,r$, let $\phi_i\in\Phi(G_{n_i})$, and $n=n_1+\dots+n_r.$ We set $G^\times:=G_{n_1}\times\cdots\times G_{n_r}$, $\phi^\times=\phi_{1} \times \cdots\times\phi_r$, and $\lambda^\times=\lambda_1\times\cdots\times\lambda_r$ be the corresponding infinitesimal parameter, where $\lambda_i=\lambda_{\phi_i}$. The Vogan variety is $V_{\lambda^\times}=V_{\lambda_1}\times\cdots\times V_{\lambda_r}$ and we have $H_{\lambda^\times}=H_{\lambda_1}\times\cdots\times H_{\lambda_r}$.

We let $\phi=\phi_1+\dots+\phi_r\in\Phi(G_n)$, $\lambda=\lambda_\phi,$
and $s\in\widehat{G}_n(\C)$ be of finite order (and hence semi-simple) such that $Z_{\widehat{G}_n(\C)}(s)\cong\widehat{G}^\times$. The resulting inclusion $\widehat{G}^\times \hookrightarrow \widehat{G}_n(\C)$ induces inclusions $H_{\lambda^\times}\hookrightarrow H_{\lambda}$ and
\[
\varepsilon: V_{\lambda^\times}\hookrightarrow V_{\lambda}
\]
which is equivariant for the action by $H_{\lambda^\times}.$
We have that 
\[V_{\lambda^\times}=V_{\lambda}^s:=\{x\in V_\lambda \ | \ \Ad(s)x=x\}.
\]
Furthermore, we have an inclusion of the dual Vogan varieties
\[
{}^t\varepsilon: V_{\lambda^\times}^*\hookrightarrow V_{\lambda}^*
\]
and hence an inclusion 
\[
\varepsilon'=\varepsilon\times{}^t\varepsilon: V_{\lambda^\times}\times V_{\lambda^\times}^*\hookrightarrow V_{\lambda}\times  V_{\lambda}^*.
\]

Let $\varepsilon^*:\mathrm{D}_{H_{\lambda}}(V_{\lambda})\rightarrow\mathrm{D}_{H_{\lambda^\times}}(V_{\lambda^\times})$ denote the equivariant restriction functor for the equivariant derived categories. As a shorthand, we write
\[
\mathcal{F}|_{V_{\lambda^\times}}:=\varepsilon^*\mathcal{F}.
\]
Recall from $\S\ref{sec Conormal bundles}$ that for an $H_\lambda$-orbit $C_\phi$ of $V_\lambda,$ we consider its conormal bundle $\Lambda_{C_\phi}^{reg}\subseteq V_\lambda\times V_\lambda^*.$ We have the following generalization of \cite[Lemma 3.1]{CR26}.

\begin{prop}\label{prop ranks agree}
    We continue with the above notation and setting. Suppose further that there exists $(x_{\phi^\times},y_{\phi^\times})\in \Lambda_{C_{\phi^\times}}^{reg}$ such that $\varepsilon'(x_{\phi^\times},y_{\phi^\times})\in \Lambda_{C_\phi}^{reg}$. Then, we have
    \[
    (-1)^{d(C_\phi)}\rank(\Evs_{C_\phi}\mathcal{F})=(-1)^{d(C_{\phi^\times})}\rank(\Evs_{C_\phi^\times}\mathcal{F}|_{V_{\lambda^\times}}),
    \]
    for any $\mathcal{F}\in D_{H_\lambda}(V_\lambda).$
\end{prop}

\begin{proof}
    This is a direct implication of Lemma \ref{lemma microlocal fundamental group is trivial} and \cite[Theorem 2.5]{CHLLRXS26}. We note that the Levi subgroup is an endoscopic group and thus transfer is simply parabolic induction. Furthermore, the Kazhdan-Lusztig hypothesis is known in our cases by \cite{Bal25, Sol25}.
\end{proof}

We remark that the above proposition can also be proved via  a straightforward adaptation of the proof of \cite[Lemma 3.1]{CR26}. However, in that setting, Cunningham and Ray worked with local Arthur parameters which implied that the microlocal fundamental group is trivial. To adapt their proof, we would use Lemma \ref{lemma microlocal fundamental group is trivial} to ensure that the image of $s$ in the microlocal fundamental group is trivial.

We proceed with a technical lemma which describes certain conormal elements.

\begin{lemma}\label{lemma generic covector factors}
    Suppose that $\beta=\frac{\alpha-1}{2} \geq m^\beta_{\mathbbm{1}_{W_F}}(\phi^\vee)$. Let $y_{\lambda^\vee}\in  V_{\lambda^\vee}^*$, $y_{\phi^\alpha}\in C_{\widehat{\phi^\alpha}}\subseteq V_{\lambda^\alpha}^*,$ and consider $y={}^t\varepsilon(y_{\lambda^\vee},y_{\phi^\alpha}).$ Then $x\in\Lambda_{y}$ if and only if $x=\varepsilon(x',0)$ for some $x'\in V_{\lambda^\vee}$ with $[x',y_{\lambda^\vee}]=0$.
\end{lemma}

\begin{proof}
We recall some setup from \S\ref{sec perverse sheaves}.
By \cite[Theorem 5.1.1]{CFMMX22}, we may assume that $\lambda_\alpha$ is unramified, i.e., trivial on $I_F$, and $\chi(\lambda_\alpha(\Fr))\in\mathbb{R}_{>0}$ for any character $\chi:\widehat{T}\rightarrow\GL_1(\BC)$, where $\widehat{T}$ is any torus in $\GL_m(\BC)$ containing $\lambda_\alpha(\Fr).$ We write
\[
\lambda_\alpha=m_1|\cdot|^{e_1}+m_2|\cdot|^{e_2}+\cdots+m_r|\cdot|^{e_r}
\]
where $m_i\in\mathbb{Z}_{\geq 1}$ denotes the multiplicity and $e_i\in\mathbb{R}$ with $e_i>e_{i+1}$ for $i=1,\dots,r-1.$ Furthermore, we may assume that $r=\alpha$ and $e_1=\frac{\alpha-1}{2},e_2=\frac{\alpha-3}{2},\dots,e_\alpha=\frac{1-\alpha}{2}$ since $\beta=\frac{\alpha-1}{2} \geq m^\beta_{\mathbbm{1}_{W_F}}(\phi^\vee)$. Indeed, in general $V_{\lambda_\alpha}$ decomposes as a direct product of Vogan varieties based on the exponents modulo $\mathbb{Z}$ and only those exponents lying in the coset $\frac{\alpha-1}{2}+\mathbb{Z}$ will play a nontrivial role in the following arguments. 

Now, for $i=1,\dots,\alpha,$ let $E_i$ denote the $q_F^{e_i}$-eigenspace of $\lambda_\alpha(\Fr)$. We have $m_i=\dim(E_i)$ and
\[
V_{\lambda_\alpha}\cong\Hom(E_1,E_2)\times\Hom(E_2,E_3)\times\cdots\times\Hom(E_{\alpha-1},E_\alpha).    
\]
Given $x\in V_{\lambda_\alpha},$ using the above isomorphism, we write $x=(x_1,\dots,x_{\alpha-1})$ where
\[
(x_1,x_2,\dots,x_{\alpha-1})\in \Hom(E_1,E_2)\times\Hom(E_2,E_3)\times\cdots\times\Hom(E_{\alpha-1},E_\alpha).
\]

We identify the dual Vogan variety $V_{\lambda_\alpha}^*$ with ${}^t V_{\lambda_\alpha}$ (recalling that $V_{\lambda_\alpha}$ lies in the Lie algebra of $\GL_m(\BC)$, i.e., the spaces of $m\times m$ matrices, the transpose is the usual one). We obtain an isomorphism
\[
V_{\lambda_\alpha}^*\cong\Hom(E_2,E_1)\times\Hom(E_3,E_2)\times\cdots\times\Hom(E_{\alpha},E_{\alpha-1}).
\]
Given $y\in V_{\lambda_\alpha}^*,$ using the above isomorphism, we write $y=(y_1,\dots,y_{\alpha-1})$ where
\[
(y_1,y_2,\dots,y_{\alpha-1})\in \Hom(E_2,E_1)\times\Hom(E_3,E_2)\times\cdots\times\Hom(E_{\alpha},E_{\alpha-1}).
\]

We have similar isomorphisms for $V_{\lambda^\vee}$ and $V_{\lambda_{\phi^\alpha}}$ which we make explicit below. 
For $i=1,\dots,\alpha,$ we let $E_{\lambda^\vee,i}$ be the corresponding $q_F^{e_i}$-eigenspace, possibly zero, of $\lambda^\vee(\Fr).$ We have
\[
V_{\lambda^\vee}\cong \Hom(E_{\lambda^\vee,1},E_{\lambda^\vee,2})\times\cdots\times\Hom(E_{\lambda^\vee,\alpha-1},E_{\lambda^\vee,\alpha}).
\]
Given $x\in V_{\lambda^\vee}$, we identify $x=(x_1,\dots,x_{\alpha-1})$. Note that if  $q_F^{e_i}$ or $q_F^{e_{i+1}}$ is not an eigenvalue $\lambda^\vee(\Fr),$ then $x_i=0.$
The isomorphism for the dual variety is obtained similarly.

For $i=1,\dots,\alpha,$ we have that $q_F^{e_i}$ is always an eigenvalue of $\lambda^{\alpha}(\Fr)$ and the corresponding eigenspace $E_{\lambda^\alpha,i}$ is 1-dimensional. We have 
\[
V_{\lambda^\alpha}\cong \Hom(E_{\lambda^\alpha,1},E_{\lambda^\alpha,2})\times\cdots\times\Hom(E_{\lambda^\alpha,\alpha-1},E_{\lambda^\alpha,\alpha}).
\]
Given $x\in V_{\lambda^\alpha}$, we identify $x=(x_1,\dots,x_{\alpha-1})$. The isomorphism for the dual variety is obtained similarly.
 
Conjugating if necessary, we may choose $s$ such that the inclusion 
$\varepsilon: V_{\lambda^\times}\hookrightarrow V_{\lambda_\alpha}$ is given as follows.
Let $x_{\lambda^\vee}=(x_{\lambda^\vee,1},\dots,x_{\lambda^\vee,\alpha-1})\in V_{\lambda^\vee}$. Also, let $x_{\lambda^\alpha}\in V_{\lambda^\alpha}$ and write $x_{\lambda^\alpha}=(x_{\lambda^\alpha,1},\dots,x_{\lambda^\alpha,\alpha-1})$.
For $i=1,\dots,\alpha-1,$ we define $\varepsilon_i(x_{\lambda^\vee,i},x_{\lambda^\alpha,i})$ as follows 
\[
\varepsilon_i(x_{\lambda^\vee,i},x_{\lambda^\alpha,i}):=\begin{pmatrix}
    x_{\lambda^\vee,i} & 0_{\dim(E_{\lambda^\vee,i})\times 1} \\
    0_{1\times \dim(E_{\lambda^\vee,i+1})} & x_{\lambda^\alpha,i}
\end{pmatrix}.
\]
Note that if $\dim(E_{\lambda^\vee,i})=0,$ then we omit the corresponding rows. Similarly, if $\dim(E_{\lambda^\vee,i+1})=0$, then we omit the corresponding columns. The inclusion is then given by
\[
\varepsilon(x_{\lambda^\vee},x_{\lambda^\alpha})=(\varepsilon_1(x_{\lambda^\vee,1},x_{\lambda^\alpha,1}),\dots,\varepsilon_{\alpha-1}(x_{\lambda^\vee,\alpha-1},x_{\lambda^\alpha,\alpha-1})).
\]
Note that this inclusion corresponds to taking \[s=\mathrm{diag}(I_{\dim{E_{\lambda^\vee,1}}},-1,I_{\dim{E_{\lambda^\vee,2}}},-1,\dots,I_{\dim{E_{\lambda^\vee,\alpha}}},-1),\] where $I_k$ denotes the $k\times k$ identity matrix.

The inclusion of the dual Vogan varieties is given similarly. Let $y_{\lambda^\vee}\in V_{\lambda^\vee}^*$ and write \[y_{\lambda^\vee}=(y_{\lambda^\vee,1},\dots,y_{\lambda^\vee,\alpha-1})\in\Hom(E_{\lambda^\vee,2},E_{\lambda^\vee,1})\times\cdots\times\Hom(E_{\lambda^\vee,\alpha},E_{\lambda^\vee,\alpha-1}).\]  Also, let $y_{\lambda^\alpha}\in V_{\lambda^\alpha}^*$ and write \[y_{\lambda^\alpha}=(y_{\lambda^\alpha,1},\dots,y_{\lambda^\alpha,\alpha-1})\in \Hom(E_{\lambda^\alpha,2},E_{\lambda^\alpha,1})\times\cdots\times\Hom(E_{\lambda^\alpha,\alpha},E_{\lambda^\alpha,\alpha-1}).\]
We have
\[
{}^t\varepsilon(y_{\lambda^\vee},y_{\lambda^\alpha})=({}^t\varepsilon_1(y_{\lambda^\vee,1},y_{\lambda^\alpha,1}),\dots,{}^t\varepsilon_{\alpha-1}(y_{\lambda^\vee,\alpha-1},y_{\lambda^\alpha,\alpha-1})),
\]
where 
\[
{}^t\varepsilon_i(y_{\lambda^\vee,i},y_{\lambda^\alpha,i}):=\begin{pmatrix}
    y_{\lambda^\vee,i} & 0_{\dim(E_{\lambda^\vee,i+1})\times 1} \\
    0_{1\times \dim(E_{\lambda^\vee,i})} & y_{\lambda^\alpha,i}
\end{pmatrix}.
\]

Since $C_{\phi^\alpha}$ is the $0$-orbit in $V_{\lambda_{\phi^\alpha}},$  from Lemma \ref{lemma dual of phi^alpha}, we have $y\in C_{\widehat{\phi^\alpha}}\subseteq V_{\lambda^\alpha}^*$ if and only if  $y_{\phi^\alpha}=(y_{\phi^\alpha,1},\dots,y_{\phi^\alpha,\alpha-1})$, where $y_{\phi^\alpha,\alpha-i}\neq0$ for any $i=1,\dots,\alpha-1.$ For simplicity, we take $y_{\phi^\alpha,\alpha-i}=1$ for any $i=1,\dots,\alpha-1$ and so $y_{\phi^\alpha}=(1,\dots,1)$, although, this is not necessary for the rest of the argument. 

Now we fix some $x\in V_{\lambda_\alpha}$ and write $x=(x_1,\dots, x_{\alpha-1})$ as above. For each $i=1,\dots,\alpha-1$, write 
\[
x_i=\begin{pmatrix}
    x_{i,1} & x_{i,2} \\
    x_{i,3} & x_{i,4}
\end{pmatrix},
\]
where $x_{i,4}$ is a $1\times 1$ matrix (which determines the dimensions of the rest of the matrices). Let $y_{\lambda^\vee}\in V_{\lambda^\vee}^*$. To prove the lemma, we must show that
\[
[x,{}^t\varepsilon(y_{\lambda^\vee},y_{\phi^\alpha})]=0
\] if and only if $x=\varepsilon(x_{\lambda^\vee},0)$ for some $x_{\lambda^\vee}\in V_{\lambda^\vee}$  with $[x_{\lambda^\vee},y_{\lambda^\vee}]=0$

Indeed, the reverse direction follows from direct computation. The forward direction also follows from direct computation, but with a bit of tedious bookkeeping, largely in cases based on whether $x_i=x_{i,4}$ or not. We will give the details under the assumption that $x_i\neq x_{i,4}$ for any $i$ as the other cases follow from similar arguments. Thus, we assume $x_i\neq x_{i,4}$ for any $i$ and $[x,{}^t\varepsilon(y_{\lambda^\vee},y_{\phi^\alpha})]=0.$ We must show that $x_{i,4}=0$, $x_{i,2}=0$, and $x_{i,3}=0$ for each $i=1,\dots,\alpha-1$. Write $y_{\lambda^\vee}=(y_1,\dots,y_{\alpha-1})$ using the above convention.

From the assumptions, for $i=2,\dots,\alpha-1$, we obtain
\begin{align*}
    \begin{pmatrix}
        x_{1,1}y_1 & x_{1,2} \\
        x_{1,3}y_1 & x_{1,4} 
    \end{pmatrix} &= 0, \\
    \begin{pmatrix}
        x_{i,1}y_i & x_{i,2} \\
        x_{i,3}y_i & x_{i,4} 
    \end{pmatrix} &= \begin{pmatrix}
        y_{i-1} x_{i-1,1} & y_{i-1}x_{i-1,2} \\
        x_{i-1,3} & x_{i-1,4} 
    \end{pmatrix}, \\
    0 &= \begin{pmatrix}
        y_{\alpha-1} x_{\alpha-1,1} & y_{\alpha-1}x_{\alpha-1,2} \\
        x_{\alpha-1,3} & x_{\alpha-1,4} 
    \end{pmatrix}. 
\end{align*}

From the first equation above, we have that $x_{1,4}=0$, from which the second equation implies that $x_{i,4}=0$ for $i=2,\dots,\alpha-1.$ Similarly, the first equation implies that $x_{1,2}=0,$ from which the second equation implies that $x_{i,2}=0$ for $i=2,\dots,\alpha-1.$ Finally, the last equation implies that $x_{\alpha-1,3}=0.$ The  middle equation then implies that $x_{i,3}=0$ for $i=1,\dots,\alpha-2.$ Thus $x_i=\begin{pmatrix}
    x_{i,1} & 0 \\
    0 & 0
\end{pmatrix}=\varepsilon(x_{\lambda^\vee},0),$
where $x_{\lambda^\vee}=(x_{1,1},\dots,x_{\alpha-1,1})\in V_{\lambda^\vee}.$ This completes the proof of the lemma.
\end{proof}

Finally we prove our fixed point formula (Theorem \ref{thm GL fixed point formula}).

\begin{thm}\label{thm fpf appendix}
     Again, we suppose that $\beta=\frac{\alpha-1}{2} \geq m^\beta_{\mathbbm{1}_{W_F}}(\phi^\vee)$. 
     Then, we have
    \[
    \langle\eta_{\phi_\alpha},[\mathcal{F}]\rangle_{\lambda_\alpha}=\langle\eta_{\phi^\times},[\mathcal{F}|_{V_{\lambda^\times}}]\rangle_{\lambda^\times},
    \]
    for any $\mathcal{F}\in D_{H_{\lambda_\alpha}}(V_{\lambda_\alpha}).$
\end{thm}

\begin{proof}
    The proof is the same as that of \cite[Proposition 4.6]{CR26} but using the above generalizations.  From Proposition \ref{prop regular conorm}, we have that Proposition \ref{prop ranks agree} applies in our setting. Combining Propositions  \ref{prop perfect pair append} and \ref{prop ranks agree}, we obtain that
    \begin{align*}
    \langle\eta_{\phi_\alpha},[\mathcal{F}]\rangle_{\lambda_\alpha}&=(-1)^{d(C_{\phi_\alpha})}\rank(\Evs_{C_{\phi_\alpha}}\mathcal{F}) \\
    &= (-1)^{d(C_{\phi^\times})}\rank(\Evs_{C_\phi^\times}\mathcal{F}|_{V_{\lambda^\times}})  \\
    &=\langle\eta_{\phi^\times},[\mathcal{F}|_{V_{\lambda^\times}}]\rangle_{\lambda^\times}.\qedhere
    \end{align*}
\end{proof}

\bibliographystyle{amsplain}
\bibliography{ABVAdams}

\providecommand{\bysame}{\leavevmode\hbox to3em{\hrulefill}\thinspace}
\providecommand{\MR}{\relax\ifhmode\unskip\space\fi MR }
\providecommand{\MRhref}[2]{%
  \href{http://www.ams.org/mathscinet-getitem?mr=#1}{#2}
}
\providecommand{\href}[2]{#2}
\begin{thebibliography}{10}

\bibitem{Ach21}
Pramod~N. Achar, \emph{Perverse sheaves and applications to representation theory}, Mathematical Surveys and Monographs, vol. 258, American Mathematical Society, Providence, RI, 2021.

\bibitem{Ada89}
Jeffrey Adams, \emph{{$L$}-functoriality for dual pairs}, Ast\'{e}risque (1989), no.~171-172, 85--129, Orbites unipotentes et repr\'{e}sentations, II.

\bibitem{ABV92}
Jeffrey Adams, Dan Barbasch, and David~A. Vogan, Jr., \emph{The {L}anglands classification and irreducible characters for real reductive groups}, Progress in Mathematics, vol. 104, Birkh\"{a}user Boston, Inc., Boston, MA, 1992.

\bibitem{Art89}
James Arthur, \emph{Unipotent automorphic representations: conjectures}, no. 171-172, 1989, Orbites unipotentes et repr\'esentations, II, pp.~13--71.

\bibitem{Art13}
\bysame, \emph{The endoscopic classification of representations}, American Mathematical Society Colloquium Publications, vol.~61, American Mathematical Society, Providence, RI, 2013, Orthogonal and symplectic groups.

\bibitem{AG17a}
Hiraku Atobe and Wee~Teck Gan, \emph{Local theta correspondence of tempered representations and {L}anglands parameters}, Invent. Math. \textbf{210} (2017), no.~2, 341--415.

\bibitem{AGIKMS24}
Hiraku Atobe, Wee~Teck Gan, Atsushi Ichino, Tasho Kaletha, Alberto Mínguez, and Sug~Woo Shin, \emph{Local intertwining relations and co-tempered $a$-packets of classical groups}, 2024.

\bibitem{AM25}
Hiraku Atobe and Alberto M{\'i}nguez, \emph{Unitary dual of $p$-adic split $\mathrm{SO}_{2n+1}$ and $\mathrm{Sp}_{2n}$: {T}he good parity case (and slightly beyond)}, 2025, arXiv:2505.09991.

\bibitem{BH21}
Petar Baki\'{c} and Marcela Hanzer, \emph{Theta correspondence for {$p$}-adic dual pairs of type {I}}, J. Reine Angew. Math. \textbf{776} (2021), 63--117.

\bibitem{BH22}
\bysame, \emph{Theta correspondence and {A}rthur packets: on the {A}dams conjecture}, 2022, arXiv 2211.08596.

\bibitem{Bal25}
Kristaps Balodis, \emph{Proof of the $p$-adic {K}azhdan-{L}usztig hypothesis for $\mathrm{GL}_n$}, 2025, arXiv:2510.09788.

\bibitem{BZSV24}
David Ben-Zvi, Yiannis Sakellaridis, and Akshay Venkatesh, \emph{Relative {L}anglands duality}, 2024, arXiv:2409.04677.

\bibitem{Bor79}
Armand Borel, \emph{{Automorphic L-functions}}, Proc. Sympos. Pure Math., {Automorphic forms, representations, and L-functions Part 2}, vol.~33, Amer. Math. Soc., 1979, pp.~27--61.

\bibitem{Cha26}
Yuan Chai, \emph{Twisted {K}azhdan-{L}usztig conjecture for {$p$}-adic general linear group}, 2026, arXiv:2604.21731.

\bibitem{CZ21a}
Rui Chen and Jialiang Zou, \emph{Local {L}anglands correspondence for even orthogonal groups via theta lifts}, Selecta Math. (N.S.) \textbf{27} (2021), no.~5, Paper No. 88, 71.

\bibitem{CZ21b}
\bysame, \emph{Local {L}anglands correspondence for unitary groups via theta lifts}, Represent. Theory \textbf{25} (2021), 861--896.

\bibitem{CFK22}
Clifton Cunningham, Andrew Fiori, and Nicole Kitt, \emph{Appearance of the {K}ashiwara-{S}aito singularity in the representation theory of {$p$}-adic {$\rm GL(16)$}}, Pacific J. Math. \textbf{321} (2022), no.~2, 239--282.

\bibitem{CFMMX22}
Clifton Cunningham, Andrew Fiori, Ahmed Moussaoui, James Mracek, and Bin Xu, \emph{Arthur packets for {$p$}-adic groups by way of microlocal vanishing cycles of perverse sheaves, with examples}, Mem. Amer. Math. Soc. \textbf{276} (2022), no.~1353, ix+216.

\bibitem{CHLLRXS26}
Clifton Cunningham, Alexander Hazeltine, Baiying Liu, Chi-Heng Lo, Mishty Ray, Bin Xu, and J{\" o}rg Sch{\" u}rmann, \emph{{G}eometrization of {K}ottwitz-{S}helstad transfer and {V}ogan's conjecture on local {A}rthur packets}, 2026, arXiv:2611.03342.

\bibitem{CR24}
Clifton Cunningham and Mishty Ray, \emph{Proof of {V}ogan's conjecture on {A}rthur packets: irreducible parameters of {$p$}-adic general linear groups}, Manuscripta Math. \textbf{173} (2024), no.~3-4, 1073--1097.

\bibitem{CR26}
\bysame, \emph{Proof of {V}ogan's conjecture on {A}rthur packets for {$GL_n$} over {$p$}-adic fields}, Manuscripta Math. \textbf{177} (2026), no.~1, Paper No. 2.

\bibitem{GGP12}
Wee~Teck Gan, Benedict~H. Gross, and Dipendra Prasad, \emph{Symplectic local root numbers, central critical {$L$} values, and restriction problems in the representation theory of classical groups}, no. 346, 2012, Sur les conjectures de Gross et Prasad. I, pp.~1--109.

\bibitem{GGP20}
\bysame, \emph{Branching laws for classical groups: the non-tempered case}, Compos. Math. \textbf{156} (2020), no.~11, 2298--2367.

\bibitem{GI14}
Wee~Teck Gan and Atsushi Ichino, \emph{Formal degrees and local theta correspondence}, Invent. Math. \textbf{195} (2014), no.~3, 509--672.

\bibitem{GS12}
Wee~Teck Gan and Gordan Savin, \emph{Representations of metaplectic groups {I}: epsilon dichotomy and local {L}anglands correspondence}, Compos. Math. \textbf{148} (2012), no.~6, 1655--1694.

\bibitem{GS17}
Wee~Teck Gan and Binyong Sun, \emph{The {H}owe duality conjecture: quaternionic case}, Representation theory, number theory, and invariant theory, Progr. Math., vol. 323, Birkh\"{a}user/Springer, Cham, 2017, pp.~175--192.

\bibitem{GT16}
Wee~Teck Gan and Shuichiro Takeda, \emph{A proof of the {H}owe duality conjecture}, J. Amer. Math. Soc. \textbf{29} (2016), no.~2, 473--493.

\bibitem{GW25}
Wee~Teck Gan and Bryan Peng~Jun Wang, \emph{Generalised {W}hittaker models as instances of relative {L}anglands duality}, Adv. Math. \textbf{463} (2025), Paper No. 110129, 57.

\bibitem{HT01}
Michael Harris and Richard Taylor, \emph{{The geometry and cohomology of some simple Shimura varieties}}, Annals of Math. Studies, vol. 151, Princeton University Press, 2001.

\bibitem{Haz24}
Alexander Hazeltine, \emph{{The Adams conjecture and intersections of local Arthur packets}}, arXiv:2403.17867.

\bibitem{HJLLZ24}
Alexander Hazeltine, Dihua Jiang, Baiying Liu, Chi-Heng Lo, and Qing Zhang, \emph{{On {A}rthur representations and the unitary dual}}, 2024, arXiv:2410.11806.

\bibitem{HLL22}
Alexander Hazeltine, Baiying Liu, and Chi-Heng Lo, \emph{On the intersection of local {A}rthur packets for classical groups and applications}, 2024, arXiv 2201.10539v3.

\bibitem{HLL26}
\bysame, \emph{The intersection of local {A}rthur packets for even orthogonal groups and applications}, 2026, arXiv:2609.27261.

\bibitem{HLLZ25}
Alexander Hazeltine, Baiying Liu, Chi-Heng Lo, and Qing Zhang, \emph{The closure ordering conjecture on local {A}rthur packets of classical groups}, J. Reine Angew. Math. \textbf{823} (2025), 1--60.

\bibitem{HL26}
Alexander Hazeltine and Chi-Heng Lo, \emph{Algorithms on the {P}yasetskii involution on local {L}anglands parameters of classical groups}, 2026, arXiv:2604.11049.

\bibitem{Hen00}
Guy Henniart, \emph{{Une preuve simple des conjectures de Langlands pour $\mathrm{GL}(n)$ sur un corps p-adique}}, Invent. Math. \textbf{139} (2000), no.~2, 439--455.

\bibitem{How79}
Roger Howe, \emph{{$\theta $}-series and invariant theory}, Automorphic forms, representations and {$L$}-functions ({P}roc. {S}ympos. {P}ure {M}ath., {O}regon {S}tate {U}niv., {C}orvallis, {O}re., 1977), {P}art 1, Proc. Sympos. Pure Math., vol. XXXIII, Amer. Math. Soc., Providence, RI, 1979, pp.~275--285.

\bibitem{Ish23}
Hiroshi Ishimoto, \emph{The endoscopic classification of representations of non-quasi-split odd special orthogonal groups}, Int. Math. Res. Not. IMRN (2024), no.~14, 10939--11012.

\bibitem{Kal13}
Tasho Kaletha, \emph{Genericity and contragredience in the local {L}anglands correspondence}, Algebra Number Theory \textbf{7} (2013), no.~10, 2447--2474.

\bibitem{KMSW14}
Tasho Kaletha, Alberto Minguez, Sug~Woo Shin, and Paul-James White, \emph{Endoscopic classification of representations: Inner forms of unitary groups}, 2014, Preprint.

\bibitem{Lan75}
Robert Langlands, \emph{{Letter from {R}. {L}anglands to {R}. {H}owe}}, 1975, Unpublished.

\bibitem{LM16}
Erez Lapid and Alberto M\'inguez, \emph{On parabolic induction on inner forms of the general linear group over a non-archimedean local field}, Selecta Math. (N.S.) \textbf{22} (2016), no.~4, 2347--2400.

\bibitem{Li89}
Jian-Shu Li, \emph{Singular unitary representations of classical groups}, Invent. Math. \textbf{97} (1989), no.~2, 237--255.

\bibitem{Li24}
Wen-Wei Li, \emph{{Arthur packets for metaplectic groups}}, 2024, arXiv:2410.13606.

\bibitem{Lo24}
Chi-Heng Lo, \emph{Vogan's conjecture on local {A}rthur packets of {$p$}-adic {${\rm GL}_n$} and a combinatorial lemma}, Pacific J. Math. \textbf{333} (2024), no.~2, 331--356.

\bibitem{Min08}
Alberto M\'{\i}nguez, \emph{Correspondance de {H}owe explicite: paires duales de type {II}}, Ann. Sci. \'{E}c. Norm. Sup\'{e}r. (4) \textbf{41} (2008), no.~5, 717--741.

\bibitem{Moe06}
Colette M{\oe}glin, \emph{Sur certains paquets d'{A}rthur et involution d'{A}ubert-{S}chneider-{S}tuhler g\'{e}n\'{e}ralis\'{e}e}, Represent. Theory \textbf{10} (2006), 86--129.

\bibitem{Moe09}
\bysame, \emph{Paquets d'{A}rthur discrets pour un groupe classique {$p$}-adique}, Automorphic forms and {$L$}-functions {II}. {L}ocal aspects, Contemp. Math., vol. 489, Amer. Math. Soc., Providence, RI, 2009, pp.~179--257.

\bibitem{Moe11c}
\bysame, \emph{Conjecture d'{A}dams pour la correspondance de {H}owe et filtration de {K}udla}, Arithmetic geometry and automorphic forms, Adv. Lect. Math. (ALM), vol.~19, Int. Press, Somerville, MA, 2011, pp.~445--503.

\bibitem{Moe11a}
\bysame, \emph{Multiplicit\'{e} 1 dans les paquets d'{A}rthur aux places {$p$}-adiques}, On certain {$L$}-functions, Clay Math. Proc., vol.~13, Amer. Math. Soc., Providence, RI, 2011, pp.~333--374.

\bibitem{MR18}
Colette M{\oe}glin and David Renard, \emph{Sur les paquets d'{A}rthur des groupes classiques et unitaires non quasi-d\'eploy\'es}, Relative aspects in representation theory, {L}anglands functoriality and automorphic forms, Lecture Notes in Math., vol. 2221, Springer, Cham, 2018, pp.~341--361.

\bibitem{MW86}
Colette M{\oe}glin and Jean-Loup Waldspurger, \emph{Sur l'involution de {Z}elevinski}, J. Reine Angew. Math. \textbf{372} (1986), 136--177.

\bibitem{Mok15}
Chung~Pang Mok, \emph{Endoscopic classification of representations of quasi-split unitary groups}, Mem. Amer. Math. Soc. \textbf{235} (2015), no.~1108, vi+248.

\bibitem{Rid23}
Connor Riddlesden, \emph{{C}ombinatorial {A}pproach to {ABV}-packets for $\mathbf{GL}_n$}, 2023, arXiv:2304.09598.

\bibitem{Sak17}
Yiannis Sakellaridis, \emph{Plancherel decomposition of {H}owe duality and {E}uler factorization of automorphic functionals}, Representation theory, number theory, and invariant theory, Progr. Math., vol. 323, Birkh\"auser/Springer, Cham, 2017, pp.~545--585.

\bibitem{Sch13}
Peter Scholze, \emph{The local {L}anglands correspondence for $\mathrm{GL}_n$ over p-adic fields}, Invent. Math. \textbf{192} (2013), no.~3, 663--715.

\bibitem{Sol25}
Maarten Solleveld, \emph{Graded {H}ecke algebras, constructible sheaves and the {$p$}-adic {K}azhdan-{L}usztig conjecture}, J. Algebra \textbf{667} (2025), 865--910.

\bibitem{SZ15}
Binyong Sun and Chen-Bo Zhu, \emph{Conservation relations for local theta correspondence}, J. Amer. Math. Soc. \textbf{28} (2015), no.~4, 939--983.

\bibitem{Vog93}
David~A. Vogan, Jr., \emph{The local {L}anglands conjecture}, Representation theory of groups and algebras, Contemp. Math., vol. 145, Amer. Math. Soc., Providence, RI, 1993, pp.~305--379.

\bibitem{Wal90}
Jean-Loup Waldspurger, \emph{D\'{e}monstration d'une conjecture de dualit\'{e} de {H}owe dans le cas {$p$}-adique, {$p\neq 2$}}, Festschrift in honor of {I}. {I}. {P}iatetski-{S}hapiro on the occasion of his sixtieth birthday, {P}art {I} ({R}amat {A}viv, 1989), Israel Math. Conf. Proc., vol.~2, Weizmann, Jerusalem, 1990, pp.~267--324.

\bibitem{Zel80}
A.~V. Zelevinsky, \emph{Induced representations of reductive {$\mathfrak{p}$}-adic groups. {II}. {O}n irreducible representations of {${\rm GL}(n)$}}, Ann. Sci. \'Ecole Norm. Sup. (4) \textbf{13} (1980), no.~2, 165--210.

\end{thebibliography}

\end{document}